\documentclass[11pt, a4paper, twoside]{article}
\usepackage{amsfonts}
\usepackage{mathrsfs}
\usepackage{amsmath, amsthm}
\usepackage{pgf,tikz}
\usepackage{amssymb}
\usepackage{fancyhdr}
\usepackage[colorlinks, linkcolor=black, anchorcolor=black, citecolor=black]{hyperref}

\numberwithin{equation}{section}

\allowdisplaybreaks

\begin{document}

\fancyhf{}

\fancyhead[EC]{W. LI}

\fancyhead[EL]{\thepage}

\fancyhead[OC]{Maximal functions associated with nonisotropic dilations of hypersufaces in  $\mathbb{R}^3$}

\fancyhead[OR]{\thepage}

\renewcommand{\headrulewidth}{0pt}
\renewcommand{\thefootnote}{\fnsymbol {footnote}}

\title{\textbf{Maximal functions associated with nonisotropic dilations of hypersufaces in  $\mathbb{R}^3$}
\footnotetext {{}{2000 \emph{Mathematics Subject
 Classification}: 42B20, 42B25}}
\footnotetext {{}\emph{Key words and phrases}:Maximal function, nonisotropic dilations, hypersurfaces, cinematic curvature condition, local smoothing estimate.}} \setcounter{footnote}{0}
\author{
Wenjuan Li}

\date{}
\maketitle

\begin{abstract}
The goal of this article is to establish $L^p$-estimates for maximal functions associated with nonisotropic dilations of hypersurfaces in $\mathbb{R}^3$.
Several results have already been obtained by Greenleaf, Iosevich-Sawyer, Iosevich-Sawyer-Seeger, Ikromov-Kempe-M\"{u}ller and Zimmermann, but for some situations such as the hypersurface parameterized as the graph of a smooth function $\Phi(x_1,x_2)=x_2^d(1+\mathcal{O}(x_2^m))$ near the origin, where $d\geq 2$, $m\geq 1$, and associated dilations $\delta_t(x)=(t^ax_1,tx_2,t^dx_3)$ for an arbitrary real number $a>0$, the question was open until recently. In fact, such problems do arise already in lower dimensions. For instance, we consider the curve $\gamma(x)=(x,x^2(1+\phi(x)))$ and associated dilations $\delta_t(x)=(tx_1,t^2x_2)$. If $\phi\equiv 0$, then the corresponding maximal function is the maximal function along parabolas in the plane, which plays an important role in the study of singular Randon transforms, and which is very well understood due to the work by Nagel-Riviere-Wainger and others. If $\phi\neq 0$ and $\phi(x)=\mathcal{O}(x^m)$, $m\geq 1$, the problem was open until recently, however, the corresponding maximal function shows features related to the Bourgain circular maximal function, which required deep ideas and local smoothing estimates established by Mockenhaupt-Seeger-Sogge  for Fourier integral operators satisfying the so-called "cinematic curvature" condition. However, we observe that in the study of $\mathcal{M}$ related to the mentioned curve $\gamma(x)$ and associated dilations, we will consider a family of corresponding Fourier integral operators which fail to satisfy the "cinematic curvature condition" uniformly, which means that classical local smoothing estimates could not be directly applied to our problem. In this article, we develop new ideas  in order to overcome the above difficulty and finally establish sharp $L^p$-estimates for the maximal function related to the curve $\gamma(x)$ with associated dilations in the plane. Later, we generalize the result to curves of finite type $d$ ($d\geq 2$) and associated dilations $\delta_t(x)=(tx_1,t^dx_2)$. Furthermore, we also obtain $L^p$-estimates for the maximal function related to the mentioned hypersurface $\Phi(x_1,x_2)$ in $\mathbb{R}^3$  with associated dilations. Moreover,  by an alternative approach, we also get  $L^p$-estimates for some particular classes of maximal functions in $\mathbb{R}^3$ established earlier by Greenleaf, Iosevich-Sawyer, Ikromov-Kempe-M\"{u}ller and Zimmermann.
\end{abstract}

\tableofcontents

\theoremstyle{plain} 
\newtheorem{thm}{\indent\sc Theorem}[section] 
\newtheorem{lem}[thm]{\indent\sc Lemma}
\newtheorem{cor}[thm]{\indent\sc Corollary}
\newtheorem{prop}[thm]{\indent\sc Proposition}
\newtheorem{claim}[thm]{\indent\sc Claim}
\theoremstyle{definition} 
\newtheorem{dfn}[thm]{\indent\sc Definition}
\newtheorem{rem}[thm]{\indent\sc Remark}
\newtheorem{ex}[thm]{\indent\sc Example}
\newtheorem{notation}[thm]{\indent\sc Notation}
\newtheorem{assertion}[thm]{\indent\sc Assertion}
%
%
\numberwithin{equation}{section}
\renewcommand{\proofname}{\indent\sc Proof.} 
\def\C{\mathbb{C}}
\def\R{\mathbb{R}}
\def\Rn{{\mathbb{R}^n}}
\def\M{\mathbb{M}}
\def\N{\mathbb{N}}
\def\Q{{\mathbb{Q}}}
\def\Z{\mathbb{Z}}
\def\F{\mathcal{F}}
\def\L{\mathcal{L}}
\def\S{\mathcal{S}}
\def\supp{\operatorname{supp}}
\def\essi{\operatornamewithlimits{ess\,inf}}
\def\esss{\operatornamewithlimits{ess\,sup}}
\def\dlim{\displaystyle\lim}

\section[Introduction]{Introduction.}

The starting point for intensive studies associated with averages over low dimensional sets is based on an
earlier bound of Stein \cite{stein3} from 1976 on the spherical maximal function
\begin{equation*}
\mathcal{M}f(x):=\sup_{t>0}|M_tf(x)|,
\end{equation*}
where $M_t$ are the spherical averaging operators
\begin{equation*}
\mathcal{M}_tf(x)=\int_{|y|=1}f(x-ty)d\sigma(y),
\end{equation*}
and $d\sigma$ is normalized surface measure on the sphere $S^{n-1}$. Then Stein's fundamental result shows that for $n\geq 3$,  the corresponding spherical maximal operator is bounded on $L^p(\mathbb{R}^n)$ for every $p>n/(n-1)$.  The
analogous result in dimension two was later proved by J. Bougain
\cite{JB}.

Then one turned to deal with generalizations of $\mathcal{M}_t$ defined as before, i.e. the sphere is replaced by a more general smooth hypersurface $S \in {\mathbb{R}}^n$. Let $\rho\in C_0^{\infty}(S)$ be a smooth non-negative function with compact support. Then the associated maximal operator is defined as
\begin{equation}\label{maximalfunsur}
\mathcal{M}f(x):=\sup_{t>0}\left|\int_{S}f(x-ty)\rho(y)d\mu(y)\right|, \hspace{0.2cm} x\in \mathbb{R}^n,
\end{equation}
where $d\mu$ denotes the surface measure on $S$. Greenleaf \cite{Greenleaf} proved that $\mathcal{M}$ is bounded on $L^p(\mathbb{R}^n)$ if $n\geq 3$ and $p>(k+1)/k$, provided $S$ has at least $k\geq 2$ non-vanishing principal curvatures and $S$ is starshaped with respect to the origin.

A fundamental and still largely open problem is to characterize the $L^p$ boundedness properties of the maximal operator associated to hypersurface where the Gaussian curvature at some points is allowed to vanish. Completely understood is is only the $2$-dimensional case, i.e. the case of finite type curves in $\mathbb{R}^2$ studied by A. Iosevich in \cite{I}.

Many authors put a lot of effort on the development of this subject and obtained partial results in high dimension. Sogge and Stein \cite{ss} showed that if the Gaussian curvature of S does not vanish of infinite order at any point of S, then there exists a $p_0(S)<\infty$ so that the maximal function is bounded on $L^p$, $p>p_0(S)$.
However, the exponent $p_0(S)$ given in that paper is in general far from being optimal. In order to find the smallest value of $p_0(S)$, one will put more restriction on the surfaces. It is worth mentioning that in 1992, Sogge \cite{Sogge} employed the local smoothing estimate to get the $L^p$-boundedness $p>2$ for $\mathcal{M}$,
 where the surface has at least one non-vanishing principal curvature everywhere in $\mathbb{R}^n$ $(n\geq 2)$. The perhaps best understood class in higher dimension is the class of convex hypersurface of finite line type, and the related work in this settings included Cowling and Mauceri \cite{cw1} \cite{cw2}, Nagel, Seeger and Wainger\cite{nsw},  Iosevich
 and Sawyer \cite{ios1} \cite{ios2}, and Iosevich, Sawyer and Seeger \cite{Iss}. In particular, Iosevich and Sawyer proved in \cite{ios2} sharp $L^p$-estimates for  maximal functions related to the surface which are given by  smooth convex functions of finite line type for $p>2$. In 2010, Ikoromov, Kempe and M\"{u}ller \cite{IKMU}
 discovered a connection between the $L^p$-boundedness ($p>max\{h(S),2\}$) of $\mathcal{M}$ and the height $h(S)$ of a smooth, compact hypersurface $S$ of finite type in $\mathbb{R}^3$ satisfying the transversality assumption (i.e. for every $x\in S$, the affine tangent plane $x+T_xS$ to $S$ through $x$ does not pass through the origin in $\mathbb{R}^3$) on $S$.
  Recently, Zimmermann in his thesis \cite{Zimmermann} proved that  maximal averages over analytic hypersurfaces passing through the origin in general behave more regularly than the maximal averages over hypersurfaces satisfying
the transversality condition.

Maximal operators defined by averages over curves or surfaces with nonisotropic dilations have also been extensively considered, i.e. the dilation $ty$ appearing in (\ref{maximalfunsur}) is replaced by  $\delta_t(y)=(t^{a_1}y_1, t^{a_2}y_2,\cdots, t^{a_n}y_n)$, where $a_j>0$, $j=1,2,\cdots,n$. In 1970, in the study of a problem related
to Poisson integrals for symmetric spaces, Stein raised the question as to when the operator $\mathcal{M}_{\gamma}$ defined by
$$\mathcal{M}_{\gamma}f(x)=\sup_{h>0}\frac{1}{h}\int_{0}^h|f(x-\gamma(t))|dt,$$
where $\gamma(t)=(A_1t^{a_1},A_2t^{a_2},\cdots,A_nt^{a_n})$ and $A_1,A_2,\cdots,A_n$ are real, $a_i>0$, is bounded on $L^p(\mathbb{R}^n)$.
Nagel, Riviere and Wainger \cite {NRW} showed that the $L^p$-boundedness of $\mathcal{M}$ holds for  $p>1$ for the special case $\gamma(t)=(t,t^2)$ in $\mathbb{R}^2$ and Stein \cite{stein} for  homogeneous curves in $\mathbb{R}^n$. More general maximal operators involving two-parameter dilations related to homogeneous curves  have studied by Marletta and Ricci \cite{MRicci}. For maximal functions $\mathcal{M}$ associated with nonisotropic dilations in higher dimensions, one can see the work by Greenleaf \cite{Greenleaf}, Sogge and Stein \cite{ss}, Iosevich and Sawyer \cite{ios2}, Ikoromov, Kempe and M\"{u}ller \cite{IKMU}, Zimmermann \cite{Zimmermann}.

\textbf{Acknowledgement} The author would like to thank her supervisor Professor D. M\"{u}ller for his constant support and encouragement. She has learned a lot of mathematics and she has immensely profited from her mathematical discussions  with Professor D. M\"{u}ller.

\subsection[Outline of the problem and statement of main theorems in the plane]{Outline of the problem and statement of main theorems in the plane.}

\noindent The maximal operator along curves in $\mathbb{R}^2$ is defined by
\begin{equation*}
\mathcal{M}f(x)=\sup_{t>0}|f*\mu_t(x)|,
\end{equation*}
where $\mu$ is a arc length measure supported on the curve and $\mu_t$ is the same measure dilated by $t>0$ and appropriately normalized.

Two different situations will arise.  (1) A maximal operator of the first type: the curve $x_2=\gamma(x_1)$ is homogeneous under the given dilations
(such as $\gamma(x_1)=x_1^d, d>0,d\neq 1$).  Basically one obtains the same operator by restricting the supremum to $t=2^j$. Under appropriate assumptions
on the curve, one can prove that $\mathcal{M}$ is bounded on $L^p$ for $p>1$. If $d=2$, then the corresponding maximal function is the maximal function along parabolas in the plane, which plays an important role in the study of singular Randon transforms, and which is very well understood due to the work by Nagel-Riviere-Wainger \cite{NRW} and \cite{stein}.
(2) A maximal operator of the second type: the curve $x_2=\gamma(x_1)$ is not homogeneous under the given dilations. Then the various $\mu_t$ are supported on
 different curves and the problem becomes much more complicated. We have the following subcases:

\noindent(2a) if the dilations are isotropic and the Gaussian curvature of
 the curve does not vanish, then $\mathcal{M}$ is bounded on $L^p$ for $p>2$, see \cite{JB} and \cite{mss2}. The range of $p$ must be further restricted if
 the curvature allows to vanish at some point \cite{I} unless one introduces a damping factor \cite{Marletta};

 \noindent(2b) if the dilations are non-isotropic and the
 curve with non-vanishing Gaussian curvature away from the origin is homogeneous,  Marletta and Ricci \cite{MRicci} combine Bourgain's theorem with a Littlewood-Paley
 decomposition to show that $\mathcal{M}$ is bounded on $L^p$ if and only if $p>2$;

A natural question to ask is whether there exists some positive constant $p_0>1$ such that $\mathcal{M}$ is bounded on $L^p$ for $p>p_0$ in the following setting:

 \noindent(2c) if the dilations
are non-isotropic and the curve $x_2=\gamma(x_1)$ is non-homogeneous and of finite type, such as $\gamma(x_1)=x_1^d(1+\mathcal{O}(x_1^m))$, where $d\geq 2$ and $m\geq 1$? The problem was open until recently, however, the corresponding maximal function shows features related to the Bourgain circular maximal function, which required deep ideas and $L^4$-techniques. An alternative approach was later given by Mockenhaupt, Seeger and Sogge,  who established local smoothing estimates for Fourier integral operators satisfying the so-called "cinematic curvature" condition. However, we observe that in the study of $\mathcal{M}$ related to the mentioned curve $\gamma(x)$ and associated dilations, in our situation, we will consider a family of corresponding Fourier integral operators which fail to satisfy the "cinematic curvature condition" uniformly, which means that classical local smoothing estimates could not be directly applied to our problem. In this article, we develop new ideas  in order to overcome the above difficulty and finally establish sharp $L^p$-estimates for the maximal function related to the curve $\gamma(x)$ with associated dilations in the plane.

Next, we list our main results in the plane.

Let $\phi \in C^{\infty}(I,\mathbb{R})$, where $I$ is a bounded interval containing the origin, and
\begin{equation}\label{planeimporcondi}
\phi(0)\neq 0; \hspace{0.1cm}\phi^{(j)}(0)= 0, \hspace{0.1cm}j=1,2,\cdots,m-1; \hspace{0.1cm}\phi^{(m)}(0)\neq 0 \hspace{0.1cm}(m\geq 1).
 \end{equation}

\begin{thm}\label{planetheorem}
Define the maximal operator
\begin{equation}
\mathcal{M}f(y):=\sup_{t>0}\left|\int_{\mathbb{R}}f(y_1-tx,y_2-t^2x^2\phi(x))\eta(x)dx\right|,
\end{equation}
where $\eta(x)$ is supported in a sufficiently small neighborhood of the origin. If $\phi$ satisfies (\ref{planeimporcondi}), then for $p>2$,
there exists a constant
$C_p$ such that the following inequality holds true:
\begin{equation}\label{equ:planem=1}
\|\mathcal{M}f\|_{L^{p}}\leq  C_p \|f\|_{L^{p}}, \hspace{0.5cm}f\in C_0^{\infty}(\mathbb{R}^2).
\end{equation}
\end{thm}

\begin{rem} If $m\rightarrow \infty$,  then we obtain the maximal operator of the first type described above which is bounded on $L^p$ for all $p>1$.
\end{rem}

\begin{rem} The above theorem is sharp if $1\leq m<\infty$, which follow from the proposition  given in Appendix of this article.
\end{rem}

Next we will generalize Theorem \ref{planetheorem} to the curve of finite type $d$ $(\geq 2)$ associated with the function $\phi$ defined by (\ref{planeimporcondi}).

\begin{thm}\label{planetheoremfinite}
Define the maximal operator
\begin{equation}
\mathcal{M}f(y):=\sup_{t>0}\left|\int_{\mathbb{R}}f(y_1-tx,y_2-t^dx^d\phi(x))\eta(x)dx\right|,
\end{equation}
where $\eta(x)$ is supported in a sufficiently small neighborhood of the origin. If $\phi$ satisfies (\ref{planeimporcondi}) and $d\geq 2$, then
for $p>2$, there exists a  constant
$C_p$ such that
\begin{equation}\label{equ:planetheoremfinite}
\|\mathcal{M}f\|_{L^{p}}\leq  C_p\|f\|_{L^{p}}, \hspace{0.5cm}f\in C_0^{\infty}(\mathbb{R}^2).
\end{equation}
\end{thm}

\begin{rem}
If $m\rightarrow \infty$, then $\mathcal{M}$ coincides with the maximal operator of the second type described above.
\end{rem}

\subsection[Outline of the problem and statement of main theorems in $\mathbb{R}^3$]{Outline of the problem and statement of main theorems in $\mathbb{R}^3$.}

Let $S$ be a smooth hypersurface in $\mathbb{R}^n$ with a smooth measure $d\mu$ compactly supported away from the boundary. Given an $n$-tuple $(a_1,a_2,\cdots,a_n)$ of nonnegative real numbers, consider the maximal operator $\mathcal{M}$ defined by
\begin{equation*}
\mathcal{M}f(x)=\sup_{t>0}\left|\int_{S}f(x-\delta_t(y))d\mu(y)\right|,
\end{equation*}
where $\delta_t$ denotes the nonisotropic dilation given by $\delta_t(x)=(t^{a_1}x_1,t^{a_2}x_2\cdots,t^{a_n}x_n)$. Iosevich and Sawyer \cite{ios2} observed that $\mathcal{M}$ often behaves much better than the maximal function with standard dilations due to a "rotational curvature" in the time parameter $t$.  Let $\mathcal{H}$ be a hyperplane and $d(x,\mathcal{H})$ denote the distance from $x$ on $S$ to $\mathcal{H}$. They showed that in the case $a_1=a_2=\cdots=a_{n-1}\neq a_n$, the condition $d(x,\mathcal{H})^{-1}\in L_{loc}^{1/p}(S)$
over all horizontal hyperplanes $\mathcal{H}$ is sufficient for the $L^p$-boundedness of $\mathcal{M}$ when the surface $S$ is given as the graph of a mixed homogeneous function with finite-type level sets. Their point of departure in dealing with $\mathcal{M}$  is the extension of the square function theorem of Sogge and Stein. In their paper, to get the $L^2$-estimate of $\mathcal{M}$ one would mind that $J(\xi):=(\int_{t\sim 1}|\nabla^{\alpha}\widehat{\mu_t}(\xi)|^2dt)^{1/2}\leq C(1+|\xi|)^{-1/2}\gamma(|\xi|)$, where $\alpha=0, 1$, $\gamma$ is bounded and nonincreasing on $[0,\infty)$, and $\sum_{n=0}^{\infty}\gamma(2^n)<\infty$, and $\widehat{\mu_t}(\xi)=\hat{\mu}(\delta_t\xi)$. However, the above method cannot work in our problems. In our paper, to apply the Sobolev embedding for control of the maximal operator $\mathcal{M}$, for example, we need to estimate the absolute value of $\widehat{\mu_t}(\xi)=\int e^{-i\langle\delta_t\xi,(x_1,x_2, x_2^2+x_2^m)\rangle}\eta_2(x_2)dx_2$ if $\xi_1\sim \xi_2$, where $\eta_2\in C_0^{\infty}(\mathbb{R})$ and $\eta_2$ is supported on an interval not containing the origin. In fact, by the standard method of stationary phase we get that $|\widehat{\mu_t}(\xi)|\sim |\xi|^{-1/2}$, but $J(\xi)\sim |\xi|^{-1/2}$.  In 2010, Ikoromov, Kempe and M\"{u}ller \cite{IKMU} discovered a connection between the $L^p$-boundedness of $\mathcal{M}$ and the height of a smooth, compact hypersurface of finite type in $\mathbb{R}^3$ satisfying a transversality assumption on $S$. Recently, Zimmermann \cite{Zimmermann} proved that the maximal averages over analytic hypersurfaces located at the origin in $\mathbb{R}^3$ generally behave more regularly than the maximal averages over hypersurfaces satisfying
the transversality condition. It is worth to mention that in \cite{Zimmermann}, he could not handle some situations such as the hypersurface parameterized as the graph of a smooth function $\Phi(x_1,x_2)=x_2^d(1+\mathcal{O}(x_2^m))$ near the origin, where $d\geq 2$, $m\geq 1$, and associated dilations $\delta_t(x)=(t^ax_1,tx_2,t^dx_3)$ for an arbitrary real number $a>0$, the question was open until recently.

In this article, we establish $L^p$-estimates for maximal functions related to hypersurfaces $(x_1,x_2)\rightarrow (x_1,x_2, x_2^d(1+\mathcal{O}(x_2^m)))$ with associated dilations $\delta_t(x)=(t^ax_1,tx_2,t^dx_3)$ for an arbitrary real number $a>0$. These results are stated in Theorem \ref{theocurvanishiin} and  Theorem \ref{vanish3}, and could not be covered by the theorems respectively developed in the above references.  Moreover,  by an alternative approach, we also get  $L^p$-estimates for a large class of maximal functions in $\mathbb{R}^3$ proved in \cite{Greenleaf}, \cite{ios2}, \cite{IKMU} and \cite{Zimmermann}, see Theorem \ref{3nonvanish}, Theorem \ref{3vanish2} and Theorem \ref{corollary}.

Let $\Omega$ be an open neighborhood of the origin. Suppose $\Gamma$ is a hypersurface in $\mathbb{R}^3$ which is parametrized as the graph
of a smooth function $\Phi: \Omega \rightarrow \mathbb{R}$ at the origin, i.e. $\Gamma=\{(x,\Phi(x)), x\in \Omega\subset \mathbb{R}^2\}$.

In $\mathbb{R}^3$,  our proofs always follow the idea. First we "freeze" the first variable $x_1$ and apply the method of stationary phase to curves in $(x_2, x_3)-$ plane, then by Sobolev-embedding, we can reduce to apply $L^p$-estimates for certain Fourier integral operators which appeared in Theorem \ref{lem:lemma4}, Theorem \ref{L^4therom}, Theorem \ref{F_lambda} and Theorem \ref{123}.

First, we show $L^p$ estimates for  maximal functions related to hypersurfaces with at least one non-vanishing principal curvature.

\begin{thm}\label{3nonvanish}
Assume that $\Phi(x_1,x_2)\in C^{\infty}(\Omega)$  satisfies
\begin{equation}\label{conditionnonvani}
\partial_2\Phi(0,0)=0,\hspace{0.5cm}\partial_2^2\Phi(0,0)\neq 0,
\end{equation}
and  $2a_2\neq a_3$.
Then there exists a sufficiently small neighborhood of the origin $U\subset \Omega$ such that
for every positive smooth function $\eta\in C_0^{\infty}(U)$, the associated maximal function
\begin{equation}
\mathcal{M}f(y):=\sup_{t>0}\left|\int_{\mathbb{R}^2}f(y-\delta_t(x_1,x_2,\Phi(x_1,x_2)))\eta(x)dx\right|,
\end{equation}
initially defined on $C_0^{\infty}({\mathbb{R}}^3)$, is bounded on $L^p({\mathbb{R}}^3)$ for $p>2$.
\end{thm}

\begin{thm}\label{theocurvanishi}Let $\phi\in C^{\infty}(I)$, where $I$ is a bounded interval containing the origin.
Define the maximal function by
\begin{equation}
\mathcal{M}f(y):=\sup_{t>0}\left|\int_{\mathbb{R}^2}f(y-\delta_t(x_1,x_2,x_2^{2}\phi(x_2)))\eta(x)dx\right|,
\end{equation}
where $\eta$ is supported in a sufficiently small neighborhood $U$ of the origin.
Assume that $\phi$ satisfies (\ref{planeimporcondi}), i.e.
\begin{equation*}
\phi(0)\neq 0; \hspace{0.1cm}\phi^{(j)}(0)= 0, \hspace{0.1cm}j=1,2,\cdots,m-1; \hspace{0.1cm}\phi^{(m)}(0)\neq 0 \hspace{0.1cm}(m\geq 1),
\end{equation*}
and $2a_2=a_3$. Then for  $p>2$, there exists a constant
$C_p$ such that  the maximal operator
satisfies the following  estimate:
\begin{equation*}\label{theocurvanishiin}
\|\mathcal{M}f\|_{L^p}\leq C_p\|f\|_{L^p}, \hspace{0.5cm} f\in
C_0^{\infty}({\mathbb{R}}^3).
\end{equation*}
\end{thm}


Then  we extend the above theorems to hypersurfaces of finite type.

\begin{thm}\label{3vanish2}
 Assume that  $\Phi(x_1,x_2)\in C^{\infty}(\Omega)$  satisfies $\Phi(0,0)\neq 0$
and  $da_2\neq a_3$, $d\geq 2$.
Then there exists a sufficiently small neighborhood of the origin $U\subset \Omega$ such that
for every positive smooth function $\eta\in C_0^{\infty}(U)$, the associated maximal function
\begin{equation}
\mathcal{M}f(y):=\sup_{t>0}\left|\int_{\mathbb{R}^2}f(y-\delta_t(x_1,x_2,x_2^d\Phi(x_1,x_2)))\eta(x)dx\right|,
\end{equation}
initially defined on $C_0^{\infty}({\mathbb{R}}^3)$, is bounded on $L^p({\mathbb{R}}^3)$ for $p>2$.
\end{thm}

The case when $da_2=a_3$ in Lemma \ref{theocurvanishi} turns out to be much harder. For this case, we have the following result.

\begin{thm}\label{vanish3}Let $\phi\in C^{\infty}(I)$, where $I$ is a bounded interval containing the origin.
Define the maximal function by
\begin{equation}
\mathcal{M}f(y):=\sup_{t>0}\left|\int_{\mathbb{R}^2}f(y-\delta_t(x_1,x_2,x_2^{d}\phi(x_2)))\eta(x)dx\right|,
\end{equation}
where $\eta$ is supported in a sufficiently small neighborhood $U$ of the origin.
Assume that $\phi$ satisfies (\ref{planeimporcondi}), $d\geq 2$ and $da_2=a_3$. Then for  $p>2$, there exists a constant
$C_p$ such that  the maximal operator
satisfies the following  estimate:
\begin{equation*}\label{equ:vanish3}
\|\mathcal{M}f\|_{L^p}\leq C_p\|f\|_{L^p}, \hspace{0.5cm} f\in
C_0^{\infty}({\mathbb{R}}^3).
\end{equation*}
\end{thm}


Finally, we notice that the surfaces in Lemma \ref{3vanish2} are required to go through the origin. Now we consider the other case.


\begin{thm}\label{corollary}
Assume that $\Phi(x_1,x_2)\in C^{\infty}(\Omega)$ satisfies
$\Phi(0,0)\neq 0$ and $da_2\neq a_3$, $d\geq 2$.
Then there exists a sufficiently small neighborhood of the origin $U\subset \Omega$ such that
for every positive smooth function $\eta\in C_0^{\infty}(U)$, the associated maximal function
\begin{equation}
\mathcal{M}f(y):=\sup_{t>0}\left|\int_{\mathbb{R}^2}f(y-\delta_t(x_1,x_2,1+x_2^d\Phi(x_1,x_2)))\eta(x)dx\right|,
\end{equation}
initially defined on $C_0^{\infty}({\mathbb{R}}^3)$, is bounded on $L^p({\mathbb{R}}^3)$ for $p>d$.
\end{thm}

\begin{thm}\label{vanishnotorigin}
Let $\phi\in C^{\infty}(I)$, where $I$ is a bounded interval containing the origin.
Define the maximal function by
\begin{equation}
\mathcal{M}f(y):=\sup_{t>0}\left|\int_{\mathbb{R}^2}f(y-\delta_t(x_1,x_2,1+x_2^{d}\phi(x_2)))\eta(x)dx\right|,
\end{equation}
where $\eta$ is supported in a sufficiently small neighborhood $U$ of the origin.
Assume that $\phi$ satisfies (\ref{planeimporcondi}), $d\geq 2$ with $da_2=a_3$ and $1\leq m<\infty$. Then for  $p>d$, there exists a constant
$C_p$ such that  the maximal operator
satisfies the following  estimate:
\begin{equation*}\label{equ:vanish3}
\|\mathcal{M}f\|_{L^p}\leq C_pm\|f\|_{L^p}, \hspace{0.5cm} f\in
C_0^{\infty}({\mathbb{R}}^3).
\end{equation*}
\end{thm}

\subsection[Organization of the article]{Organization of the article}

In our settings, the averaging operator is always written as a Fourier integral operator, then we turn to prove $L^p$-boundedness for the Fourier integral operator.
In Section 2, first, we give an overview of the theory of Fourier integral operators and a concrete description of the so-called cinematic curvature condition. Furthermore,
we establish $L^p$-boundedness of the maximal operator related to the curve $x\rightarrow (x, x^2(1+\phi(x)))$, where $\phi(x)=\mathcal{O}(x^m)$ and $m\geq 1$,
with associated dilations $(x_1,x_2)\rightarrow(tx_1,t^2x_2)$. After a scaling argument, application of the method of stationary phase and Littlewood-Paley theory, we observe that the phase
 function of the corresponding Fourier integral operator will not satisfy the cinematic curvature condition uniformly, which causes major difficulties to apply the local
 smoothing estimate directly to our Fourier integral operator.  We separate the problem into two parts, depending on a dyadic decomposition for the frequency variables of
 the corresponding Fourier integral operator, i.e. low frequency and high frequency. For low frequency, we mainly use the better endpoint regularity estimates,
 the M. Riesz interpolation theorem to get the desired result. For high frequency, comparing the light cone $\{(\xi,|\xi|)\}$, the associated cone of the  Fourier integral operator $\tilde{\mathcal{F}}_{\lambda}^{\delta}$ which does not satisfy the cinematic curvature condition uniformly becomes flatter. A very natural question is to ask how the $L^4$-estimate for a Fourier integral operator depends on the cinematic curvature. The simplest model is to study how the key $L^4$-boundedness of the Fourier integral operator $\mathcal{F}_{\lambda}^{\delta}$ which can approximate the Fourier integral operator $\tilde{\mathcal{F}}_{\lambda}^{\delta}$. In order to clarify the above questions, a direct idea is whether we can develop a new method based on the main idea of the local smoothing estimate
 from \cite{mss2} which is mainly to get the key $L^4$-boundedness of the Fourier integral operator associated to the light cone. However, we need various modifications to overcome a lot of difficulties for $L^4$-boundedness of the Fourier integral operator $\mathcal{F}_{\lambda}^{\delta}$, the details will appear in Section 3. Based on some ideas from \cite{mss}, we obtain the $L^4$-estimate of the Fourier integral operator $\tilde{\mathcal{F}}_{\lambda}^{\delta}$ in Section 4. In fact, we observe that $L^4$-estimate of the Fourier integral operator $\mathcal{F}_{\lambda}^{\delta}$ remains valid under small, sufficiently smooth perturbation, and the constant $C_p$ depends only on a finite number of derivatives of the phase function and the symbol. In the last part of Section 2, we generalize the result to curves of finite type $d$ ($d\geq 2$) and associated dilations $(x_1,x_2)\rightarrow(tx_1,t^dx_2)$.
 These results answer the question which was unsolved until recently as to how to characterize the range of all $p$ for which maximal functions associated with nonisotropic dilations of non-homogeneous curves of finite type in the plane is bounded.

In Section 5, employing some arguments used and some results obtained in the plane, we also establish $L^p$-estimates for the maximal function related to
the hypersurface $(x_1,x_2)\rightarrow (x_1,x_2, x_2^d(1+\mathcal{O}(x_2^m)))$ with associated dilations $\delta_t(x)=(t^ax_1,tx_2,t^dx_3)$ for
 an arbitrary real number $a>0$. These results could not be covered by the theorems about maximal functions associated with nonisotropic dilations
 of hypersurfaces in $\mathbb{R}^3$ from Greenleaf \cite{Greenleaf}, Iosevich-Sawyer \cite{ios2},  Ikromov-Kempe-M\"{u}ller \cite{IKMU} and Zimmermann \cite{Zimmermann}.
 Moreover,  by an alternative approach, we also get  $L^p$-estimates for some classes of maximal functions in $\mathbb{R}^3$ proved in \cite{Greenleaf}, \cite{ios2}, \cite{IKMU} and \cite{Zimmermann}.

\textbf{Conventions}: Throughout this article, we shall use the well known notation $A\ll B$, which means if there is a sufficiently large constant $G$, which does not depend on the relevant parameters arising in the context in which
the quantities $A$ and $B$ appear, such that $G A\leq B$.  We write
$A\approx B$, and mean that $A$ and $B$ are comparable. We write
$A \lesssim B$ if $A\ll B$ or $A \approx B$.  $A\wedge B$ means if $A\leq B$, then $A\wedge B=A$; if $A\geq B$, then $A\wedge B=B$.

\section[Maximal functions associated with nonisotropic dilations
 of curves in the plane]{Maximal functions associated with nonisotropic dilations
 of curves in the plane.}

\subsection[Background on Fourier integral operators and
auxiliary results]{Background on Fourier integral operators and
auxiliary results.}

\subsubsection[Local smoothing of Fourier integral operators]{Local smoothing of Fourier integral operators.}
In our settings, the averaging operator is always expressed as a Fourier integral operator, then we turn to prove the $L^p$ boundedness for these operators.
So here we will make a brief introduction to local smoothing of Fourier integral operators  in \cite{mss}.

We consider a class of Fourier integral operators $I^{\nu} (\mathbb{R}^{n+1},\mathbb{R}^n;\Lambda)$, which is determined by the properties of its canonical relation $\Lambda$,
which is a conic Lagrangian in  $T^*Y\backslash0$ to $T^*Z\backslash0$ with respect to the symplectic form $d\zeta\wedge dz-d\eta\wedge dy$,
and closed in $T^*Z\backslash0\times T^*Y\backslash0$. In fact, these assumptions imply that $\Lambda\subset T^*Z\backslash0\times T^*Y\backslash0$ is a conic (immersed) submanifold of dimension $2n+1$.

To guarantee local regularity properties of operators $\mathcal{F} \in I^{\nu} (\mathbb{R}^{n+1},\mathbb{R}^n;\Lambda)$,
we shall impose conditions on $\Lambda$ which are based on the properties of the following three projections

\begin{center}
\begin{tikzpicture}
  \tikzstyle{l1} = []
  \tikzstyle{l2} = []
  \node[l2] {$\Lambda$} [grow'=down]
  child {
    node[l2] {$T_{z_0}^*Z\backslash 0$}
  }
  child {
    node[l1] (t2) {Z}
  }
  child {
    node[l1] (t2) {$T^*Y\backslash 0$}
  };
\end{tikzpicture}

\end{center}

We assume that $\Pi_{X}$ is the projection of $\Lambda$ onto X, X=$T_{z_0}^*Z\backslash 0$, Z, or $T^*Y\backslash 0$.  The condition has two parts:

(1) non-degeneracy condition:
\begin{equation}\label{nondegeneracy1}
\textrm{rank}\hspace{0.2cm} d\Pi_{T^*Y}\equiv 2n,
\end{equation}
\begin{equation}\label{nondegeneracy2}
\textrm{rank}\hspace{0.2cm} d\Pi_{Z}\equiv n+1;
\end{equation}
analogue of the Carleson-Sj\"{o}lin condition for nonhomogeneous phase (for every $z_0\in Z$, $\Gamma_{z_0}:=\Pi_{T_{z_0}^*Z}(\Lambda)$ has $n$ non-vanishing principal curvature);

(2) cone condition: for every $z_0\in Z$, $\Gamma_{z_0}$ is a smooth conic $n$-dimensional hypersurface,   $n-1$ principal curvatures do not vanish.
A Fourier integral operator which satisfies (\ref{nondegeneracy1}), (\ref{nondegeneracy2}) and the cone condition is said to satisfy the "cinematic curvature" condition.

The exact description of this condition can be found in \cite{sogge2}. Here we like to see how the condition can be reformulated if we use local coordinates.
The non-degeneracy condition implies that near a given point $(z_0,\zeta_0,y_0,\eta_0)\in \Lambda$,  local coordinates can be chosen so that $\Lambda\ni (z,\zeta,y, \eta)\rightarrow (z,\eta)$ has bijective differential and there must be a phase function $\varphi(z,\eta)$ so that $\Lambda$ takes the form
\begin{equation}\label{cone:condition1}
\{(z,\varphi'_z(z,\eta),\varphi'_{\eta}(z,\eta),\eta):\hspace{0.2cm}\eta \in \mathbb{R}^n\backslash 0 \hspace{0.1cm}\textrm{in}\hspace{0.1cm}\textrm{ a } \hspace{0.1cm}\textrm{conic}\hspace{0.1cm} \textrm{neighborhood} \hspace{0.1cm}\textrm{of}\hspace{0.1cm} \eta_0\}.
\end{equation}

In this case, the condition (\ref{nondegeneracy1}) becomes
\begin{equation}\label{nondegeneracy}
\textrm{rank}\hspace{0.1cm} \varphi_{z,\eta}''\equiv n,
\end{equation}
which means if we fix $z_0$, then,
$$\Gamma_{z_0}=\{\varphi_z'(z_0,\eta):\hspace{0.2cm}\eta\in \mathbb{R}^n\backslash0 \hspace{0.3cm} \textrm{in} \hspace{0.1cm}\textrm{ a } \hspace{0.1cm}\textrm{conic}\hspace{0.1cm} \textrm{neighborhood} \hspace{0.1cm}\textrm{of}\hspace{0.1cm} \eta_0\}\subset T_{z_0}^*Z\backslash 0$$
must be a smooth conic submanifold of dimension $n$. Then if $\Gamma_{z_0} \ni \zeta=\varphi_z'(z_0,\eta)$ and $\theta \in S^n$ is normal to $\Gamma_{z_0}$ at $\zeta$, it follows that $\pm \theta$ are the unique directions for which $\bigtriangledown_{\eta}\langle\varphi'_z(z_0,\eta),\theta\rangle=0$. The cone condition (2) is just that
\begin{equation}\label{cone:condition}
\textrm{rank}(\frac{\partial^2}{\partial\eta_j\partial\eta_k})\langle\varphi_z'(z_0,\eta),\theta\rangle=n-1 \hspace{0.2cm}\textrm{if}\hspace{0.1cm} \eta,\hspace{0.1cm}\theta \hspace{0.1cm}\textrm{are}\hspace{0.1cm} \textrm{obtained} \hspace{0.1cm}\textrm{from}\hspace{0.1cm} \textrm{the}\hspace{0.1cm}\textrm{ above}.
\end{equation}


Suppose $\mathcal{F}\in I^{\mu-1/4}(Z,Y;\Lambda)$, where $\Lambda$ satisfies (\ref{nondegeneracy1}), (\ref{nondegeneracy2}) and (\ref{cone:condition1}), $\mathcal{F}f$ can be written as a finite sum of the form
\begin{equation}\label{localinteg}
\int_{{\mathbb{R}}^n}e^{i\varphi(z,\eta)}a(z,\eta)\widehat{f}(\eta)d\eta,\hspace{0.3cm}f\in C_0^{\infty}(\mathbb{R}^n),
\end{equation}
where the phase function $\varphi$  satisfyies (\ref{nondegeneracy}) and (\ref{cone:condition}) and the symbol $a$ of order $\mu$ has small conic support in $\mathbb{R}^{n+1}\times \mathbb{R}^n$, which means that $a$ vanishes for all $z$
outside a small compact set and for all $\eta=(\eta_1,\eta')$
outside a narrow cone $\{\eta:|\eta'|\leq \varepsilon\eta_1\}$.

In this article, we always consider the case $n=2$ and $z=(x,t)\in \mathbb{R}^2\times \mathbb{R}$.  We fix $\beta\in C_0^{\infty}(\mathbb{R})$ supported in $[1/2,2]$ and set $a_{\lambda}(x,t,\eta)=\lambda^{-\mu}\beta(|\eta|/\lambda)a(x,t,\eta)$ for fixed $\lambda>1$. Then $a_{\lambda}$ is a symbol of order zero and satisfies the usual symbol estimates uniformly in $\lambda$.  Mockenhaupt, Seeger and Sogge show that the dyadic estimate of the Fourier integral operator $\mathcal{F}$ is as following.
\begin{thm}\cite{mss}\label{lem:lemma4}
\begin{equation*}
\biggl(\int_{1/2}^4\int_{{\mathbb{R}}^2}\left|\int_{{\mathbb{R}}^2}e^{i\varphi(x,t,\eta)}a_{\lambda}(x,t,\eta)\widehat{f}(\eta)d\eta\right|^pdxdt\biggl)^{1/p}\leq
C_p {\lambda}^{1/2-1/p-\epsilon(p)},
\end{equation*}

where $\epsilon(p)=\frac{1}{2p}$, if $4\leq p<\infty$;
$\epsilon(p)=\frac{1}{2}(\frac{1}{2}-\frac{1}{p})$, if $2<p\leq 4$.
\end{thm}

A very natural question is to ask how the $L^4$-estimate for a Fourier integral operator depends on the cinematic curvature. In Section 3 and Section 4, we show  that a class of Fourier integral operators
which do not satisfy the cinematic curvature condition uniformly, still satisfy a local smoothing
estimate. Just when we finished our work and searched for background materials on "cinematic curvature" online,
it came to our attention that in 2000, Kung already
obtained some results for a related problem in his thesis \cite{Kung} which has not been
 published until now and only show the results in \cite{Kungpaper} which we had not been aware of. We are thankful to Kung who provided us with his thesis soon after we wrote an Email to him. Through reading his thisis we know that the basic structure of both Kung's and our approach might appear similar,
 since both strategies rely on papers \cite{mss} and \cite{mss2} by Mockenhaupt, Seeger and Sogge. Nevertheless, our approach differs from \cite{Kung}, since we made use of a different angular
 decomposition. In this way, we obtain stronger $L^4$-estimates for a class of Fourier integral operators than in \cite{Kung}, which can be applied on establishing the $L^p$-boundedness of the maximal operator associated with
 dilations $\delta_t(x)=(tx_1,t^dx_2)$ of the curve $\gamma(s)=s^d(1+\phi(s))$, where $\phi(s)=\mathcal{O}(s^m)$, $m\geq 1$ and $d\geq 2$. In a model case,
 the associated cone for the corresponding Fourier integral operator $\mathcal{F}_{\lambda}^{\delta}$ which is localized to frequencies $|\xi|\approx \lambda$ is of the
  form $\{(\xi,\delta q(\xi))\}$, where $\delta>0$ is very small and $q(\xi)$ is homogeneous of degree one, smooth on the support of the symbol of $\mathcal{F}_{\lambda}^{\delta}$.
   Comparing the light cone, we observe that the level curves of the cone $\{(\xi,\delta q(\xi))\}$ become flatter and the cinematic curvature of $\mathcal{F}_{\lambda}^{\delta}$ is only greater than or equal to $\delta$.
   In \cite{Kungpaper}, Kung got that $\|\mathcal{F}_{\lambda}^{\delta}\|_{L^4\rightarrow L^4}\leq \lambda^{1/8+\epsilon_1}\delta^{-1/2}$. However, in this article
   we obtain the better estimate that $\|\mathcal{F}_{\lambda}^{\delta}\|_{L^4\rightarrow L^4}\leq \lambda^{1/8+\epsilon_1}\delta^{-(3/8+\epsilon_2-\epsilon_1)}$ (
   see Section 3 of this article ).

 Moreover, in \cite{Kungpaper}, Kung
extends his estimate for the operator $\mathcal{F}_{\lambda}^{\delta}$ to more general Fourier integral operators which correspond the cones $\{(\xi,q(x,t,\xi))\}$,
where for all $x$, $t$, the curvature of the curve $\{\xi:q(x,t,\xi)=1\}$ is greater than or equal to $\delta$. In view of an application to our problems,
we extend our estimate for the operator $\mathcal{F}_{\lambda}^{\delta}$ to a class of Fourier integral operators which correspond the cones $\{(\xi,\delta q_{\delta}'(\xi,t))\}$,
where $q_{\delta}(\xi,t)$ is homogeneous of degree one in $\xi$, smooth on the support of the symbol of the corresponding Fourier integral operator ( see Section 4 of this article ).
However, it is still open which exponent of $\delta$ is optimal.

\subsubsection[Auxiliary results]{Auxiliary results.}

We will often use the following method of stationary phase.
\begin{lem} (Theorem 1.2.1 in \cite{sogge2})\label{lem:Lemma1}
Let S be a smooth hypersurface in $\mathbb{R}^n$ with non-vanishing Gaussian curvature and $d\mu$ be the Lebesgue measure on $S$. Then,
\begin{equation}
|\widehat{d\mu}(\xi)|\leq C(1+|\xi|)^{-\frac{n-1}{2}}.
\end{equation}

Moreover, suppose that $\Gamma\subset \mathbb{R}^n\backslash 0$ is the cone consisting of all $\xi$ which are normal to $S$ at some point $x\in S$ belonging to a fixed relatively compact neighborhood $\mathcal{N}$ of supp $d\mu$. Then
\begin{equation*}
\left|\left(\frac{\partial}{\partial\xi}\right)^{\alpha}\widehat{d\mu}(\xi)\right|=\mathcal{O}(1+|\xi|)^{-N}\hspace{0.2cm}for\hspace{0.1cm} all \hspace{0.1cm}N\in \mathbb{N}, \hspace{0.2cm}if \hspace{0.2cm}\xi\not \in \Gamma,
\end{equation*}
\begin{equation}
\widehat{d\mu}(\xi)=\sum e^{-i\langle x_j,\xi \rangle}a_j(\xi)\hspace{0.5cm}if \hspace{0.2cm}\xi\in \Gamma,
\end{equation}
where the finite sum is taken over all $x_j\in \mathcal{N}$ having $\xi$ as the normal and
\begin{equation}
\left|\left(\frac{\partial}{\partial\xi}\right)^{\alpha}a_j(\xi)\right|\leq C_{\alpha}(1+|\xi|)^{-\frac{n-1}{2}-|\alpha|}.
\end{equation}

\end{lem}

 We also use the following well-known estimate.
\begin{lem} (Theorem 2.4.2 in \cite{sogge2})\label{lem:Lemma3}
Suppose that $F$ is $C^1(\mathbb{R})$. Then if $p>1$ and $1/p+1/p'=1$,
\begin{equation*}
\sup_{\lambda}|F(\lambda)|^p\leq |F(0)|^p+p\biggl(\int|F(\lambda)|^pd\lambda\biggl)^{1/p'}\biggl(\int|F'(\lambda)|^pd\lambda\biggl)^{1/p}.
\end{equation*}
\end{lem}

Next, we will introduce the definition of the nonisotropic Littlewood-Paley  operator associated with a parameter and show a related lemma.

Let $ \beta \in C_0^{\infty}(\mathbb{R})$ be non-negative and supp $\beta\subset [1/\sqrt{c_0},\sqrt{c_0}]$, where $c_0>1$ is a real number. For a fixed integer $j>0$, real numbers $a_1, a_2>0$ and $\ell\in \mathbb{Z}$, we define the non-isotropic Littlewood-Paley  operator with parameter $j$ as
\begin{equation}
\widehat{\Delta_{\ell}^j f}(\xi)=\beta(2^{-j}|\delta_{2^{\ell}}\xi|)\hat{f}(\xi),
\end{equation}
where $\delta_{2^{\ell}}\xi=(2^{\ell a_1}\xi_1,2^{\ell a_2}\xi_2)$.

\begin{lem}\label{Appendix}
Let
\begin{equation*}
g_j(f)=\biggl(\sum_{\ell\in \mathbb{Z}}|\Delta_{\ell}^j f|^2\biggl)^{1/2},
\end{equation*}
then the following holds true:
\begin{equation}\label{g_jf}
\|g_j(f)\|_{L^p(\mathbb{R}^2)}\leq C_p \|f\|_{L^p(\mathbb{R}^2)}.
\end{equation}
if $p\in (1,\infty)$, and $C_p$ only depends on $p$.
\end{lem}

\begin{proof}
The idea of proof is very similar to the classic one in \cite{steinbook2}. For completeness, we give the proof here. For $a=(a_1, a_2)$ and $\alpha=(\alpha_1, \alpha_2)$,
we define $(a,\alpha)=a_1\alpha_1 +a_2\alpha_2 $, $\xi^{\alpha}=\xi_1^{\alpha_1}\xi_2^{\alpha_2}$. Let $|\cdot|_{\delta}$ be a homogeneous norm,
i.e. $|\delta_{r}\xi|_{\delta}=r|\xi|_{\delta}$ for $r>0$. H\"{o}rmander's theorem in nonisotropic case is usually stated in the following way.

\begin{lem}\cite{Fabes} Let $m(x)\in L^{\infty}(\mathbb{R}^2)$, and assume $m(x)$ is $N$ times continuously differentiable where $N>|a|/2$; moreover, assume that
\begin{equation*}
\int_{R/2\leq |\xi|_{\delta}\leq 2R}\left|R^{(a,\alpha)}\left(\frac{\partial}{\partial \xi}\right)^{\alpha}m(\xi)\right|^2\frac{d\xi}{R^{|a|}}\leq C_{\alpha},\hspace{0.3cm} |m(\xi)|\leq C \hspace{0.3cm} a.e.,
\end{equation*}
where $C$ is independent of $R$, say $C\geq 1$.
Then there exists a constant $A_p$ such that
\begin{equation}
\|T_mf\|_{L^p}=\|\mathcal{F}^{-1}(m\hat{f})\|_{L^p}\leq A_p \|f\|_{L^p}, \hspace{0.2cm}1<p<\infty, \hspace{0.2cm}f\in C_0^{\infty}(\mathbb{R}^2),
\end{equation}
where $A_p$ depends only on $a$ and $p$. In particular, if $m $ satisfies that
\begin{equation}\label{T_m}
\left|\left(\frac{\partial}{\partial \xi}\right)^{\alpha}m(\xi)\right|\leq C |\xi|_{\delta}^{-(a,\alpha)},
\end{equation}
for all $\xi\neq 0$, then $T_m$ is bounded on $L^p(\mathbb{R}^2)$.
\end{lem}

Set
$$ R_{\ell}:=\left\{
\begin{aligned}
r_{2\ell}, \hspace{0.2cm}if\hspace{0.2cm} \ell\in \mathbb{Z}, \hspace{0.2cm}\ell >0;\\
r_{|2\ell|+1}, \hspace{0.2cm} if \hspace{0.2cm}\ell\in \mathbb{Z}, \hspace{0.2cm} \ell \leq 0,
\end{aligned}
\right.
$$
where $\{r_k\}_{k\in \mathbb{N}}$ is the Rademacher function system on $[0,1]$. Then
\begin{equation}\label{har}
g_j(f)(x)^p \approx \int_0^1|\sum_{\ell\in \mathbb{Z}}\Delta_{\ell}^j f(x)R_{\ell}(t)|^pdt.
\end{equation}

Since $|\cdot|_{\delta}\approx |\cdot|$ for $|\xi|\approx 1$, then it is easy to see that
$\sum_{\ell\in \mathbb{Z}}\beta(2^{-j}|\delta_{2^{\ell}}\xi|)\leq C$. So we can define
\begin{equation}
T_{m_j,t}f:=\left(\sum_{\ell\in \mathbb{Z}}R_{\ell}(t)\Delta_{\ell}^j\right)f,
\end{equation}
then $T_{m_j,t}$ is a Fourier multiplier operator with multiplier
\begin{equation*}
m_{j,t}(\xi):=\sum_{\ell\in \mathbb{Z}}R_{\ell}(t)\beta(2^{-j}|\delta_{2^{\ell}}\xi|).
\end{equation*}

However, $|R_{\ell}(t)|\leq 1$ for  all $\ell\in \mathbb{Z}$ and $\|T_{m_j,t}\|_{L^p\rightarrow L^p}=\|T_{m_0,t}\|_{L^p\rightarrow L^p}$. In fact, if  $\tau_{j}f(x)=2^{2j/p}f(2^{j}x)$, then
\begin{align*}
T_{m_j,t}f&=T_{2^{2j/p}\tau_{-j}m_{0,t}}f\\
&=\mathcal{F}^{-1}(2^{2j/p}\tau_{-j}m_{0,t}\hat{f})\\
&=2^{2j/p}\mathcal{F}^{-1}(\tau_{-j}(m_{0,t}\hat{f}(2^j\cdot)))\\
&=2^{2j}2^{-2j/p}\tau_j\mathcal{F}^{-1}(m_{0,t}\hat{f}(2^j\cdot))\\
&=\tau_j\mathcal{F}^{-1}(m_{0,t}2^{-2j/p}2^{2j}\hat{f}(2^j\cdot))\\
&=\tau_j\mathcal{F}^{-1}(m_{0,t}\widehat{\tau_{-j}f})=\tau_jT_{m_{0,t}}(\tau_{-j}f),
\end{align*}
and  $\|\tau_jT_{m_{0,t}}\tau_{-j}\|_{L^p\rightarrow L^p}=\|T_{m_{0,t}}\|_{L^p\rightarrow L^p}$ implies the desired estimate.

It thus suffices to show that $m_{0,t}$ satisfies (\ref{T_m}).  Since for some $\ell\in \mathbb{Z}$,
$\beta(|\delta_{2^{\ell}}\xi|)\neq 0\Rightarrow|\xi|_{\delta}\approx 2^{-\ell}$,
then
\begin{equation*}
\sum_{\ell\in \mathbb{Z}}\left(\frac{\partial}{\partial \xi}\right)^{\alpha}\beta(|\delta_{2^{\ell}}\xi|) \leq C_{\alpha}C |\xi|_{\delta}^{-(a,\alpha)},
\end{equation*}
for arbitrary $\xi\neq 0$, which implies that $m_{0,t}$ satisfies (\ref{T_m}). Then there exists a constant $A_p$ such that
\begin{equation*}
\|T_{m_{j,t}}f\|_{L^p}\leq A_p\|f\|_{L^p}, \hspace{0.2cm}\textmd{for} \hspace{0.1cm}\textmd{all }\hspace{0.1cm}f\in L^p(\mathbb{R}^2).
\end{equation*}

Furthermore, (\ref{har}) and Fubini's theorem gives
\begin{align*}\|g_jf\|_{L^p}^p&\approx\int_{\mathbb{R}^2}\biggl(\int_0^1|T_{m_{j,t}}f|^pdt\biggl)dx\\
&=\int_0^1\|T_{m_{j,t}}f\|_{L^p}^pdt\\
&\leq \int_0^1A_p^p\|f\|_{L^p}^pdt=A_p^p\|f\|_{L^p}^p.
\end{align*}

\end{proof}


\subsection[The proof for curves with non-vanishing Gaussian curvature]{The proof for curves with non-vanishing Gaussian curvature.}\label{nonvanishingcur}

In this section, we will give the proof of Theorem \ref{planetheorem}.

\subsubsection[The case when $m=1$]{The case when $m=1$.} \label{nonvanishingsec1}

In this section, $\phi \in C^{\infty}(\mathbb{R})$ satisfies the following condition
\begin{equation}\label{phi}
\phi(0)\neq 0,\hspace{0.2cm}\phi'(0)\neq 0.
 \end{equation}

We choose $B>0$ very small and  $\tilde{\rho}\in C_0^{\infty}(\mathbb{R})$ such that supp $\tilde{\rho}\subset\{x:B/2\leq|x|\leq 2B\}$  and $\sum_k\tilde{\rho}(2^kx)=1$ for $x\in \mathbb{R}$.

Put
\begin{align*}
A_tf(y):&=\int f(y_1-tx,y_2-t^2x^2\phi(x))\eta(x)dx\\
&=\sum_k\int f(y_1-tx,y_2-t^2x^2\phi(x))\tilde{\rho}(2^kx)\eta(x)dx=\sum_kA_t^kf(y),
\end{align*}
where
\begin{equation}
A_t^kf(y):=\int f(y_1-tx,y_2-t^2x^2\phi(x))\tilde{\rho}(2^kx)\eta(x)dx.
\end{equation}

Since $\eta$ is supported in a sufficiently small neighborhood of the origin, then we only need to consider $k>0$ sufficiently large.

Considering the isometric operator on $L^p(\mathbb{R}^2)$ defined by $T_kf(x_1,x_2)=2^{3k/p}f(2^kx_1,2^{2k}x_2)$, one can compute that
\begin{equation}
T_k^{-1}A_t^kT_kf(y)=2^{-k}\int f(y_1-tx,y_2-t^2x^2\phi(\frac{x}{2^k}))\tilde{\rho}(x)\eta(2^{-k}x)dx.
\end{equation}

Then it suffices to prove the following estimate
\begin{equation}\label{Target}
\sum_k2^{-k}\left\|\sup_{t>0}|\widetilde{A_t^k}|\right\|_{L^p\rightarrow L^p}\leq C_p, \hspace{0.2cm}\textrm{for }\hspace{0.1cm}\textrm{all} \hspace{0.1cm} p>2,
\end{equation}
where
\begin{equation}
\widetilde{A_t^k}f(y):=\int f(y_1-tx,y_2-t^2x^2\phi(\frac{x}{2^k}))\tilde{\rho}(x)\eta(2^{-k}x)dx.
\end{equation}

By means of the Fourier inversion formula, we have
\begin{align*}
\widetilde{A_t^k}f(y)&=\int_{{\mathbb{R}}^2}e^{i\xi\cdot y}\int_{\mathbb{R}}e^{-i(t\xi_1x+t^2\xi_2x^2\phi(\frac{x}{2^k}))}\tilde{\rho}(x)\eta(2^{-k}x)dx\hat{f}(\xi)d\xi
\\
&=\int_{{\mathbb{R}}^2}e^{i\xi\cdot y}\widehat{d\mu_k}(\delta_t\xi)\hat{f}(\xi)d\xi,
\end{align*}
where
\begin{equation}\label{dmu_k}
\widehat{d\mu_k}(\xi):=\int_{\mathbb{R}}e^{-i(\xi_1x+\xi_2x^2\phi(\frac{x}{2^k}))}\tilde{\rho}(x)\eta(2^{-k}x)dx.
\end{equation}

We choose a non-negative function $\beta\in C_0^{\infty}(\mathbb{R})$ such that supp $\beta\subset[1/2,2]$ and $\sum_{j\in\mathbb{Z}}\beta(2^{-j}r)=1$ for $r>0$. Define the dyadic operators
\begin{equation}
\widetilde{A_{t,j}^k}f(y)=\int_{{\mathbb{R}}^2}e^{i\xi\cdot y}\widehat{d\mu_k}(\delta_t\xi)\beta(2^{-j}|\delta_t\xi|)\hat{f}(\xi)d\xi,
\end{equation}
and denote by $\widetilde{\mathcal{M}_{j}^k}$ the corresponding maximal operator.  Now we have that
\begin{equation*}
\sup_{t>0}|\widetilde{A_t^k}f(y)|\leq \widetilde{\mathcal{M}^{k,0}}f(y)+\sum_{j\geq 1}\widetilde{\mathcal{M}_{j}^k}f(y), \hspace{0.2cm}\textmd{for}\hspace{0.2cm}y\in \mathbb{R}^2.
\end{equation*}
where
\begin{equation}
\widetilde{\mathcal{M}^{k,0}}f(y):=\sup_{t>0}|\sum_{j\leq 0}\widetilde{A_{t,j}^k}f(y)|.
\end{equation}

We observe that $\widetilde{\mathcal{M}^{k,0}}f(y)=\sup_{t>0}|f*K_{\delta_{t^{-1}}}(y)|$, where $K_{\delta_{t^{-1}}}(x)=t^{-3}K(\frac{x_1}{t},\frac{x_2}{t^2})$ and
\begin{equation}
K(y):=\int_{{\mathbb{R}}^2}e^{i\xi\cdot y}\widehat{d\mu_k}(\xi)\rho(|\xi|)d\xi,
 \end{equation}
where $\rho\in C_0^{\infty}(\mathbb{R})$ is supported in $[0,1]$.  Since  $\phi$ satisfies (\ref{phi}), supp $\tilde{\rho}\subset \{x:B/2\leq |x|\leq 2B\}$,
then Lemma \ref{lem:Lemma1} implies that for a multi-index $\alpha$,
\begin{equation}
\left|\left(\frac{\partial}{\partial\xi}\right)^{\alpha}\widehat{d\mu_k}(\xi)\right|\leq C_{B,\alpha}(1+|\xi|)^{-1/2}.
\end{equation}

By integration by parts, we obtain that
\begin{equation}
|K(y)|\leq C_N (1+|y|)^{-N}.
\end{equation}
Then $\widetilde{\mathcal{M}^{k,0}}f(y)\leq C_N'Mf(y)$, where $M$ is the Hardy-Littlewood maximal operator defined by
\begin{equation}\label{Hardylittlewood}
Mf(y):=\sup_{t>0}\frac{1}{|B(y,t)|}\int_{B(y,t)}f(x)dx,
\end{equation}
where $B(y,t)=\{x:|y-x|_{\delta}<t\}$ and $|x|_{\delta}=\max\{|x_1|,|x_2|^{1/2}\}$. For the exact description, one can see page 8-13 in \cite{steinbook}. So it suffices to prove that
\begin{equation}
\sum_k2^{-k}\sum_{j\geq 1}\|\widetilde{\mathcal{M}_{j}^k}\|_{L^p\rightarrow L^p} \leq C_p.
\end{equation}

Since $\widetilde{A_{t,j}^k}$ is localized to frequency $|\delta_t\xi|\approx 2^j$, we will show that
\begin{equation*}
 \|\widetilde{\mathcal{M}_{j}^k}\|_{L^p\rightarrow L^p}\lesssim \|\mathcal{M}_{j,loc}^k\|_{L^p\rightarrow L^p},
\end{equation*}
where $\mathcal{M}_{j,loc}^kf(y)=\sup_{t\in [1,2]}|\widetilde{A_{t,j}^k}f(y)|$. In fact, for fixed $j\geq 1$ and all $\ell\in \mathbb{Z}$, let $\bigtriangleup^j_{\ell}$ be the non-isotropic Littlewood-Paley
operator in ${\mathbb{R}}^2$ defined by
$\widehat{\bigtriangleup^j_{\ell}f}(\xi)=\tilde{\beta}(2^{-j}|\delta_{2^{\ell}}\xi|)\hat{f}(\xi)$, here $\tilde{\beta}\in C_0^{\infty}(\mathbb{R})$ is nonnegative and satisfies $\beta(|\delta_{t}\xi|)=\beta(|\delta_{t}\xi|)\tilde{\beta}(|\xi|)$, for any $t\in [1,2]$. Then
\begin{align*}
\widetilde{M_{j}^k}f(y)&=\sup_{\ell\in \mathbb{Z}}\sup_{t\in[1,2]}\left|\int_{{\mathbb{R}}^2}e^{i\xi\cdot y}\widehat{d\mu_{k}}(\delta_{2^{\ell}t}\xi)\beta(2^{-j}
|\delta_{2^{\ell}t}\xi|)\tilde{\beta}(2^{-j}|\delta_{2^{\ell}}\xi|)
\hat{f}(\xi)d\xi\right|
\\
&\leq \biggl(\sum_{\ell\in \mathbb{Z}}\sup_{t\in[1,2]}\left|\int_{{\mathbb{R}}^2}e^{i\xi\cdot y}\widehat{d\mu_{k}}(\delta_{2^{\ell}t}\xi)\beta(2^{-j}
|\delta_{2^{\ell}t}\xi|)\widehat{\bigtriangleup^j_{\ell}f}(\xi)d\xi\right|^p\biggl)^{1/p}
\\
&=\biggl(\sum_{\ell\in \mathbb{Z}}\sup_{t\in[1,2]}\left|2^{-3\ell }\int_{{\mathbb{R}}^2}e^{i\xi\cdot \delta_{2^{-\ell}}y}\widehat{d\mu_{k}}(\delta_{t}\xi)\beta(2^{-j}
|\delta_{t}\xi|)\widehat{\bigtriangleup^j_{\ell}f}(\delta_{2^{-\ell}}\xi)d\xi\right|^p\biggl)^{1/p}
\\
&=\biggl(\sum_{\ell\in \mathbb{Z}}\biggl|\mathcal{M}_{j,loc}^k(\bigtriangleup^j_{\ell}f\circ\delta_{2^{\ell}})(\delta_{2^{-\ell}}y)\biggl|^p\biggl)^{1/p}.
\end{align*}

Since $p>2$, Lemma \ref{Appendix} implies that
\begin{align*}
\|\widetilde{M_{j}^k}f\|_{L^p}^p
&\leq \sum_{\ell\in \mathbb{Z}}\int_{{\mathbb{R}}^2}\biggl|\mathcal{M}_{j,loc}^k(\bigtriangleup^j_{\ell}f\circ\delta_{2^{\ell}})(\delta_{2^{-\ell}}y)\biggl|^pdy
\\
& =\sum_{\ell\in \mathbb{Z}}2^{3\ell }\biggl\|\mathcal{M}_{j,loc}^k(\bigtriangleup^j_{\ell}f\circ\delta_{2^{\ell}})\biggl\|_{L^p}^p
\\
& \leq \|\mathcal{M}_{j,loc}^k\|_{L^p\rightarrow L^p}^p \sum_{\ell\in \mathbb{Z}}\int_{{\mathbb{R}}^2}|\bigtriangleup^j_{\ell}f(y)|^pdy
\\
&\leq \|\mathcal{M}_{j,loc}^k\|_{L^p\rightarrow L^p}^p \biggl\|\biggl(\sum_{\ell\in \mathbb{Z}}|\bigtriangleup^j_{\ell}f(y)|^2\biggl)^{1/2}\biggl\|_{L^p({\mathbb{R}}^2)}^p
\\
&\lesssim\|\mathcal{M}_{j,loc}^k\|_{L^p\rightarrow L^p}^p\|f\|_{L^p({\mathbb{R}}^2)}^p.
\end{align*}

Based on the above argument, next we will only consider
\begin{equation}\label{loc}
\sum_k2^{-k}\sum_{j\geq 1}\|\mathcal{M}_{j,loc}^k\|_{L^p\rightarrow L^p} \leq C_p.
\end{equation}

In order to get (\ref{loc}), first we will estimate
\begin{equation}
\widehat{d\mu_k}(\delta_t\xi)=\int_{\mathbb{R}}e^{-it^2\xi_2(-sx+x^2\phi(\delta x))}\tilde{\rho}(x)\eta(\delta x)dx,
\end{equation}
where $2^{-k}=\delta$ and
\begin{equation}
s:=s(\xi,t)=-\frac{\xi_1}{t\xi_2}, \hspace{0.3cm}\textrm{for} \hspace{0.2cm} \xi_2\neq 0.
\end{equation}
If $\xi_2=0$, then
\begin{equation*}
|\widehat{d\mu_k}(\delta_t\xi)|=|(\eta(\delta\cdot)\tilde{\rho})^{\wedge}(t\xi_1)|\leq \frac{C_N'}{(1+|t\xi_1|)^N}=\frac{C_N'}{(1+|\delta_t\xi|)^N},
\end{equation*}
and for multi-index $\alpha$,
\begin{equation*}
|D_{\xi}^{\alpha}\widehat{d\mu_k}(\delta_t\xi)|=|D_{\xi}^{\alpha}(\eta(\delta\cdot)\tilde{\rho})^{\wedge}(t\xi_1)|\leq \frac{C_{\alpha,N}}{(1+|\delta_t\xi|)^N}.
\end{equation*}
Since $t\approx 1$, then we can deduce the case $\xi_2=0$ into $B_k$ of the following (\ref{station}).

Put
\begin{equation}
\Phi(s,x,\delta)=-sx+x^2\phi(\delta x),
\end{equation}
then we have
\begin{equation*}
\partial_x\Phi(s,x,\delta)=-s+2x\phi(\delta x)+x^2\delta\phi'(\delta x)
\end{equation*}
and
\begin{equation*}
\partial_x^2\Phi(s,x,\delta)=2\phi(\delta x)+4x\delta\phi'(\delta x)+x^2\delta^2\phi''(\delta x).
\end{equation*}
Since $k$ is sufficiently large and $\phi(0)\neq 0$, then the implicit function theorem implies that there exists a smooth solution $x_c=\tilde{q}(s,\delta)$ of the equation $\partial_x\Phi(s,x,\delta)=0$. For the sake of simplicity, we may assume $\phi(0)=1/2$.

Meanwhile, we observe when $k$ tends to infinity, $\tilde{q}(s,\delta)$ smoothly converges to the solution $\tilde {q}(s,0)=s$ of the equation  $\partial_x\Phi(s,x,0)=0$.

Let $\tilde{\Phi}(s,\delta):=\Phi(s,\tilde{q}(s,\delta),\delta)$. From the above arguments and Taylor expansion of smooth functions $\tilde{q}(s,\delta)$ and $\phi(x)$,  the phase function
\begin{equation}\label{phase function}
-t^2\xi_2\tilde{\Phi}(s,\delta)=\frac{\xi_1^2}{2\xi_2}+\delta\frac{\xi_1^3}{t\xi_2^2}\phi'(0) + \delta^2R(t,\xi,\delta),
\end{equation}
where $R(t,\xi,\delta)$ is homogeneous of degree one in $\xi$.  $-t^2\xi_2\tilde{\Phi}(s,\delta)$ can be considered as a small perturbation of  $\frac{\xi_1^2}{2\xi_2}+\delta\frac{\xi_1^3}{t\xi_2^2}\phi'(0)$.

By applying the method of stationary phase, we have
 \begin{equation}\label{station}
\widehat{d\mu_k}(\delta_t\xi)=e^{-it^2\xi_2\tilde{\Phi}(s,\delta)}\chi_k(\frac{\xi_1}{t\xi_2})
\frac{A_k(\delta_t\xi)}{(1+|\delta_t\xi|)^{1/2}}+B_k(\delta_t\xi),
\end{equation}
where $\chi_k$ is a smooth function supported in the interval $[c_k,\tilde{c}_k]$, for certain non-zero positive constants $c_1\leq c_k, \tilde{c_k}\leq c_2$ depending only on $k$. $\{A_k(\delta_t\xi)\}_k$ is contained in a bounded subset of symbols of order zero. More precisely, for arbitrary $t \in [1,2]$,
 \begin{equation}\label{symbol}
 |D_{\xi}^{\alpha}A_k(\delta_t\xi)|\leq C_{\alpha}(1+|\xi|)^{-\alpha},
\end{equation}
where $C_{\alpha}$ do not depend on $k$ and $t$. Furthermore, $B_k$ is a remainder term and satisfies for arbitrary $t\in [1,2]$,
\begin{equation}\label{symbol2}
 |D_{\xi}^{\alpha}B_k(\delta_t\xi)|\leq C_{\alpha,N}(1+|\xi|)^{-N},
\end{equation}
where $C_{\alpha,N}$ are admissible constants and again do not depend on $k$ and $t$.

First, let us consider the remainder part of (\ref{loc}). Set
\begin{equation}
M_{j}^{k,0}f(y):=\sup_{t\in[1,2]}\left|\int_{{\mathbb{R}}^2}e^{i\xi\cdot y}B_k(\delta_t\xi)\beta(2^{-j}|\delta_t\xi|)\hat{f}(\xi)d\xi \right|.
\end{equation}
By (\ref{symbol2}) and integration by parts, it is easy to get $|(B_k\beta(2^{-j}\cdot))^{\vee}(x)|\leq C_N 2^{-jN}(1+|x|)^{-N}$. So we have that
\begin{align*}
\|M_{j}^{k,0}\|_{L^p\rightarrow L^p}&=\sup\{\|M_{j}^{k,0}f\|_{L^p}:\|f\|_{L^p}=1\}\\
&=\sup\{\|\sup_{t\in[1,2]}|t^{-3}(B_k\beta(2^{-j}\cdot))^{\vee}(\delta_{t^{-1}}\cdot)*f|\|_{L^p}:\|f\|_{L^p}=1\}\\
&\leq \sup\{C_N 2^{-jN}\|Mf\|_{L^p}:\|f\|_{L^p}=1\}\\
&\leq C_N 2^{-jN},
\end{align*}
where $M$ is the Hardy-Littlewood maximal operator defined by (\ref{Hardylittlewood}), and the $L^p$-boundedness  $(1<p<\infty)$ of $M$ implies (\ref{Target}) for remainder part of (\ref{loc}).

Put
\begin{equation}
A_{t,j}^kf(y):=\int_{{\mathbb{R}}^2}e^{i(\xi \cdot y-t^2\xi_2\tilde{\Phi}(s,\delta))}\chi_k(\frac{\xi_1}{t\xi_2})
\frac{A_k(\delta_t\xi)}{(1+|\delta_t\xi|)^{1/2}}\beta(2^{-j}|\delta_t \xi|)\hat{f}(\xi)d\xi.
\end{equation}
Denote by $M_{j}^{k,1}$ the corresponding maximal operator over $[1,2]$. It remains to prove that
\begin{equation} \label{object}
\sum_k 2^{-k}\sum_{j\geq 1} \|M_{j}^{k,1}\|_{L^p\rightarrow L^p}\leq C_{p,N}.
\end{equation}

Since $\tilde{\Phi}(s,\delta)$ is homogeneous of degree zero in $\xi$ and $\frac{\xi_1}{\xi_2}\approx 1$, then
$$\left|\nabla_{\xi}[\xi \cdot (y-x)-t^2\xi_2\tilde{\Phi}(s,\delta)]\right|\geq C|y-x|$$
provided $|y-x|\geq L$, where $L$ is very large and determined by $c_1$, $c_2$ and $\|\phi\|_{\infty\hspace{0.1cm} (I)}$. By integration by parts, we will see that the kernel of the operator $A_{t,j}^k$ is dominated by $2^{-jN}\mathcal{O}(|y-x|^{-N})$ if $|y-x|\geq L$.  From now on, we will restrict our view on the situation
\begin{equation}\label{case}
|y-x|\leq L.
\end{equation}

Let $B_i(L)$ be a ball with center $i$ and radius $L$. Furthermore, we will show that
\begin{equation}\label{local}
\sup\{\|M_{j}^{k,1}f\|_{L^p}:\|f\|_{L^p}=1, \textmd{supp}\hspace{0.2cm} f\subset B_0(L)\}\leq C_{p}2^{-j\tilde{\epsilon}_1(p)}2^{k\tilde{\epsilon}_2(p)}
\end{equation}
implies that
\begin{equation}\label{local2}
\|M_{j}^{k,1}\|_{L^p\rightarrow L^p}\leq C_{p}2^{-j\tilde{\epsilon}_1(p)}2^{k\tilde{\epsilon}_2(p)},
\end{equation}
where $C_{p}$ depends on $p$, $c_1$, $c_2$ and $\|\phi\|_{\infty\hspace{0.1cm} (I)}$, and $\tilde{\epsilon}_1(p)$, $\tilde{\epsilon}_2(p)>0$. Then in order to prove inequality (\ref{loc}), it suffices to prove inequality (\ref{local}).

In fact, we can decompose $f=\sum_{i\in L\mathbb{Z}^2} f_{i}$, where supp $f_{i}\subset B_i(L)$.  If we have proved (\ref{local}), then
\begin{align*}
\|M_{j}^{k,1}f\|_{L^p}&\leq \|\sum_{i\in L\mathbb{Z}^2}M_{j}^{k,1}f_i\|_{L^p}\\
&\leq \left(\sum_{i\in L\mathbb{Z}^2}\|M_{j}^{k,1}(f_{i}(\cdot+i))\|_{L^p}^p\right)^{1/p}\\
&\leq C_{p}2^{-j\tilde{\epsilon}_1(p)}2^{k\tilde{\epsilon}_2(p)} \left(\sum_{i\in L\mathbb{Z}^2}\|f_{i}(\cdot+i)\|_{L^p}^p\right)^{1/p}\\
&\leq C_{p}2^{-j\tilde{\epsilon}_1(p)}2^{k\tilde{\epsilon}_2(p)} \|f\|_{L^p},
\end{align*}
which implies (\ref{local2}).

Now we observe inequality (\ref{local}),  together with the assumption (\ref{case}),  we can choose $\rho_1\in C_0^{\infty}(\mathbb{R}^2\times[\frac{1}{2},4])$ such that (\ref{local}) will follow from that
\begin{equation}
\|\widetilde{M_{j}^{k,1}}\|_{L^p\rightarrow L^p}\leq C_{p}2^{-j\tilde{\epsilon}_1(p)}2^{k\tilde{\epsilon}_2(p)},
\end{equation}
where
\begin{equation}
\widetilde{M_{j}^{k,1}}f(y):=\sup_{t\in[1,2]}\left|\rho_1(y,t)A_{t,j}^kf(y)\right|.
\end{equation}

\begin{lem}\label{lemmaL^2}
\begin{equation*}
\|\widetilde{M_{j}^{k,1}}f\|_{L^p}\leq C_{p} 2^{-(j\wedge k)/p}\|f\|_{L^p}, \hspace{0.5cm}2\leq p\leq \infty.
\end{equation*}
\end{lem}

\begin{proof}
To prove the above lemma, we employ the M. Riesz interpolation theorem between the $L^2$ and the $L^{\infty}$-estimate. By Lemma \ref{lem:Lemma3}, we have
\begin{align*}
&\|\widetilde{M_{j}^{k,1}}f\|_{L^2}^2\\
&\leq\biggl(\int_{1/2}^4\int_{{\mathbb{R}}^2}|\rho_1(y,t)A_{t,j}^kf(y)|^2dydt\biggl)^{1/2}
\biggl(\int_{1/2}^4\int_{{\mathbb{R}}^2}|\frac{\partial}{\partial t}(\rho_1(y,t)A_{t,j}^kf(y))|^2dydt\biggl)^{1/2}\\
&=\|\|\rho_1(y,t)A_{t,j}^kf(y)\|_{(L^2,dy)}\|_{(L^2,dt)}
\biggl\|\left\|\frac{\partial}{\partial t}(\rho_1(y,t)A_{t,j}^kf(y))\right\|_{(L^2,dy)}\biggl\|_{(L^2,dt)}.
\end{align*}

Moreover,
\begin{equation}\label{h_k}
\frac{\partial}{\partial t}(\rho_1(y,t)A_{t,j}^kf(y))=\int_{{\mathbb{R}}^2}e^{i({\xi \cdot y-t^2\xi_2\tilde{\Phi}(s,\delta)})}h_k(y,t,\xi,j)\hat{f}(\xi)d\xi,
\end{equation}
where
\begin{align*}
h_k(y,t,\xi,j)&=\left(\frac{\partial}{\partial t}\rho_1(y,t)+i\frac{\partial}{\partial t}(-t^2\xi_2\tilde{\Phi}(s,\delta))\rho_1(y,t)\right)\chi_k(\frac{\xi_1}{t\xi_2})\frac{A_k(\delta_t\xi)}{(1+|\delta_t\xi|)^{1/2}}\beta(2^{-j}|\delta_t \xi|)\\
& \quad +\rho_1(y,t)\frac{\partial}{\partial t}(\chi_k(\frac{\xi_1}{t\xi_2}))\frac{A_k(\delta_t\xi)}{(1+|\delta_t\xi|)^{1/2}}\beta(2^{-j}|\delta_t \xi|)\\
& \quad +\rho_1(y,t)\chi_k(\frac{\xi_1}{t\xi_2})\frac{\partial}{\partial t}\left(\frac{A_k(\delta_t\xi)}{(1+|\delta_t\xi|)^{1/2}}\right)\beta(2^{-j}|\delta_t \xi|)\\
& \quad +\rho_1(y,t)\chi_k(\frac{\xi_1}{t\xi_2})\frac{A_k(\delta_t\xi)}{(1+|\delta_t\xi|)^{1/2}}\frac{\partial}{\partial t}(\beta(2^{-j}|\delta_t \xi|)).
\end{align*}
From (\ref{phase function}) and  (\ref{symbol}), we get
$|\frac{\partial}{\partial t}(-t^2\xi_2\tilde{\Phi}(s,\delta))|=|-\delta\frac{\xi_1^3}{t^2\xi_2^2}\phi'(0)+\delta^2\frac{\partial}{\partial t}R(t,\xi,\delta)|\lesssim \delta 2^{j}$, $|\frac{\partial}{\partial t}(\chi_k(\frac{\xi_1}{t\xi_2}))|\approx 1$,
$|\frac{\partial}{\partial t}\left(\frac{A_k(\delta_t\xi)}{(1+|\delta_t\xi|)^{1/2}}\right)| \lesssim 2^{-j/2}$ and $|\frac{\partial}{\partial t}(\beta(2^{-j}|\delta_t \xi|))|\lesssim 1$.
Therefore $|h_k(y,t,\xi,j)|\lesssim 2^{j/2}\delta+2^{-j/2}$. Then $(2^{j/2}\delta+2^{-j/2})^{-1}\frac{\partial}{\partial t}(\rho_1(y,t)A_{t,j}^kf(y))$ behaves like $2^{j/2}A_{t,j}^k$ which is a symbol of order zero. So we only consider the $L^2$-boundedness of the operator $2^{j/2}\rho A_{t,j}^k$, for fixed $t\in[1/2,4]$, however this is easily obtained by the Plancherel theorem, i.e.
\begin{align*}
\|2^{j/2}\rho_1(y,t)A_{t,j}^kf(y)\|_{(L^2,dy)}&\leq \|e^{-it^2\xi_2\tilde{\Phi}(s,\delta)}2^{j/2}\chi_k(\frac{\xi_1}{t\xi_2})
\frac{A_k(\delta_t\xi)}{(1+|\delta_t\xi|)^{1/2}}\beta(2^{-j}|\delta_t \xi|)\hat{f}(\xi)\|_{(L^2,d\xi)}\\
&\leq C  \|f\|_{L^2}.
\end{align*}

From the above arguments, we obtain
\begin{align*}
\|\widetilde{M_{j}^{k,1}}f\|_{L^2}^2&
\leq 2^{-j/2}(2^{j/2}\delta+2^{-j/2})\|\|2^{j/2}\rho_1(y,t)A_{t,j}^kf(y)\|_{(L^2,dy)}\|_{(L^2,dt)}\\
&\quad\times\biggl\|\left\|\frac{\partial}{\partial t}((2^{j/2}\delta+2^{-j/2})^{-1}\rho_1(y,t)A_{t,j}^kf(y))\right\|_{(L^2,dy)}\biggl\|_{(L^2,dt)}\\
&\leq C 2^{-j/2}(2^{j/2}\delta+2^{-j/2})\|f\|_{L^2}^2\\
&\leq C (\delta+2^{-j})\|f\|_{L^2}^2,
\end{align*}
then
\begin{equation}
\|\widetilde{M_{j}^{k,1}}f\|_{L^2}\leq C (\delta+2^{-j})^{1/2}\|f\|_{L^2}.
\end{equation}

Next, let us turn to prove
\begin{equation}\label{lemmaL^infty}
\|\widetilde{M_{j}^{k,1}}f\|_{L^{\infty}}\leq C\|f\|_{L^{\infty}}.
\end{equation}
Inequality (\ref{lemmaL^infty}) follows from the fact that the kernels of $A_{t,j}^k$ are uniformly bounded in $L^1({\mathbb{R}}^2)$. The idea of the proof can be found in page 406-408 in \cite{steinbook}. For completeness, in fact, this idea also will be used later, so we give a brief overview here.

First, we introduce the angular decomposition as in \cite{steinbook} of the $\xi$-space in the plane.  For each positive integer $j$, we consider a roughly equally spaced set of points with grid length $2^{-j/2}$ on the unit circle $S^1$; that is, we fix a collection $\{\xi_j^{\nu}\}_{\nu}$ of unit vectors, that satisfy:

$(a)$ $|\xi_j^{\nu}-\xi_j^{\nu'}|\geq 2^{-j/2}$, if $\nu\neq\nu'$;

$(b)$ if $\xi\in S^1$, then there exists a $\xi_j^{\nu}$ so that $|\xi-\xi_j^{\nu}|<2^{-j/2}$.

Let $\Gamma_j^{\nu}$ denote the corresponding cone in the $\xi$-space whose central direction is $\xi_j^{\nu}$, i.e.
\begin{equation*}
\Gamma_j^{\nu}=\{\xi:|\xi/|\xi|-\xi_j^{\nu}|\leq2\cdot 2^{-j/2}\}.
\end{equation*}

We can construct an associated partition of unity:
\begin{lem}\cite{steinbook}\label{angularde}
$\chi_j^{\nu}$ is  homogeneous of degree zero in $\xi$ and supported in $\Gamma_j^{\nu}$, with
\begin{equation}\label{anghomofunct1}
\sum_{\nu}\chi_j^{\nu}(\xi)=1 \hspace{0.5cm}\textrm{for}\hspace{0.2cm}\textrm{ all}\hspace{0.2cm} \xi\neq 0 \hspace{0.2cm}\textrm{and}\hspace{0.2cm} \textrm{all}\hspace{0.2cm} j,
\end{equation}
and
\begin{equation}\label{anghomofunct2}
|\partial_{\xi}^{\alpha}\chi_j^{\nu}(\xi)|\leq A_{\alpha}2^{|\alpha|j/2}|\xi|^{-|\alpha|}.
\end{equation}
\end{lem}

Hence, in order to establish (\ref{lemmaL^infty}), it is sufficient to prove
\begin{equation}\label{kernalL^infty}
\int_{{\mathbb{R}}^2}|K_t^{\nu}(y)|dy\leq C 2^{-j/2},
\end{equation}
where $C$ does not depend on $t$, $j$, $k$ and $\nu$, and
\begin{equation}\label{ker1}
K_t^{\nu}(y)=\rho_1(y,t)\int_{{\mathbb{R}}^2}e^{i(\xi \cdot y-t^2\xi_2\tilde{\Phi}(s,\delta))}\widetilde{A_k}(\xi,t)\beta(2^{-j}|\delta_t \xi|)\chi_j^{\nu}(\xi)d\xi,
\end{equation}
and  $\widetilde{A_k}(\xi,t)=\chi_k(\frac{\xi_1}{t\xi_2})
\frac{A_k(\delta_t\xi)}{(1+|\delta_t\xi|)^{1/2}}$.

Then we have
\begin{align*}
\|\widetilde{M_{j}^{k,1}}f\|_{L^{\infty}}&:=\|\sup_{t>0}|K_t*f|\|_{L^{\infty}}\\
&\leq \|f\|_{L^{\infty}}\sup_{t>0}\sum_{\nu\lesssim 2^{j/2}}\|K_t^{\nu}\|_{L^1}\\
&\lesssim\|f\|_{L^{\infty}},
\end{align*}
and get inequality (\ref{lemmaL^infty}).

For fixed $\nu$, the inner integral (\ref{ker1}) is supported in the truncated cone $\Gamma_j^{\nu}$. Let $T$ be the transpose operator, i.e. $T(x_1,x_2)=(x_1,x_2)^{\mathrm{t}}$. We can find a rotation $\rho_j^{\nu}$ such that $T^{-1}\rho_j^{\nu}T(1,0)=\xi_j^{\nu}$. Write $\tilde{T}=T^{-1}\rho_j^{\nu}T$ and $\tilde{T}_i$ the $i$-th variable after the action of $\tilde{T}$,
this leads to prove the following estimate
\begin{equation}\label{ker2}
\int_{{\mathbb{R}}^2}\biggl|\rho_1(y,t)\int_{{\mathbb{R}}^2}e^{i(\tilde{T}\xi \cdot y-t^2\tilde{T}_2\xi\tilde{\Phi}(\tilde{s},\delta))}\widetilde{A_k}(\tilde{T}\xi,t)\beta(2^{-j}|\delta_t \tilde{T}\xi|)\chi_j^{\nu}(\tilde{T}\xi)d\xi\biggl|dy\leq C 2^{-j/2},
\end{equation}
where $\tilde{s}=-\frac{\tilde{T}_1\xi}{t\tilde{T}_2\xi}$.
From the $\xi$-support of the inner integral of (\ref{ker2}), we know that $\xi$ belongs to the following region, see Figure 1.
\begin{center}
\includegraphics[height=6cm]{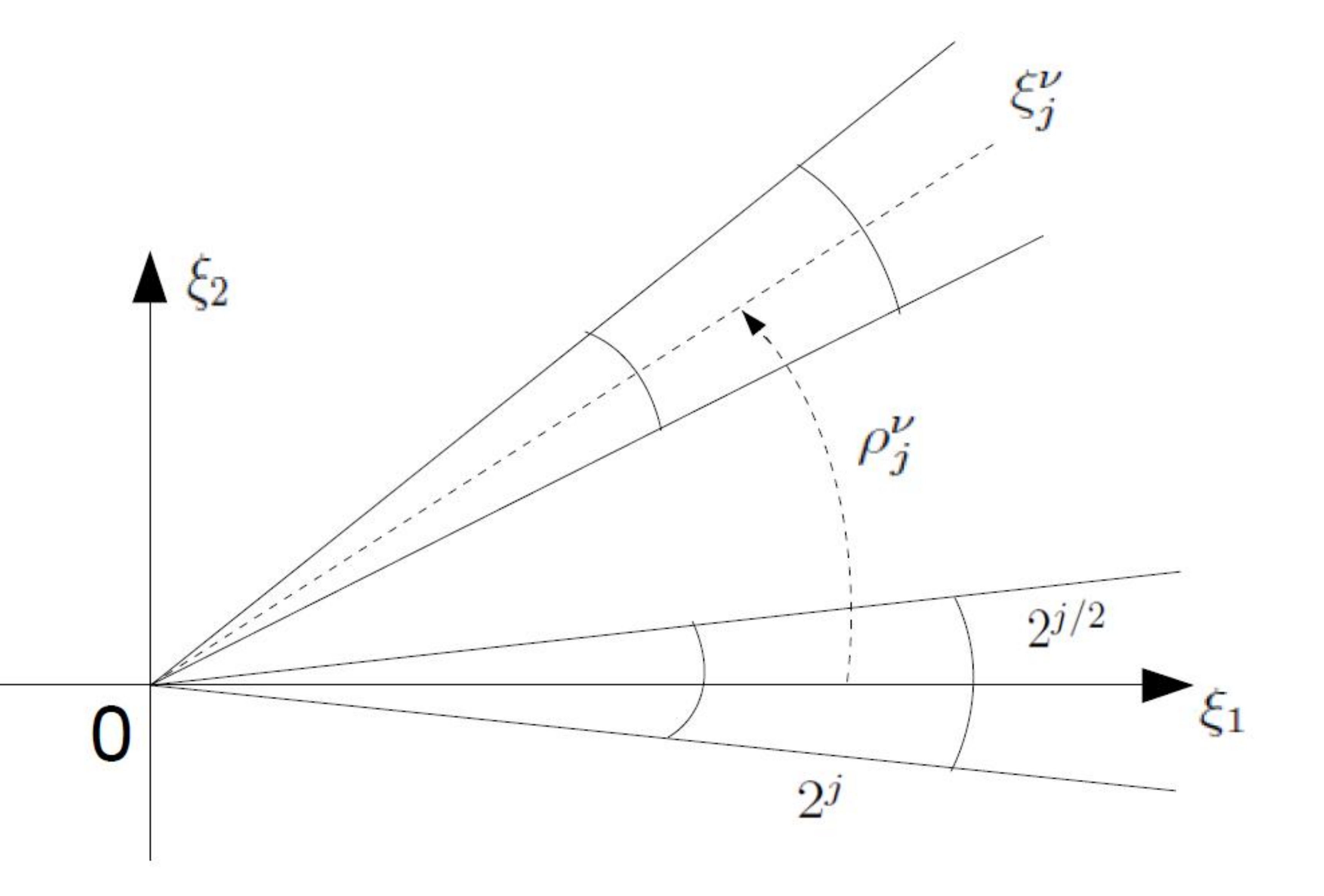}
\end{center}
\begin{center}
Figure 1
\end{center}

Put $\tilde{\tilde{\Phi}}(y,\xi)=\tilde{T}\xi \cdot y-t^2\tilde{T}_2\xi\tilde{\Phi}(\tilde{s},\delta)$. Clealy  $\tilde{\tilde{\Phi}}$ is homogeneous of degree one in $\xi$, hence $\tilde{\tilde{\Phi}}(y,\xi)=\xi\cdot \nabla_{\xi}\tilde{\tilde{\Phi}}(y,\xi)$. Let $h(\xi)=\tilde{\tilde{\Phi}}(y,\xi)-\xi\cdot \nabla_{\xi}\tilde{\tilde{\Phi}}(y,\bar{\xi})$, where $\bar{\xi}=(1,0)$. We claim that

\begin{equation}\label{h1}
\left|\left(\frac{\partial}{\partial\xi_1}\right)^Nh(\xi)\right|\leq A_N\cdot 2^{-jN},
\end{equation}
\begin{equation}\label{h2}
\left|\left(\frac{\partial}{\partial\xi_2}\right)^Nh(\xi)\right|\leq A_N\cdot 2^{-jN/2}.
\end{equation}

Since $h(\xi)$ is also homogeneous of degree one, then $h(\bar{\xi})=0$ and $\nabla_{\xi}h(\bar{\xi})=0$. As a result
\begin{equation*}
\left(\frac{\partial}{\partial\xi_1}\right)^Nh(\bar{\xi})=\frac{\partial}{\partial \xi_2}\left(\frac{\partial}{\partial\xi_1}\right)^Nh(\bar{\xi})=0.
\end{equation*}
Thus
$\left(\frac{\partial}{\partial\xi_1}\right)^Nh(\xi)=\mathcal{O}(\frac{|\xi_2|^2}{|\xi|^{N+1}})$. Since $|\xi_2|\lesssim 2^{j/2}$ and $|\xi|\approx 2^{j}$, we have that $\frac{|\xi_2|^2}{|\xi|^{N+1}}\leq \mathcal{O}(2^{-jN})$, we get inequality (\ref{h1}).

Again, $\frac{\partial}{\partial\xi_2}h(\bar{\xi})=0$, so $\frac{\partial}{\partial\xi_2}h(\xi)=\mathcal{O}(\frac{\xi_2}{|\xi|})\approx 2^{-j/2}$. For $N>1$, we have that $\left|\left(\frac{\partial}{\partial\xi_2}\right)^Nh(\xi)\right|\leq A_N|\xi|^{1-N}$, by homogeneity, and $2^{j(1-N)}\leq 2^{-Nj/2}$,  so we finish the proof of inequality (\ref{h2}).

We observe that $\chi_j^{\nu}\circ \tilde{T}$ satisfies the same conditions (\ref{anghomofunct1}) and (\ref{anghomofunct2}) as $\chi_j^{\nu}$. Meanwhile, supp $ (\chi_j^{\nu}\circ \tilde{T})\subset \{\xi:|\frac{\xi}{|\xi|}-(1,0)|\leq 2\cdot 2^{-j/2}\}$. If we set $\epsilon=2^{-j/2}$ in (iii) of Lemma 14 in \cite{AM}, then it follows that (\ref{anghomofunct2}) has an improvement for $\chi_j^{\nu}\circ \tilde{T}$, i.e.
\begin{equation}\label{anghomofunct2impr}
\left|\left(\frac{\partial}{\partial\xi_1}\right)^N(\chi_j^{\nu}\circ \tilde{T})(\xi)\right|\leq C_N|\xi|^{-N}\approx2^{-jN}.
\end{equation}

We now rewrite
\begin{equation*}
K_t^{\nu}(y)=\rho_1(y,t)\int_{{\mathbb{R}}^2}e^{i\xi\cdot \nabla_{\xi}\tilde{\tilde{\Phi}}(y,\bar{\xi})}e^{ih(\xi)}\widetilde{A_k}(\tilde{T}\xi,t)\beta(2^{-j}|\delta_t \tilde{T}\xi|)(\chi_j^{\nu}\circ\tilde{T})(\xi)d\xi.
\end{equation*}
Next, we introduce the operator $\mathcal{L}$ defined by $\mathcal{L}=I-2^{2j}\frac{{\partial}^2}{{\partial \xi_1^2}}-2^{j}\frac{\partial^2}{\partial \xi_2^2}$. Because of (\ref{anghomofunct2}), (\ref{anghomofunct2impr}) and the fact that $\widetilde{A_k}$ is a symbol of order $-1/2$, we get that
\begin{equation*}
|\mathcal{L}^N(e^{ih(\xi)}\widetilde{A_k}(\tilde{T}\xi,t)\beta(2^{-j}|\delta_t \tilde{T}\xi|)(\chi_j^{\nu}\circ\tilde{T})(\xi))|\leq A_N 2^{-j/2}.
\end{equation*}
However,
\begin{equation*}
\mathcal{L}^N e^{i\xi\cdot \nabla_{\xi}\tilde{\tilde{\Phi}}(y,\bar{\xi})}=\biggl(1+2^{2j}|\frac{\partial}{\partial \xi_1}\tilde{\tilde{\Phi}}(y,\bar{\xi})|^2+2^{j}|\frac{\partial}{\partial \xi_2}\tilde{\tilde{\Phi}}(y,\bar{\xi})|^2\biggl)^N\cdot e^{i\xi\cdot \nabla_{\xi}\tilde{\tilde{\Phi}}(y,\bar{\xi})}.
\end{equation*}

From Figure 1, we note that the support of the integrand has volume at most $\mathcal{O}(2^j\cdot 2^{j/2})\approx 2^{3j/2}$, thus we obtain by integration by parts that
\begin{equation*}
|K_t^{\nu}(y)|\leq A_N 2^{3j/2}2^{-j/2}(1+2^{2j}|\frac{\partial}{\partial \xi_1}\tilde{\tilde{\Phi}}(y,\bar{\xi})|^2+2^{j}|\frac{\partial}{\partial \xi_2}\tilde{\tilde{\Phi}}(y,\bar{\xi})|^2)^{-N}.
\end{equation*}
In order to prove (\ref{kernalL^infty}), we make the change of variables $y\mapsto(\frac{\partial}{\partial \xi_1}\tilde{\tilde{\Phi}}(y,\bar{\xi}), \frac{\partial}{\partial \xi_2}\tilde{\tilde{\Phi}}(y,\bar{\xi}))$. Since
\begin{equation*}
\textrm{det }\hspace{0.1cm}
\left(
\begin{array} {lcr}
\frac{{\partial}^2\tilde{\tilde{\Phi}}(y,\bar{\xi})}{\partial \xi_i\partial y_i}
\end{array}
\right)\neq 0,
\end{equation*}
and the Jacobian is bounded from below, then we have
\begin{align*}
\int_{{\mathbb{R}}^2}|K_t^{\nu}(y)|dy\leq A_N 2^j\int_{{\mathbb{R}}^2}\frac{dy}{(1+|2^{j}y_1|+|2^{j/2}y_2|)^{2N}}
\leq C2^{-j/2}.
\end{align*}
\end{proof}

Now we split the set of $j$ into two parts $j>2k$ and $j\leq 2k$.  By Lemma \ref{lemmaL^2}, we get
\begin{equation}
\sum_k2^{-k}\sum_{j\leq 2k}\|\widetilde{M_{j}^{k,1}}\|_{L^p\rightarrow L^p}\leq C_p.
\end{equation}

For $j>2k$, we introduce some new notations. Let $j=k+j'$, then $j'>k$. Set $\lambda=2^{j'}>2^k=\delta^{-1}$, then $2^j=\frac{\lambda}{\delta}$ . By Lemma \ref{lem:Lemma3},
\begin{equation}\label{well}
\begin{aligned}
&\|\widetilde{M_{j}^{k,1}}f\|_{L^4}^4\\
&\leq C
\frac{\delta^2}{\lambda}\biggl(\int_{{\mathbb{R}}^2}\int^4_{1/2}\biggl|\rho_1(y,t)\int_{{\mathbb{R}}^2}e^{i(\xi \cdot y-t^2\xi_2\tilde{\Phi}(s,\delta))}(\lambda/\delta)^{1/2}\widetilde{A_k}(\xi,t)\beta(\frac{\delta}{\lambda}|\delta_t \xi|)\hat{f}(\xi)d\xi\biggl|^4dydt \biggl)^{3/4}\\
& \quad \times
\biggl(\int_{{\mathbb{R}}^2}\int^4_{1/2}\biggl|\frac{\partial}{\partial t}(\rho_1(y,t)\int_{{\mathbb{R}}^2}e^{i(\xi \cdot y-t^2\xi_2\tilde{\Phi}(s,\delta))}(\delta\lambda)^{-1/2}\widetilde{A_k}(\xi,t)\beta(\frac{\delta}{\lambda}|\delta_t \xi|)\hat{f}(\xi)d\xi)\biggl|^4dydt \biggl)^{1/4}.
\end{aligned}
\end{equation}

In order to simplify the notations, we choose $\tilde{\chi}\in C_0^{\infty}([c_1,c_2])$ so that  $\tilde{\chi}(\frac{\xi_1}{\xi_2})\chi(\frac{\xi_1}{t\xi_2})=\chi(\frac{\xi_1}{t\xi_2})$
 for arbitrary $t\in [1/2,4]$ and $k$ sufficiently large. In a similar way we choose $\rho_0\in C_0^{\infty}((-10,10))$ such that $\rho_0(|\xi|)\beta(|\delta_t \xi|)=\beta(|\delta_t \xi|)$ for arbitrary $t\in [1/2,4]$. Furthermore, since $A_k$ satisfies (\ref{symbol}), if  $a(\xi,t):=(\lambda/\delta)^{1/2}\widetilde{A}_k(\xi,t)\beta(\frac{\delta}{\lambda}|\delta_t \xi|)$ for $k$ sufficiently large, then $a(\xi,t)$ is a symbol of order zero, i.e. for any $t\in [1/2,4]$, $\alpha\in {\mathbb{N}}^2$,
\begin{equation}\label{symbola}
\left|\left(\frac{\partial}{\partial \xi}\right)^{\alpha}a(\xi,t)\right|\leq C_{\alpha}(1+|\xi|)^{-|\alpha|}.
\end{equation}

Assume that we have obtained the following theorem:
\begin{thm}\label{L^4therom}
For all $\lambda >1/\delta$,
\begin{equation}
\biggl(\int_{{\mathbb{R}}^2}\int^4_{1/2}|\tilde{F}_{\lambda}^{\delta}f(y,t)|^4dydt \biggl)^{1/4}
\leq C \lambda^{1/8+\epsilon_1}\delta^{-(1/2+\epsilon_2)}\|f\|_{L^4({\mathbb{R}}^2)},\hspace{0.2cm}\textmd{some }\hspace{0.2cm}\epsilon_1, \epsilon_2 >0,
\end{equation}
where
\begin{equation}\label{estimate:L4}
 \tilde{F}_{\lambda}^{\delta}f(y,t)=\rho_1(y,t)\int_{{\mathbb{R}}^2}e^{i(\xi \cdot y-t^2\xi_2\tilde{\Phi}(s,\delta))}a(\xi,t)\rho_0(\frac{\delta}{\lambda}|\xi|)\tilde{\chi}(\frac{\xi_1}{\xi_2})\hat{f}(\xi)d\xi.
\end{equation}
\end{thm}

\begin{rem}
If we replace the localized frequency $|\xi|\approx \lambda$ of Theorem 3.3 in \cite{Kungpaper} by $|\xi|\approx \frac{\lambda}{\delta}$, then we obtain the factor $\lambda^{1/8+\epsilon}2^{5k/8+\epsilon}$ which is larger than the factor $\lambda^{1/8+\epsilon_1}2^{k(1/2+\epsilon_2)}$ appeared in the above Theorem \ref{L^4therom}. This means that we get a better $L^4$-estimate for Fourier integral operators not satisfying the cinematic curvature uniformly.
\end{rem}

By (\ref{well}), we obtain
\begin{align*}
\|\widetilde{M_{j}^{k,1}}f\|_{L^4({\mathbb{R}}^2)}&\leq C(\frac{\delta^2}{\lambda})^{1/4}\lambda^{1/8+\epsilon_1}\delta^{-(1/2+\epsilon_2)}\|f\|_{L^4({\mathbb{R}}^2)}\\
&=C2^{j'(-1/8+\epsilon_1)}{2}^{k\epsilon_2}\|f\|_{L^4({\mathbb{R}}^2)}\\
&=C2^{j(-1/8+\epsilon_1)}2^{k(1/8-\epsilon_1+\epsilon_2)}\|f\|_{L^4({\mathbb{R}}^2)}.
\end{align*}

Employing the M. Riesz interpolation theorem and Lemma \ref{lemmaL^2}, we have for $2<p\leq 4$,
\begin{equation*}
\|\widetilde{M_{j}^{k,1}}f\|_{L^p({\mathbb{R}}^2)}
\leq C2^{j(-1/8+\epsilon_1)(2-4/p)}2^{k(3/4-5/(2p)+(2-4/p)(\epsilon_2-\epsilon_1))}\|f\|_{L^p({\mathbb{R}}^2)},
\end{equation*}
and
\begin{align*}
\sum_k 2^{-k}\sum_{j>2k}\|\widetilde{M_{j}^{k,1}}\|_{L^p\rightarrow L^p}\leq C\sum_k2^{-k(1/4+5/(2p)-(2-4/p)(\epsilon_2-\epsilon_1))}\sum_{j>2k}2^{j(-1/8+\epsilon_1)(2-4/p)}
\leq C_p.
\end{align*}

For $4\leq p<\infty$, the M. Riesz interpolation theorem and Lemma \ref{lemmaL^2} imply that
\begin{equation*}
\|\widetilde{M_{j}^{k,1}}f\|_{L^p({\mathbb{R}}^2)}
\leq C2^{j(-1/2+4\epsilon_1)/p}2^{4k(1/8-\epsilon_1+\epsilon_2)/p}\|f\|_{L^p({\mathbb{R}}^2)},
\end{equation*}
and
\begin{align*}
\sum_k 2^{-k}\sum_{j>2k}\|\widetilde{M_{j}^{k,1}}\|_{L^p\rightarrow L^p}\leq C\sum_k2^{k(1/(2p)-1+4(\epsilon_2-\epsilon_1)/p)}\sum_{j>2k}2^{j(-1/(2p)+4\epsilon_1/p)}\leq C_p.
\end{align*}
We have finished the proof of Theorem \ref{planetheorem}.

It remains to prove Theorem \ref{L^4therom}. Since $t\sim 1$, we will replace $t$ by $1/t$ in (\ref{estimate:L4}).  From (\ref{phase function}), the associated cone for the Fourier integral operator
 $\tilde{F}_{\lambda}^{\delta}$ which is localized to frequencies $|\xi|\sim \frac{\lambda}{\delta}$ is of the form $\{(\xi,\delta \phi'(0)\frac{\xi_1^3}{\xi_2^2}+\delta^2\frac{\partial}{\partial t}R(1/t,\xi,\delta))\}$,
where $R(1/t,\xi,\delta)$ is homogeneous of degree one in $\xi$. Since $t\sim 1$ and $\delta$ is sufficiently small, $\delta\frac{\partial}{\partial t}R(1/t,\xi,\delta)$ can be considered as a small perturbation of $\frac{\xi_1^3}{\xi_2^2}$ which is independent on $t$. If we just consider a translation-invariant  model operator which is defined by
\begin{equation}\label{model}
 \mathcal{F}_{\lambda}^{\delta}f(y,t)=\rho_1(y,t)\int_{{\mathbb{R}}^2}e^{i(\xi\cdot y+E(\xi)+t q_{\delta}(\xi))}a(\xi,t)\rho_0(\frac{\delta}{\lambda}|\xi|)\tilde{\chi}(\frac{\xi_1}{\xi_2})\hat{f}(\xi)d\xi,
\end{equation}
where
\begin{equation}
E(\xi)=\frac{\xi_1^2}{2\xi_2}\hspace{0.3cm} \textmd{and} \hspace{0.3cm} q_{\delta}(\xi)=\delta\cdot \frac{\xi_1^3}{\xi_2^2}\phi'(0),
\end{equation}
then the associated cone is $\{(\xi,\delta \phi'(0)\frac{\xi_1^3}{\xi_2^2})\}$. Comparing the light cone $\{(\xi,|\xi|)\}$, the level curve of the cone $\{(\xi,\delta \phi'(0)\frac{\xi_1^3}{\xi_2^2})\}$ become flatter and the cinematic curvature of $\mathcal{F}_{\lambda}^{\delta}$ is only greater than of equal to $\delta$. A very natural question is to ask how the $L^4$-estimate for a Fourier integral operator depends on the cinematic curvature. The simplest model is to study how the $L^4$-boundedness of the Fourier integral operator $\mathcal{F}_{\lambda}^{\delta}$ depends on $\delta$ and $\lambda$. In order to clarify the above questions, a direct idea is whether we can follow the proof of Proposition 3 in \cite{mss2}, which is mainly to get the $L^4$-boundedness of the Fourier integral operator associated to the light cone. In fact we need various modifications to overcome a lot of difficulties for $L^4$-boundedness of the Fourier integral operator $\mathcal{F}_{\lambda}^{\delta}$. However, finally we get the following theorem.
\begin{thm}\label{planetheorem2}
For all $\lambda >1/\delta$,
we have
\begin{equation}\label{Osci:L4}
\|\mathcal{F}_{\lambda}^{\delta}f(y,t)\|_{L^4({\mathbb{R}}^3)}\leq C \lambda^{1/8+\epsilon_1}\delta^{-(1/2+\epsilon_2)}\|f\|_{L^4({\mathbb{R}}^2)}.
\end{equation}
\end{thm}

Notice since $t\approx 1$, then
the model operator $\mathcal{F}_{\lambda}^{\delta}$ approximates the operator $\tilde{F}_{\lambda}^{\delta}$ whose phase function has of the form (\ref{phase function}). We give the proof of Theorem \ref{planetheorem2} in Section 3. In Section 4, the proof of Theorem \ref{L^4therom} will be shown for the phase function $-t^2\xi_2\tilde{\Phi}(s,\delta)$ in (\ref{phase function}).


\subsubsection[The case when $m>1$]{The case when $m>1$.}\label{vanishm>1}

Here we only show the differences with the case $m=1$ and  continue to follow the notation used in the previous section. We will mainly estimate
\begin{equation*}
\widehat{d\mu_{k,m}}(\delta_t\xi)=\int_{\mathbb{R}}e^{-i(t\xi_1x+t^2\xi_2x^2\phi(\delta x))}\tilde{\rho}(x)\eta(\delta x)dx,
\end{equation*}
where $\delta=2^{-k}$.

Set
\begin{equation}
s:=s(\xi,t)=-\frac{\xi_1}{t\xi_2}, \hspace{0.3cm}\textrm{for }\hspace{0.2cm} \xi_2\neq 0,
\end{equation}
and
\begin{equation}
\Phi(s,x,\delta)=-sx+x^2\phi(\frac{x}{2^k}).
\end{equation}

The implicit function theorem implies that for enough large $k$, there exists a smooth solution $\tilde{q}(s,\delta)$ for the equation $\partial_2\Phi(s,x,\delta)=0$. Meanwhile, if we choose $k$ sufficiently large and assume $\phi(0)=1/2$, then $\tilde{q}(s,\delta)$ smoothly converges to the solution $\tilde{q}(s,0)=s$ of the equation $\partial _2\Phi(s,x,0)=0$.

Let $\tilde{\Phi}(s,\delta):=\Phi(s,q(s,\delta),\delta)$. The phase function can be written as
\begin{equation*}
-t^2\xi_2\tilde{\Phi}(s,\delta)=\frac{\xi_1^2}{2\xi_2}+(-1)^{m+1}\frac{{\phi}^{(m)}(0)}{m!}\delta^{m}\frac{\xi_1^{m+2}}{t^m\xi_2^{m+1}} + R(t,\xi,\delta),
\end{equation*}
where $R(t,\xi,\delta)$ is homogeneous of degree one in $\xi$ and has at least m+1 power of $\delta$.

Using the similar argument as in the last section, it suffices to prove that
\begin{equation}\label{target2}
\sum_k2^{-k}\sum_{j\geq 1}\|\sup_{t\in[1,2]}|\widetilde{A_{t,j}^k}|\|_{L^p\rightarrow L^p}\leq C_p,
\end{equation}
where
\begin{equation*}
\widetilde{A_{t,j}^k}f(y):=\int_{\mathbb{R}^2}e^{i\xi\cdot y}\widehat{d\mu_{k,m}}(\delta_t\xi)\beta(2^{-j}|\delta_t\xi|)\hat{f}(\xi)d\xi.
\end{equation*}

A standard application of the method of stationary phase yields that
\begin{align}
\widehat{d\mu_{k,m}}(\delta_t\xi)=e^{-it^2\xi_2\tilde{\Phi}(s,\delta)}\chi_{k,m}(\frac{\xi_1}{t\xi_2})
\frac{A_{k,m}(\delta_t\xi)}{(1+|\delta_t\xi|)^{1/2}}+B_{k,m}(\delta_t\xi),
\end{align}
where  $\chi_{k,m}$ is a smooth function supported in $[c_{k,m},\widetilde{c_{k,m}}]$, for certain non-zero constant $c_{k,m}$ and $\widetilde{c_{k,m}}$ dependent only on $k$ and $m$. $A_{k,m}$ is a symbol of order zero in $\xi$ and $\{A_{k,m}(\delta_t\xi)\}_{k,m}$ is contained in a bounded subset of symbol of order zero. More precisely, for arbitrary $t \in [1,2]$,
 \begin{equation}\label{msymbol}
 |D_{\xi}^{\alpha}A_{k,m}(\delta_t\xi)|\leq C_{\alpha}(1+|\xi|)^{-\alpha},
\end{equation}
where $C_{\alpha}$ do not depend on $k$ and $m$. Furthermore, $B_{k,m}$ is a remaind term and satisfies for arbitrary $t\in [1,2]$,
\begin{equation}
 |D_{\xi}^{\alpha}B_{k,m}(\delta_t\xi)|\leq C_{\alpha,N}(1+|\xi|)^{-N},
\end{equation}
where $C_{\alpha,N}$ are admissible constants and again do not depend on $k$ and $m$.

Put
\begin{equation}
A^k_{t,j}f(y):=\int_{{\mathbb{R}}^2}e^{i(\xi \cdot y-t^2\xi_2\tilde{\Phi}(s,\delta))}\chi_{k,m}(\frac{\xi_1}{t\xi_2})
\frac{A_{k,m}(\delta_t\xi)}{(1+|\delta_t\xi|)^{1/2}}\beta(2^{-j}|\delta_t \xi|)\hat{f}(\xi)d\xi.
\end{equation}
A similar discussion as before allows us to choose $\rho_1(y,t)\in C_0^{\infty}(\mathbb{R}^2\times[1/2,4])$ and  to show that it is sufficient to prove that
\begin{equation}\label{target3}
\sum_k2^{-k}\sum_{j\geq 1}\|\widetilde{M_{j}^{k,1}}\|_{L^p\rightarrow L^p}\leq C_p,
\end{equation}
where
\begin{equation*}
\widetilde{M_{j}^{k,1}}f(y):=\sup_{t\in[1,2]}|\rho_1(y,t)A_{t,j}^kf(y)|,
\end{equation*}
( compare (\ref{local}), (\ref{local2}) ).

Moreover,
\begin{equation}\label{h_k_m}
\frac{\partial}{\partial t}(\rho_1(y,t)A^k_{t,j}f(y))=\int_{{\mathbb{R}}^2}e^{i(\xi \cdot y-t^2\xi_2\tilde{\Phi}(s,\delta))}h_{k,m}(y,t,\xi,j)\hat{f}(\xi)d\xi,
\end{equation}
where $|h_{k,m}(y,t,\xi,j)|\lesssim  2^{j/2}\delta^m+2^{-j/2}$, which follows from the similar argument with $h_k$ in (\ref{h_k}), together with the facts that
$|\frac{\partial}{\partial t}(-t^2\xi_2\tilde{\Phi}(s,\delta))|=|(-1)^{m}\frac{{\phi}^{(m)}(0)}{m!}\delta^{m}\frac{\xi_1^{m+2}}{t^{m+1}\xi_2^{m+1}} + \frac{\partial}{\partial t}R(t,\xi,\delta)|\lesssim \delta^m 2^{j}$.

So we still have the  following regularity estimates.
\begin{lem}\label{mlemmaL^2}
\begin{equation}\label{lemmmm}
\|\widetilde{M_{j}^{k,1}}f\|_{L^p({\mathbb{R}}^3)}\leq C_p 2^{-(j\wedge k)/p}\|f\|_{L^p({\mathbb{R}}^2)}, \hspace{0.5cm}2\leq p\leq \infty.
\end{equation}
\end{lem}

The proof is similar to the one of  Lemma \ref{lemmaL^2}. Hence, we have that
\begin{align*}
\sum_k2^{-k}\sum_{0<j\leq 9km}\|\widetilde{M_{j}^{k,1}}\|_{L^p\rightarrow L^p}
& \leq C_p +\sum_k2^{-k}\sum_{km<j\leq 9km}2^{-km/p}\\
&\leq C_p  +8m\sum_k k2^{-k(1+m/p)}\lesssim_p 1,
\end{align*}
where $2< p<\infty$.

Based on these results, let us now assume that $j>9km$. We introduce some notations. Let $j=km+j'$ and $\delta'=\delta^m$, then $\lambda=2^{j'}>{\delta'}^{-8}$ and $2^j=\frac{\lambda}{\delta'}$. After the same simplification as in the previous section, we will prove the following theorem:
\begin{thm}\label{F_lambda}
For $\lambda>{\delta'}^{-8}$, the following inequality
\begin{equation}\label{mestimate:L4}
\biggl(\int_{{\mathbb{R}}^2}\int^4_{1/2}|\tilde{F}_{\lambda}^{\delta'}f(y,t)|^4dydt \biggl)^{1/4}
\leq C \lambda^{1/8+\epsilon_1}{\delta'}^{-(1/2+\epsilon_2)}\|f\|_{L^4({\mathbb{R}}^2)}
\end{equation}
holds true for some $\epsilon_1$, $\epsilon_2>0$, where
\begin{equation*}
 \tilde{F}_{\lambda}^{\delta'}f(y,t)=\rho_1(y,t)\int_{{\mathbb{R}}^2}e^{i(\xi \cdot y-t^2\xi_2\tilde{\Phi}(s,\delta))}a(\xi,t)\rho_0(\frac{\delta'}{\lambda}|\xi|)\tilde{\chi}(\frac{\xi_1}{\xi_2})\hat{f}(\xi)d\xi.
\end{equation*}
\end{thm}

In order to use the proof of Chapter 4 to get (\ref{mestimate:L4}), we make the change of variable $t\rightarrow \tilde{t}^{-1/m}$, then it suffices to prove that
\begin{equation}\label{mestimate:L5}
\biggl(\int_{{\mathbb{R}}^2}\int^{2^m}_{2^{-2m}}|\tilde{F}_{\lambda}^{\delta'}f(y,\tilde{t}^{-1/m})|^4dyd\tilde{t} \biggl)^{1/4}\leq C \lambda^{1/8+\epsilon_1}{\delta'}^{-(1/2+\epsilon_2)}\|f\|_{L^4({\mathbb{R}}^2)},
\end{equation}
where
\begin{equation*}
\tilde{F}_{\lambda}^{\delta'}f(y,\tilde{t}^{-1/m})=\rho_1(y,\tilde{t}^{-1/m})\int_{{\mathbb{R}}^2}e^{i(\xi \cdot y-\tilde{t}^{-2/m}\xi_2\tilde{\Phi}(\tilde{s},\delta))}a(\xi,\tilde{t}^{-1/m})\rho_0(\frac{\delta'}{\lambda}|\xi|)
 \tilde{\chi}(\frac{\xi_1}{\xi_2})\hat{f}(\xi)d\xi,
\end{equation*}
and
 \begin{equation}
\tilde{s}=-\tilde{t}^{1/m}\frac{\xi_1}{\xi_2},\hspace{0.4cm}-\tilde{t}^{-2/m}\xi_2\tilde{\Phi}(\tilde{s},\delta)=\frac{\xi_1^2}{2\xi_2}+\tilde{t}\delta'(-1)^{m+1}\frac{{\phi}^{(m)}(0)}{m!}\cdot \frac{\xi_1^{m+2}}{\xi_2^{m+1}}+R(\tilde{t}^{-1/m},\xi,\delta).
\end{equation}
 Hence, we employ the similar idea of Chapter 4 to obtain (\ref{mestimate:L5}),  then Theorem \ref{F_lambda} is proved.

Based on the above theorem, we get that
\begin{equation}\label{lemmm2}
\begin{aligned}
\|\widetilde{M_{j}^{k,1}}f\|_{L^4}
&\leq C(\frac{\delta'^2}{\lambda})^{1/4}\lambda^{1/8+\epsilon_1}\delta'^{-(1/2+\epsilon_2)}\|f\|_{L^4}\\
&=C2^{j'(-1/8+\epsilon_1)}2^{km\epsilon_2}\|f\|_{L^4}\\
&=C2^{j(-1/8+\epsilon_1)}2^{km(1/8+\epsilon_2-\epsilon_1)}\|f\|_{L^4}.
\end{aligned}
\end{equation}

Employing the M. Riesz interpolation theorem between (\ref{lemmmm}) for $p=2$ and (\ref{lemmm2}),  we have for $2<p\leq 4$,
\begin{equation*}
 \|\widetilde{M_{j}^{k,1}}f\|_{L^p({\mathbb{R}}^3)}
\leq C2^{-j(1/8-\epsilon_1)(2-4/p)}2^{km[3/4-5/(2p)+(2-4/p)(\epsilon_2-\epsilon_1)]}\|f\|_{L^p({\mathbb{R}}^2)}.
\end{equation*}

If $3/4-5/(2p)+(2-4/p)(\epsilon_2-\epsilon_1)\leq 0$, then (\ref{target3}) will follow directly for $j>9km$. If $3/4-5/(2p)+(2-4/p)(\epsilon_2-\epsilon_1)>0$, then $km<j/9$ yields that
\begin{align*}
\|\widetilde{M_{j}^{k,1}}f\|_{L^p({\mathbb{R}}^3)}
&\leq C2^{-j(1/8-\epsilon_1)(2-4/p)}2^{j[3/4-5/(2p)+(2-4/p)(\epsilon_2-\epsilon_1)]/9}\|f\|_{L^p({\mathbb{R}}^2)}\\
&=C2^{j[-1/6+2/(9p)+(2-4/p)(8\epsilon_1+\epsilon_2)/9]}\|f\|_{L^p({\mathbb{R}}^2)}.
\end{align*}

Employing the M. Riesz interpolation theorem (\ref{lemmmm}) for $p=\infty$ and (\ref{lemmm2}), we have for $4\leq p<\infty$,
\begin{equation*}
 \|\widetilde{M_{j}^{k,1}}f\|_{L^p({\mathbb{R}}^3)}
\leq C2^{4j(-1/8+\epsilon_1)/p}2^{4km(1/8+(\epsilon_2-\epsilon_1))/p}\|f\|_{L^p({\mathbb{R}}^2)}.
\end{equation*}
Since $1/8+(\epsilon_2-\epsilon_1)>0$, then $km<j/9$ implies that
\begin{align*}
 \|\widetilde{M_{j}^{k,1}}f\|_{L^p({\mathbb{R}}^3)}
&\leq C2^{4j(-1/8+\epsilon_1)/p}2^{4j(1/8+(\epsilon_2-\epsilon_1))/(9p)}\|f\|_{L^p({\mathbb{R}}^2)}
\\
& = C2^{4j(-1+8\epsilon_1+\epsilon_2)/(9p)}\|f\|_{L^p({\mathbb{R}}^2)},
\end{align*}
which finishes the proof of (\ref{target3}), hence of Theorem \ref{planetheorem} when $m>1$.


\subsection[The proof for curves of finite type]{The proof for curves of finite type.} \label{vanishfinite}
In this section, we will prove Theorem \ref{planetheoremfinite}.

The proof of Theorem \ref{planetheorem} inspires us to prove Theorem \ref{planetheoremfinite} directly. Because at the beginning of the proof, we employ a dyadic decomposition to restrict on an interval far away from the origin, and this means whether $d=2$, the curve has still non-vanishing Gaussian curvature in this constant interval and we can still apply the method of stationary phase. We proceed as in the proof of Theorem \ref{planetheorem}.

We choose $\tilde{\rho}\in C_0^{\infty}(\mathbb{R})$ such that supp $\tilde{\rho}\subset\{x:B/2\leq|x|\leq 2B\}$  and $\sum_k\tilde{\rho}(2^kx)=1$. Since the support of $\eta$ is sufficiently small, then we can choose $k$ sufficiently large.

Put
\begin{equation*}
A_tf(y):=\int_{\mathbb{R}}f(y_1-tx,y_2-t^dx^d\phi(x))\eta(x)dx=\sum_k\widetilde{A_t^k}f(y),
\end{equation*}
where $\widetilde{A_t^k}f(y):=\int_{\mathbb{R}}f(y_1-tx,y_2-t^dx^d\phi(x))\eta(x)\tilde{\rho}(2^kx)dx$.

Consider the isometric operator on $L^p(\mathbb{R}^2)$ defined by
\begin{equation*}
Tf(x_1,x_2)=2^{(d+1)k/p}f(2^kx_1,2^{dk}x_2).
\end{equation*}
As the argument in the last section,  it suffices to prove the following estimate
\begin{equation*}
\sum_k2^{-k}\|\sup_{t>0}|A_t^k|\|_{L^p}\leq C_p, \hspace{0.2cm}\textmd{for}\hspace{0.2cm} p>2,
\end{equation*}
where $A_t^kf(y):=\int_{\mathbb{R}}f(y_1-tx,y_2-t^dx^d\phi(\frac{x}{2^k}))\tilde{\rho}(x)\eta(2^{-k}x)dx$.

By means of the Fourier inversion formula, we can write
\begin{equation*}
 A_t^kf(y)=\frac{1}{(2\pi)^2}\int_{{\mathbb{R}}^2}e^{i\xi\cdot
  y}\widehat{d\mu_{k,d,m}}(\delta_t\xi)\hat{f}(\xi)d\xi,
\end{equation*}
where
\begin{equation*}
 \widehat{d\mu_{k,d,m}}(\delta_t\xi)=\int_{\mathbb{R}}e^{-i(t\xi_1x+t^d\xi_2x^d\phi(\frac{x}{2^k}))}\tilde{\rho}(x)\eta(2^{-k}x)dx.
\end{equation*}

Choosing a non-negative function $\beta\in C_0^{\infty}(\mathbb{R})$ as before,
set
\begin{equation*}A_{t,j}^kf(y):=\frac{1}{(2\pi)^2}\int_{{\mathbb{R}}^2}e^{i\xi\cdot
  y}\widehat{d\mu_{k,d,m}}(\delta_t\xi)\beta(2^{-j}|\delta_t\xi|)\hat{f}(\xi)d\xi
\end{equation*}
and denote by $M_j^k$ the corresponding maximal operator.

Since $A_{t,j}^k$ is localized to frequencies $|\delta_t\xi|\approx 2^{j}$, we can still use Lemma \ref{Appendix}
to prove that
\begin{equation*}
\|M_j^k\|_{L^p\rightarrow L^p}\lesssim \|M_{j,loc}^k\|_{L^p\rightarrow
L^p},
\end{equation*}
where $M_{j,loc}^kf(y):=\sup_{t\in[1,2]}|A_{t,j}^kf(y)|$.

Set
\begin{equation}
\delta:=2^{-k}, \hspace{0.2cm} s:=s(\xi,t)=-\frac{\xi_1}{t^{d-1}\xi_2}, \hspace{0.3cm}\textrm{for} \hspace{0.2cm} \xi_2\neq 0,
\end{equation}
and
\begin{equation}
\Phi(s,x,\delta)=-sx+x^d\phi(\delta x).
\end{equation}

From the proof of Theorem \ref{planetheorem}, for fixed $t\in[1,2]$, we would mainly estimate
 \begin{equation*}
A^k_{t,j}f(y):=\int_{{\mathbb{R}}^2}e^{i(\xi\cdot y-t^d\xi_2\tilde{\Phi}(s,\delta))}\chi_{k,d,m}(\frac{\xi_1}{t^{d-1}\xi_2})
\frac{A_{k,d,m}(\delta_t\xi)}{(1+|\delta_t\xi|)^{1/2}}
\beta(2^{-j}|\delta_t\xi|)\hat{f}(\xi)d\xi,
\end{equation*}
where $\chi_{k,d,m}$ is a smooth function supported in the conical region $[c_{k,d,m},\widetilde{c_{k,d,m}}]$, for certain non-zero constant $c_{k,d,m}$ and $\widetilde{c_{k,d,m}}$ dependent only on $k, m$ and $d$. $A_{k,d,m}$ is a symbol of order zero in $\xi$ and $\{A_{k,d,m}(\delta_t\xi)\}_k$ is contained in a bounded subset of symbol of order zero. Let $\phi(0)=1/d$ and  $\tilde{\Phi}(s,\delta):=\Phi(s,\tilde{q}(s,\delta),\delta)$, then the phase function can be written as
\begin{align*}
-t^d\xi_2\tilde{\Phi}(s,\delta)=\biggl(\frac{1}{d^{\frac{d}{d-1}}}-\frac{1}{d^{\frac{1}{d-1}}}\biggl)(-\frac{d\xi_1^d}{\xi_2})^{\frac{1}{d-1}}-\frac{\delta^m\phi^{(m)}(0)}{t^mm!}
\biggl(-\frac{\xi_1}{\xi_2^{\frac{m+1}{m+d}}}\biggl)^{\frac{d+m}{d-1}}+ R(t,\xi,\delta,d),
\end{align*}
where $R(t,\xi,\delta,d)$ is homogeneous of degree one in $\xi$ and has at least $m+1$ power of $\delta$.

The similar argument with (\ref{local}) allows us to choose $\rho_1(y,t)\in C_0^{\infty}(\mathbb{R}^2\times[1/2,4])$ and  it is sufficient to prove that
\begin{equation}
\sum_k2^{-k}\sum_{j\geq 1}\|\widetilde{M_{j}^{k,1}}\|_{L^p\rightarrow L^p}\leq C_p,
\end{equation}
where
\begin{equation*}
\widetilde{M_{j}^{k,1}}f(y):=\sup_{t\in[1,2]}|\rho_1(y,t)A_{t,j}^kf(y)|.
\end{equation*}

Moreover,
\begin{equation}\label{h_k_m}
\frac{\partial}{\partial t}(\rho_1(y,t)A^k_{t,j}f(y))=\int_{{\mathbb{R}}^2}e^{i(\xi \cdot y-t^d\xi_2\tilde{\Phi}(s,\delta))}h_{k,m,d}(y,t,\xi,j)\hat{f}(\xi)d\xi,
\end{equation}
where $|h_{k,m}(y,t,\xi,j)|\lesssim  2^{j/2}\delta^m+2^{-j/2}$, which follows from the similar argument with $h_k$ in (\ref{h_k}), together with the facts that
$$|\frac{\partial}{\partial t}(-t^2\xi_2\tilde{\Phi}(s,\delta))|=\left|m\frac{\delta^m\phi^{(m)}(0)}{t^{m+1}m!}
\biggl(-\frac{\xi_1}{\xi_2^{\frac{m+1}{m+d}}}\biggl)^{\frac{d+m}{d-1}}+\frac{\partial}{\partial t}R(t,\xi,\delta,d)\right|\lesssim \delta^m 2^{j}.$$

So we still have the same regularity estimates as (\ref{mlemmaL^2}) and
\begin{equation*}
\sum_k2^{-k}\sum_{0<j\leq 9km}\|\widetilde{M_{j}^{k,1}}\|_{L^p\rightarrow L^p}\leq C_p, \hspace{0.2cm}\textmd{for}\hspace{0.2cm} 2< p<\infty.
\end{equation*}

Based on these arguments, for $j>9km$, let $j=km+j'$ and $\delta'=\delta^m$, then $j'>8km$ and $\lambda=2^{j'}>{\delta'}^{-8}$. We would be done if we could prove the following theorem.

\begin{thm}\label{123}
\begin{equation}
\biggl(\int_{{\mathbb{R}}^2}\int^{2^m}_{2^{-2m}}|\tilde{F}_{\lambda}^{\delta'}f(y,\tilde{t}^{-1/m})|^4dyd\tilde{t} \biggl)^{1/4}\leq C \lambda^{1/8+\epsilon_1}{\delta'}^{-(1/2+\epsilon_2)}\|f\|_{L^4({\mathbb{R}}^2)},
\end{equation}
where
\begin{equation*}
\tilde{F}_{\lambda}^{\delta'}f(y,\tilde{t}^{-1/m})=\rho_1(y,\tilde{t}^{-1/m})\int_{{\mathbb{R}}^2}e^{i(\xi\cdot y-\tilde{t}^{-d/m}\xi_2\tilde{\Phi}(\tilde{s},\delta))}a(\xi,\tilde{t}^{-1/m})\rho_0(\frac{\delta'}{\lambda}|\xi|)\tilde{\chi}(\frac{\xi_1}{\xi_2})\hat{f}(\xi)d\xi,
\end{equation*}
and
 \begin{equation*}
\tilde{s}=-\tilde{t}^{(d-1)/m}\frac{\xi_1}{\xi_2},
\end{equation*}
\begin{equation}
-\tilde{t}^{-d/m}\xi_2\tilde{\Phi}(\tilde{s},\delta)=\biggl(\frac{1}{d^{\frac{d}{d-1}}}-\frac{1}{d^{\frac{1}{d-1}}}\biggl)(-\frac{d\xi_1^d}{\xi_2})^{\frac{1}{d-1}}
-\tilde{t}\delta'\frac{\phi^{(m)}(0)}{m!}
\biggl(-\frac{\xi_1}{\xi_2^{\frac{m+1}{m+d}}}\biggl)^{\frac{d+m}{d-1}}+R(\tilde{t}^{-1/m},\xi,\delta,d).
\end{equation}
\end{thm}

The proof is similar to the one in Section 4.

\section[A translation-invariant model Fourier integral operator]{A translation-invariant model Fourier integral operator.}

In this section, we will prove Theorem \ref{planetheorem2}, which will be obtained in a similar way as the proof of $L^4$-boundedness of the following operator:
\begin{equation}
 \mathcal{F}_{\lambda}^{\delta}f(y,t)=\rho_1(y,t)\int_{{\mathbb{R}}^2}e^{i(y\cdot\xi+E(\xi)+t q_{\delta}(\xi))}a(\xi,t)\rho_0(\frac{\delta}{\lambda}|\xi|)\tilde{\chi}(\frac{\xi_1}{\xi_2})\hat{f}(\xi)d\xi,
\end{equation}
where $E(\xi)=\frac{\xi_1^2}{\xi_2}$, $q_{\delta}(\xi)=\delta\cdot \frac{\xi_1^3}{\xi_2^2}$, $\tilde{\chi}\in C_0^{\infty}([c_1,c_2])$ ($c_1$, $c_2$ are very small positive constants), $\rho_0\in C_0^{\infty}((-10,10))$ and $a$ is a symbol of order zero. So we need to prove that for all $\lambda >1/\delta$,
\begin{equation}
\|\mathcal{F}_{\lambda}^{\delta}f(y,t)\|_{L^4({\mathbb{R}}^3)}\leq C \lambda^{1/8+\epsilon_1}\delta^{-(1/2+\epsilon_2)}\|f\|_{L^4({\mathbb{R}}^2)}.
\end{equation}

The following approach follows the proof of Proposition 3 in \cite{mss2}, but we need various modifications.

We may assume that the $\xi$-support of the symbol $a$ is in the first quadrant. We rewrite
\begin{align*}
\mathcal{F}_{\lambda}^{\delta}f(y,t)&=\int_{\mathbb{R}}e^{it\tau} \mathcal{F}_{\lambda}^{\delta}f(y,\hat{\tau})d\tau\\
&=\int_{\mathbb{R}}\int_{{\mathbb{R}}^2}e^{i(y\cdot \xi+t\tau)}e^{iE(\xi)}(\rho_1(y,\cdot)a(\xi,\cdot))^{\wedge}(\tau-q_{\delta}(\xi))
\rho_0(\frac{\delta}{\lambda}|\xi|)\tilde{\chi}(\frac{\xi_1}{\xi_2})\hat{f}(\xi)d\xi d\tau.
\end{align*}

We can reduce our proof to the situation $|\tau-q_{\delta}(\xi)|\leq C_0\lambda$ for an appropriate constant $C_0$ which is small and  satisfies $q_{\delta}(\xi)-2C_0\lambda>0$. In fact, if  $|\tau-q_{\delta}(\xi)|\geq C_0\lambda$, then

$(i)$ if $\tau\gg q_{\delta}(\xi)\approx \lambda >1$, by $|\xi|\approx \lambda/\delta$, then
$$|\tau-q_{\delta}(\xi)|\geq C\tau\geq C (\lambda+|(\delta \xi,\tau)|);$$

$(ii)$ if $\tau \lesssim q_{\delta}(\xi)\approx \lambda >1$, by assumption, then
$$|\tau-q_{\delta}(\xi)|\geq C_0\lambda \geq C(\lambda+|(\delta \xi,\tau)|).$$

Define the  operator $\mathcal{L}=I-\triangle_{\xi,\tau}$, where $\triangle_{\xi,\tau}=\frac{{\partial}^2}{{\partial \xi_1^2}}+\frac{{\partial}^2}{{\partial \xi_2^2}}+\frac{{\partial}^2}{{\partial \tau^2}}$, then the above arguments of $(i)$ and $(ii)$, $\rho_1 \in C_0^{\infty}({\mathbb{R}}^2\times [1/2,4])$, together with the fact that  $a$ is a symbol of order zero, imply that for sufficiently large $N'$,
\begin{equation}
\left|\mathcal{L}^2\left(\rho_1(y,\cdot)a(\cdot,\xi))^{\wedge}(\tau-q_{\delta}(\xi)\right)\right|\leq C_{N'}(\lambda+|(\delta \xi,\tau)|)^{-{N'}}
\end{equation}
and
\begin{align*}
&\left|{\mathcal{L}}^2\left(e^{iE(\xi)}(\rho_1(y,\cdot)a(\cdot,\xi))^{\wedge}(\tau-q_{\delta}(\xi))
\rho_0(\frac{\delta}{\lambda}|\xi|)\tilde{\chi}(\frac{\xi_1}{\xi_2})\right)\right|\\
&\leq C_{N'}(\lambda+|(\delta \xi,\tau)|)^{-N'}\\
& \leq C_{N'}\lambda^{-N'/2}(1+|(\delta \xi,\tau)|)^{-N'/2}.
\end{align*}

Since ${\mathcal{L}}^2(e^{i((y-x)\cdot \xi+t\tau)})=(1+|(y-x,t)|^2)^2$, by integration by parts in $(\xi,\tau)$ and the assumption $|\tau-q_{\delta}(\xi)|\geq C_0\lambda$, the kernel $K(y,t;x)$ of the operator $\mathcal{F}_{\lambda}^{\delta}$ can be controlled by
\begin{align*}
|K(y,t;x)|&\leq C_{N'} \frac{1}{(1+|(y-x,t)|^2)^2}{\lambda}^{-N'/2}\int_{{\mathbb{R}}^3}(1+|(\delta \xi,\tau)|)^{-N'/2}d\xi d\tau\\
&\leq C_{N'} \frac{1}{(1+|(y-x,t)|^2)^2}\lambda^{-N'/2}{\delta}^{-2}.
\end{align*}

Since $\delta \lambda >1$, then for sufficiently large $N<N'$,
\begin{equation}\label{kernel}
|K(y,t;x)|\leq C_N \lambda^{-N}\frac{1}{(1+|(y-x,t)|^2)^2}.
\end{equation}

In the next steps, we will always assume that $|\tau-q_{\delta}(\xi)|\leq C_0\lambda$ for a small appropriate constant $C_0$, which shows that
\begin{equation}\label{asuumption}
\lambda\approx q_{\delta}(\xi)-C_0\lambda\leq \tau\leq q_{\delta}(\xi)+C_0\lambda\approx \lambda.
\end{equation}

Let $\varphi \in C_0^{\infty}(\mathbb{R})$ satisfy supp $\varphi \subset [-1,1]$ and  $\sum_{n\in \mathbb{Z}}\varphi^2(\cdot-n)=1$. In order to give a decomposition for $\tau$, the dual to $t$, we define the operator $P^n$ on functions in ${\mathbb{R}}^3$ by
\begin{equation}
(P^ng)^{\wedge}(\eta,\tau)=\varphi(\lambda^{-1/2}\tau-n)\widehat{g}(\eta,\tau).
\end{equation}

Similarly, we define $\widehat{f_n}(\xi)=\varphi((\lambda^{-1/2}q_{\delta}(\xi)-n)/10)\hat{f}(\xi)$.

\begin{lem}\label{planelemma1}
For $N>0$ and $1\leq p\leq \infty$,
\begin{equation}
\biggl\|\sum_n(P^n)^2\mathcal{F}_{\lambda}^{\delta}(f-f_n)\biggl\|_{L^p({\mathbb{R}}^3)}\leq C_N\lambda^{-N}\|f\|_{L^p({\mathbb{R}}^2)}.
\end{equation}
\end{lem}
\begin{proof}
For fixed $n$, we have
\begin{align*}
(P^n)^2\mathcal{F}_{\lambda}^{\delta}(f-f_n)(y,t)
&=\int_{\mathbb{R}}\int_{{\mathbb{R}}^2}e^{i(y\cdot \xi+t\tau+E(\xi))}\varphi^2(\lambda^{-1/2}\tau-n)(\rho_1(y,\cdot)a(\xi,\cdot))^{\wedge}(\tau-q_{\delta}(\xi))
\\
& \quad \times\rho_0(\frac{\delta}{\lambda}|\xi|)\tilde{\chi}(\frac{\xi_1}{\xi_2})[1-\varphi((\lambda^{-1/2}q_{\delta}(\xi)-n)/10)]\hat{f}(\xi)d\xi d\tau.
\end{align*}

From the support of $\varphi$, we have $\lambda^{1/2}(n-1)\leq \tau\leq \lambda^{1/2}(n+1)$, together with (\ref{asuumption}), we get $n\approx \lambda^{1/2}$.
We observe that $q_{\delta}(\xi)\not\in[\lambda^{1/2}(n-10),\lambda^{1/2}(n+10)]$, then we get   $|\tau-q_{\delta}(\xi)|\gtrsim (\lambda+|(\delta\xi,\tau)|)^{1/2}$.
Finally, as in the estimate (\ref{kernel}), the kernel of $(P^n)^2\mathcal{F}_{\lambda}^{\delta}$ defined by
\begin{align*}
\int_{\mathbb{R}}\int_{{\mathbb{R}}^2}&e^{i[(y-x)\cdot \xi+t\tau+E(\xi)]}\varphi^2(\lambda^{-1/2}\tau-n)(\rho_1(y,\cdot)a(\xi,\cdot))^{\wedge}(\tau-q_{\delta}(\xi))
\rho_0(\frac{\delta}{\lambda}|\xi|)\tilde{\chi}(\frac{\xi_1}{\xi_2})\\
&\quad \times [1-\varphi((\lambda^{-1/2}q_{\delta}(\xi)-n)/10)]d\xi d\tau
\end{align*}
can be dominated by $\mathcal{O}(\lambda^{-N}(1+|(y-x,t)|^2)^{-2})$ and
\begin{equation*}
\biggl\|\sum_n(P^n)^2\mathcal{F}_{\lambda}^{\delta}(f-f_n)\biggl\|_{L^p({\mathbb{R}}^3)}\leq C_{N'}\sum_{n\approx \lambda^{1/2}}\lambda^{-N'}\|f\|_{L^p({\mathbb{R}}^2)}\leq C_N \lambda^{-N}\|f\|_{L^p({\mathbb{R}}^2)}.
\end{equation*}

\end{proof}


\begin{lem}\label{planelemma2}
For $2\leq p\leq \infty$, we have
\begin{equation*}
\biggl\|\sum_{n\approx \lambda^{1/2}}(P^n)^2\mathcal{F}_{\lambda}^{\delta}f_n\biggl\|_{L^p({\mathbb{R}}^3)}\leq C\lambda^{(1/4-1/(2p))}\biggl\|\left(\sum_{n\approx \lambda^{1/2}}|P^n\mathcal{F}_{\lambda}^{\delta}f_n|^2\right)^{1/2}\biggl\|_{L^p({\mathbb{R}}^3)}.
\end{equation*}
\end{lem}
The idea can be found in \cite{mss2}.


So far we have used a radial decomposition of the Fourier integral operator $\mathcal{F}_{\lambda}^{\delta}$ with respect to frequency variables. Now we make a further angular decomposition. As we have introduced in Lemma \ref{angularde}, here we redefine the homogeneous partition of unity of ${\mathbb{R}}^2\setminus \{0\}$ that depends on the scale $\lambda/\delta$. Specially, for fixed $j>0$, we choose functions $\chi_{\nu}, \nu=0,1,2,\cdots$, and the size of every angle is $(\delta/\lambda)^{1/2}$. Meanwhile $\chi_{\nu}$ satisfies the following conditions:
\begin{equation}\label{I}
\sum_{\nu}\chi_{\nu}(\xi)=1 \hspace{0.2cm}\textmd{for} \hspace{0.2cm}\textmd{all }\hspace{0.2cm}\xi \neq 0,
\end{equation}
\begin{equation}\label{II}
|\partial_{\xi}^{\alpha}\chi_{\nu}|\leq A_{\alpha} (\lambda/\delta)^{|\alpha|/2}|\xi|^{-|\alpha|} \hspace{0.2cm}\textmd{for} \hspace{0.2cm}\textmd{all }\hspace{0.2cm}\xi \neq 0.
\end{equation}

Define $Q_{\nu}$ by
\begin{equation*}
(Q_{\nu}g)^{\wedge}(\eta,\tau)=\Psi_{\nu}^{\delta}(\eta,\tau)\widehat{g}(\eta,\tau)=\psi\biggl(\frac{\tau-q_{\delta}(\eta)}{|\tau/{\delta}|^{\epsilon}}\biggl)\chi_{\nu}(\eta')\widehat{g}(\eta,\tau),
\end{equation*}
where $\eta'=\frac{\eta}{|\eta|}$ and  $\psi\in C_0^{\infty}(\mathbb{R})$ is supported in $[-2,2]$ and equal to $1$ in $[-1,1]$. In fact, $(Q_{\nu}g)^{\wedge}(\eta,\tau)$ is supported in a thin sector intersected with a thin neighborhood of the cone $\{(\eta,q_{\delta}(\eta))\}$, see Figure 2.

Put
\begin{equation*}
\mathcal{F}_{\lambda,\nu}^{\delta}f(y,t)=\rho_1(y,t)\int_{{\mathbb{R}}^2}e^{i(y\cdot \xi+E(\xi)+tq_{\delta}(\xi))}
a(\xi,t)\tilde{\chi}(\frac{\xi_1}{\xi_2})\rho_0(\frac{\delta}{\lambda}|\xi|)\chi_{\nu}(\xi)\widehat{f}(\xi)d\xi.
\end{equation*}

\begin{lem}\label{planelemma3}
Given a fixed $\epsilon>0$ and $n\approx \lambda^{1/2}$,  we define
\begin{equation}
R_{\lambda,\nu}^{\delta}f=P^n(\mathcal{F}_{\lambda,\nu}^{\delta}f)-P^nQ_{\nu}(\mathcal{F}_{\lambda,\nu}^{\delta}f).
\end{equation}
For any $N>0$, there is a uniform constant $C_N$ so that
\begin{equation}
\|R_{\lambda,\nu}^{\delta}f_n\|_{L^p({\mathbb{R}}^3)}\leq C_N(\frac{\lambda}{\delta})^{-N}\|f\|_{L^p({\mathbb{R}}^2)}.
\end{equation}
\end{lem}

\begin{proof}
\begin{align*}
R_{\lambda,\nu}^{\delta}f_n(y,t)&=\int_{{\mathbb{R}}^2}e^{iE(\xi)}\biggl[\int_{{\mathbb{R}}^2}\int_{\mathbb{R}}e^{i(y\cdot \eta+t\tau)}
(1-\Psi_{\nu}^{\delta}(\eta,\tau))\varphi(\lambda^{-1/2}\tau-n)\\
& \quad \times
(\rho_1\cdot a(\xi,\cdot))^{\wedge}(\eta-\xi,\tau-q_{\delta}(\xi))d\eta d\tau \biggl]
\rho_0(\frac{\delta}{\lambda}|\xi|)\tilde{\chi}(\frac{\xi_1}{\xi_2})\chi_{\nu}(\xi)\widehat{f_n}(\xi)d\xi.
\end{align*}

In order to estimate $|(\eta-\xi,\tau-q_{\delta}(\xi))|$, we split our consideration into different cases:

1. When $\eta' \not \in \hspace{0.1cm}\textrm{supp }\hspace{0.1cm} \chi_{\nu}$, but $\xi\in \hspace{0.1cm}\textrm{supp} \hspace{0.1cm}\chi_{\nu}$, then $|\eta-\xi|\geq C(\frac{\lambda}{\delta})^{1/2}$ and the following inequalities will hold true:

(1). if $|\eta|\gg|\xi|\approx \frac{\lambda}{\delta}$,
\begin{equation*}
|\eta-\xi|\geq C|\eta|\geq C(\frac{\lambda}{\delta}+|(\eta, \frac{\tau}{\delta})|),
\end{equation*}

(2). if $|\eta|\lesssim|\xi|\approx \frac{\lambda}{\delta}$,
\begin{equation*}
|\eta-\xi|\geq C(\frac{\lambda}{\delta})^{1/2}\geq C(\frac{\lambda}{\delta}+|(\eta, \frac{\tau}{\delta})|)^{1/2}.
\end{equation*}

2. When $\eta' \in \hspace{0.1cm}\textrm{supp} \hspace{0.1cm} \chi_{\nu}$, but $\frac{|\tau-q_{\delta}(\eta)|}{|\tau/\delta|^{\epsilon}}\geq 2$ and $\xi\in \hspace{0.1cm}\textrm{supp} \hspace{0.1cm} \chi_{\nu}$, then  the following inequalities will hold true:

(1). if $|\eta|\gg|\xi|\approx \frac{\lambda}{\delta}$,
\begin{equation*}
|\eta-\xi|\geq C|\eta|\geq C(\frac{\lambda}{\delta}+|(\eta, \frac{\tau}{\delta})|),
\end{equation*}

(2). if $|\eta|\lesssim|\xi|\approx \frac{\lambda}{\delta}$, for $\mu \ll \epsilon$, we consider two subcases:

a. when $|\eta-\xi|\geq C|\tau/\delta|^{\mu}$, we have
\begin{equation*}
|\eta-\xi|\geq C(\frac{\lambda}{\delta}+|(\eta, \frac{\tau}{\delta})|)^{\mu},
\end{equation*}

b. when $|\eta-\xi|\leq C|\tau/\delta|^{\mu}$, since
\begin{equation*}
|\tau-q_{\delta}(\xi)|\geq |\tau-q_{\delta}(\eta)|-|q_{\delta}(\xi)-q_{\delta}(\eta)|
\end{equation*}
and
\begin{equation*}
|q_{\delta}(\xi)-q_{\delta}(\eta)|\approx \delta|\xi-\eta|\leq C\delta |\frac{\tau}{\delta}|^{\mu},
\end{equation*}
then we have
\begin{equation*}
|\tau-q_{\delta}(\xi)|\geq 2|\frac{\tau}{\delta}|^{\epsilon}-C\delta|\frac{\tau}{\delta}|^{\mu}\approx (\frac{\lambda}{\delta}+|(\eta, \frac{\tau}{\delta})|)^{\epsilon}.
\end{equation*}

Combing all the cases, if we set $\mathcal{L}=I-\triangle_{\eta,\tau}$, for any $2 <M\ll \mu N'$, we obtain,
\begin{equation}\label{equ1:planelemma2}
\left|\mathcal{L}^M\left(\widehat{\rho_1}(\eta-\xi,\tau-q_{\delta}(\xi))\right)\right|\leq C_{N'} (\frac{\lambda}{\delta}+|(\eta, \frac{\tau}{\delta})|)^{-\mu N'}.
\end{equation}

The estimate (\ref{II}) implies that $|\chi_{\nu}(\eta')|\leq (\lambda/\delta)^{1/2}$, in addition to (\ref{equ1:planelemma2}), we get
\begin{equation}
\mathcal{L}^M\biggl((1-\Psi_{\nu}^{\delta}(\eta,\tau))\varphi(\lambda^{-1/2}\tau-n)
\widehat{\rho_1}(\eta-\xi,\tau-q_{\delta}(\xi))\biggl)\leq C_{M,N'} (\frac{\lambda}{\delta})^{M/2}(\frac{\lambda}{\delta}+|(\eta, \frac{\tau}{\delta})|)^{-\mu N'}.
\end{equation}

By integration by parts for $(\eta,\tau)$, we have
\begin{align*}
&|R_{\lambda,\nu}^{\delta}f_n(y,t)|\\
&\leq \int_{{\mathbb{R}}^2}\left|\int_{{\mathbb{R}}^2}\int_{\mathbb{R}}e^{i(y\cdot \eta+t\tau)}
(1-\Psi_{\nu}^{\delta}(\eta,\tau))\varphi(\lambda^{-1/2}\tau-n)
(\rho_1\cdot a(\xi,\cdot))^{\wedge}(\eta-\xi,\tau-q_{\delta}(\xi))d\eta d\tau\right|\\
& \quad \times
\left|\rho_0(\frac{\delta}{\lambda}|\xi|)\tilde{\chi}(\frac{\xi_1}{\xi_2})\chi_{\nu}(\xi)\widehat{f_n}(\xi)\right|d\xi\\
&\leq \int_{{\mathbb{R}}^2}\int_{{\mathbb{R}}^2}\int_{\mathbb{R}}\left|{\mathcal{L}}^M\biggl(
(1-\Psi_{\nu}^{\delta}(\eta,\tau))
\varphi(\lambda^{-1/2}\tau-n)(\rho_1\cdot a(\xi,\cdot))^{\wedge}(\eta-\xi,\tau-q_{\delta}(\xi))\biggl)\right|d\eta d\tau\\
& \quad \times \frac{1}{(1+|(y,t)|^2)^M}
\left|\rho_0(\frac{\delta}{\lambda}|\xi|)\tilde{\chi}(\frac{\xi_1}{\xi_2})\chi_{\nu}(\xi)\widehat{f_n}(\xi)\right|d\xi\\
&\leq \frac{C_{M,N'}}{(1+|(y,t)|^2)^M}\int_{{\mathbb{R}}^2}\int_{{\mathbb{R}}^2}\int_{\mathbb{R}}(\frac{\lambda}{\delta})^{M}(\frac{\lambda}{\delta}+|(\eta, \frac{\tau}{\delta})|)^{-\mu N'}d\eta d\tau\left|\rho_0(\frac{\delta}{\lambda}|\xi|)
\tilde{\chi}(\frac{\xi_1}{\xi_2})\chi_{\nu}(\xi)\widehat{f_n}(\xi)\right|d\xi\\
&\leq C_{M,N'} (\frac{\lambda}{\delta})^{M-\mu N'/2}\frac{\delta }{(1+|(y,t)|^2)^M}\int_{{\mathbb{R}}^2}\left|\rho_0(\frac{\delta}{\lambda}|\xi|)
\tilde{\chi}(\frac{\xi_1}{\xi_2})\chi_{\nu}(\xi)\varphi(\frac{{\lambda}^{-1/2}q_{\delta}(\xi)-n}{10})\right|d\xi\|\widehat{f}\|_{L^{\infty}}\\
&\leq C_{M,N'}\delta \cdot \frac{\lambda}{{\delta}^{3/2}} (\frac{\lambda}{\delta})^{M-\mu N'/2}\frac{1}{(1+|(y,t)|^2)^M}\|\widehat{f}\|_{L^{\infty}}.\\
\end{align*}
The last inequality follows from the volume of set $\{\xi: \rho_0(\frac{\delta}{\lambda}|\xi|)
\tilde{\chi}(\frac{\xi_1}{\xi_2})\chi_{\nu}(\xi)\neq 0 \}$.

Since we have supposed that $f$ is supported in a fixed compact set, then there exists a sufficiently large $C_{M,N',p}$ so that the following inequality holds true,
\begin{align*}
|R_{\lambda,\nu}^{\delta}f_n(y,t)|
\leq C_{M,N',p}(\frac{\lambda}{\delta})^{-N}\frac{1}{(1+|(y,t)|^2)^M}\|f\|_{L^p}.
\end{align*}
We have finished the proof of Lemma \ref{planelemma3}.
\end{proof}

Until now, since $\sum_n(P^n)^2$ is the identity, applying Lemma \ref{planelemma1},  then we get
\begin{align*}
\|\mathcal{F}_{\lambda}^{\delta}f\|_{L^4}&=\biggl\|\sum_n(P^n)^2\mathcal{F}_{\lambda}^{\delta}f\biggl\|_{L^4}\\
& \leq \|\sum_n(P^n)^2\mathcal{F}_{\lambda}^{\delta}f_n\|_{L^4}+\biggl\|\sum_n(P^n)^2\mathcal{F}_{\lambda}^{\delta}(f-f_n)\biggl\|_{L^4}\\
& \leq \biggl\|\sum_n(P^n)^2\mathcal{F}_{\lambda}^{\delta}f_n\biggl\|_{L^4}+C_N\lambda^{-N}\|f\|_{L^4}.
\end{align*}

It is more convenient to consider a related square function than to estimate the sum of the main term in the right-hand side directly, so by Lemma \ref{planelemma2} and Lemma \ref{planelemma3}, we obtain
\begin{align*}
\biggl\|\sum_n(P^n)^2\mathcal{F}_{\lambda}^{\delta}f_n\biggl\|_{L^4}&\leq C\lambda^{1/8}\biggl\|\left(\sum_n|P^n\mathcal{F}_{\lambda}^{\delta}f_n|^2\right)^{1/2}\biggl\|_{L^4}\\
& \leq C\lambda^{1/8}\biggl\|\left(\sum_n|P^n\sum_{\nu}Q_{\nu}
(\mathcal{F}_{\lambda,\nu}^{\delta}f_n)|^2\right)^{1/2}\biggl\|_{L^4}+C_N\lambda^{-N}\|f\|_{L^4}
\end{align*}

To estimate the main term in the last expression we notice that, for each fixed $n\approx \lambda^{1/2}$ and  $\nu$, the Fourier transform of $P^nQ_{\nu}g$ is supported in the following region (see Figure 2),
\begin{equation*}
\mathcal{U}_{\nu}^{n}=\{(\eta,\tau)\in \mathbb{R}^3:(\eta,\tau)\in \hspace{0.1cm}\textrm{supp} \hspace{0.1cm} \Psi_{\nu}^{\delta}(\eta,\tau), |\lambda^{-1/2}\tau-n|\leq 1\}.
\end{equation*}

We need an estimate of the number of overlaps of algebraic sums of the sets $\mathcal{U}_{\nu}^{n}$, $\mathcal{U}_{\nu'}^{n'}$, for fixed $n, n'$ corresponding to indices $\nu$ for which $\xi_{\nu}$ lie in the first quadrant.

\begin{center}
\includegraphics[height=10.5cm]{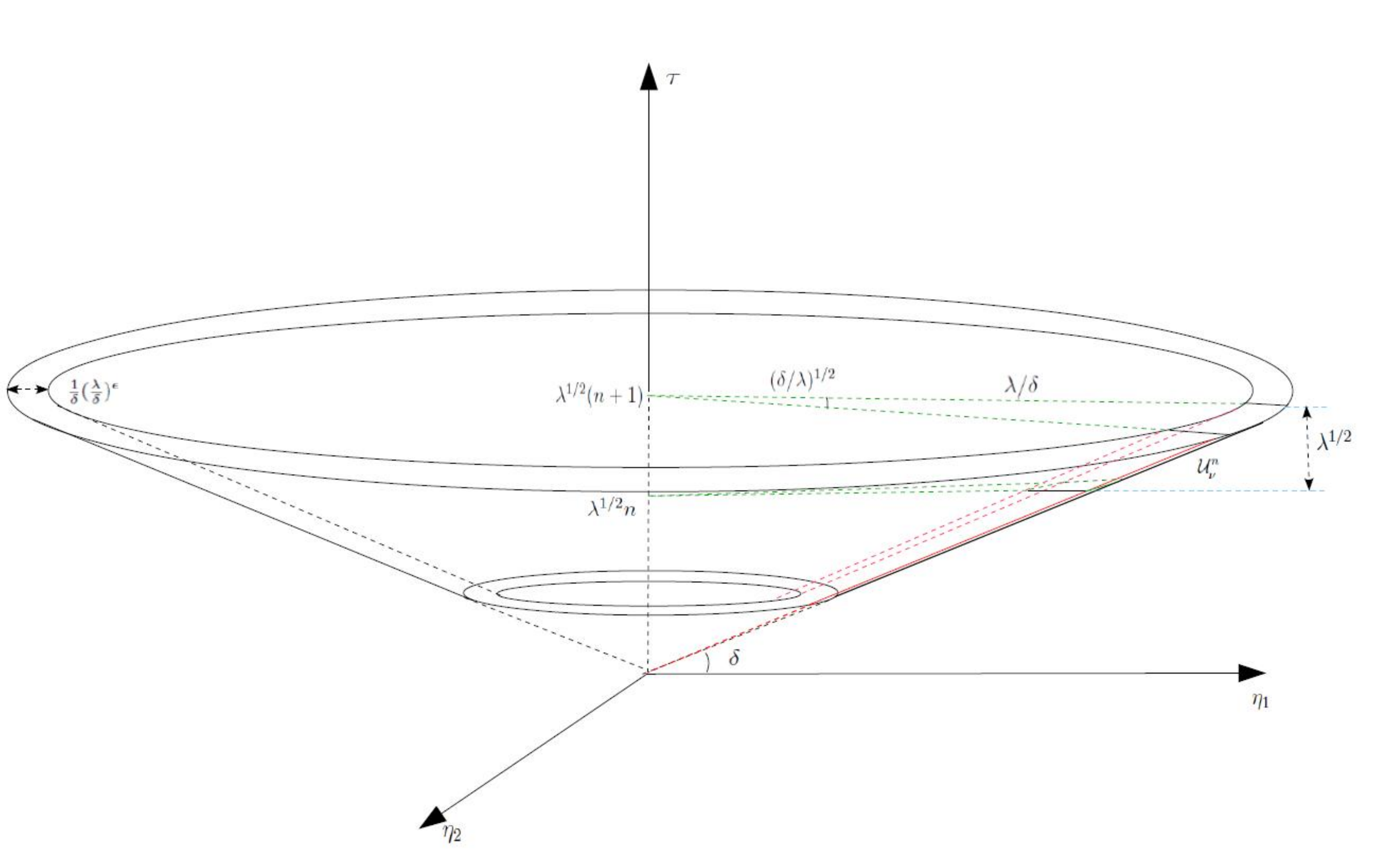}
\end{center}

\begin{center}
Figure 2: A thin neighborhood of the cone $\tau=\delta |\eta|$ and $\mathcal{U}_{\nu}^{n}$\\
( The figure of the cone $\tau=q_{\delta}(\eta)$ is similar to a sector of the cone $\tau=\delta |\eta|$. )
\end{center}


\begin{lem}\label{coverlemma1}
For fixed $n$, $n'\in \mathbb{Z}$, there is a constant $C$, independent of $n$, $n'$, $\lambda$ and $\delta$, such that
\begin{equation}
\sum_{\nu,\nu'}\chi_{\mathcal{U}_{\nu}^{n}+\mathcal{U}_{\nu'}^{n'}}(\eta,\tau)\leq C(\frac{\lambda}{\delta})^{2\epsilon}\log_2(\frac{\lambda}{\delta})^{1/2}.
\end{equation}
\end{lem}

\begin{proof}
First, we split every $\mathcal{U}_{\nu}^{n}$ into $(\frac{\lambda}{\delta})^{\epsilon}$ pieces along the normal direction to the cone $\tau=q_{\delta}(\eta)$ (vertical cross-section of $\mathcal{U}_{\nu}^{n}$ will be shown in Figure 3 (1)), then it suffices to prove that
\begin{equation}\label{coverlemma1'}
\sum_{\nu,\nu'}\chi_{\mathcal{\tilde{U}}_{\nu}^{n}+\mathcal{\tilde{U}}_{\nu'}^{n'}}(\eta,\tau)\leq C\log_2(\frac{\lambda}{\delta})^{1/2},
\end{equation}
where $\mathcal{\tilde{U}}_{\nu}^{n}$ is comparable to a rectangle tangential to the cone  of dimension $1\times \frac{\lambda^{1/2}}{\delta}\times(\frac{\lambda}{\delta})^{1/2}$ ( vertical cross-section of $\mathcal{\tilde{U}}_{\nu}^{n}$ and the comparable rectangle will be shown in Figure 3 (2) and the left-hand one of Figure 4).

\begin{center}
\includegraphics[height=7cm]{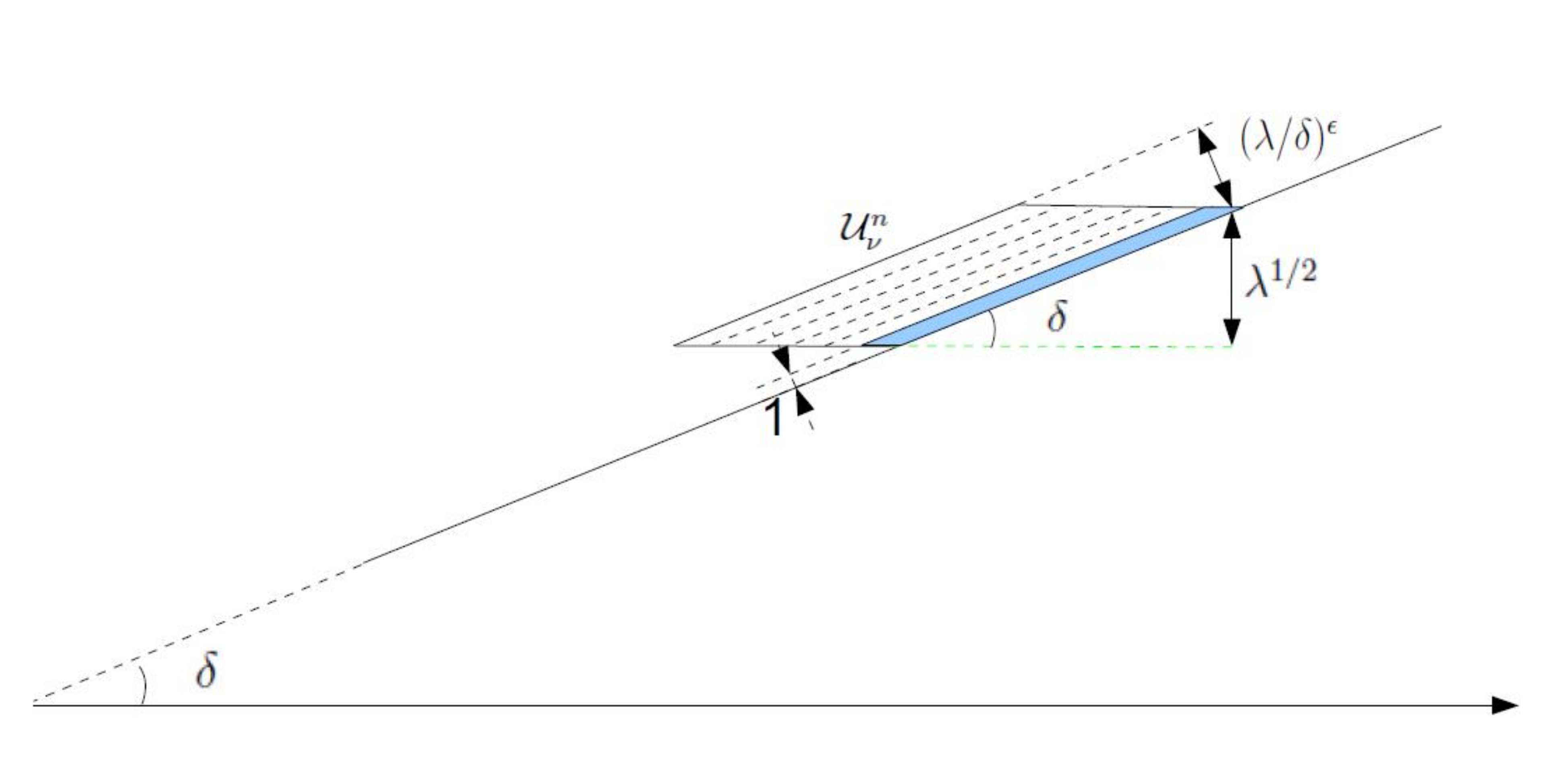}
\end{center}

\begin{center}
Figure 3 (1): Vertical cross-section of $\mathcal{U}_{\nu}^{n}$
\end{center}

\begin{center}
\includegraphics[height=7cm]{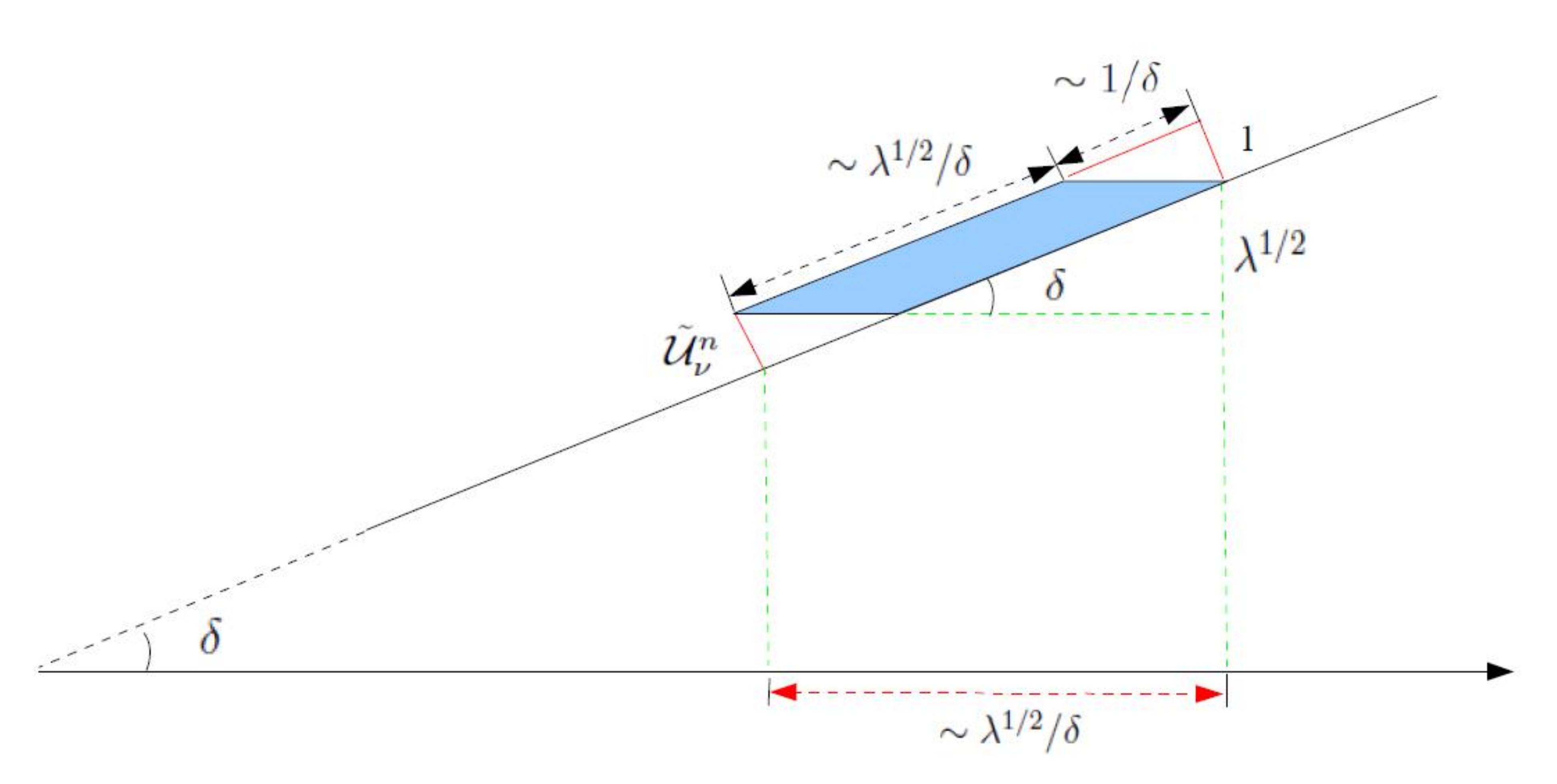}
\end{center}

\begin{center}
Figure 3 (2): Vertical cross-section of $\mathcal{\tilde{U}}_{\nu}^{n}$ and the comparable rectangle
\end{center}

Furthermore, here we use similar arguments as in \cite{mss} to prove a scaled version by $1/\lambda$ of inequality (\ref{coverlemma1'}).

To this end we define $\Gamma_{\nu}=\{\xi\in \mathbb{R}^2\backslash 0:\sqrt{\frac{\delta}{\lambda}}\nu\leq|\frac{\xi_2}{\xi_1}|\leq \sqrt{\frac{\delta}{\lambda}}(\nu+1)$
 for $\nu =(\frac{\pi}{2}-c_2)\sqrt{\frac{\lambda}{\delta}}, \cdots, (\frac{\pi}{2}-c_1)\sqrt{\frac{\lambda}{\delta}}$ and here non-zero positive numbers $c_1, c_2$ with $c_1<c_2$
 depend on the support of $\tilde{\chi}$. Next, for $n=0,1,\cdots, \lambda^{1/2}$, let $\Lambda_{\nu}^n$ be the set of all $(\eta,\tau)$
 such that dist$((\eta,\tau),(\xi,q_{\delta}(\xi)))\leq \lambda^{-1}$ for some $\xi \in \Gamma_{\nu}$
 with $q_{\delta}(\xi)\in [(\frac{\pi}{2}-c_1)^{-2}+\frac{n}{\sqrt{\lambda}}, (\frac{\pi}{2}-c_1)^{-2}+\frac{(n+1)}{\sqrt{\lambda}}]$. Thus $\Lambda_{\nu}^n$
 is basically a $\lambda^{-1}\times \frac{1}{\delta\sqrt{\lambda}}\times\frac{1}{\sqrt{\delta\lambda}}$ rectangle (see the right one of Figure 4)
 lying on $\{(\xi,q_{\delta}(\xi)):q_{\delta}(\xi)\in [(\frac{\pi}{2}-c_1)^{-2},(\frac{\pi}{2}-c_2)^{-2}]\}$. Its shortest side points in the normal direction to the cone.

\begin{center}
\includegraphics[height=6cm]{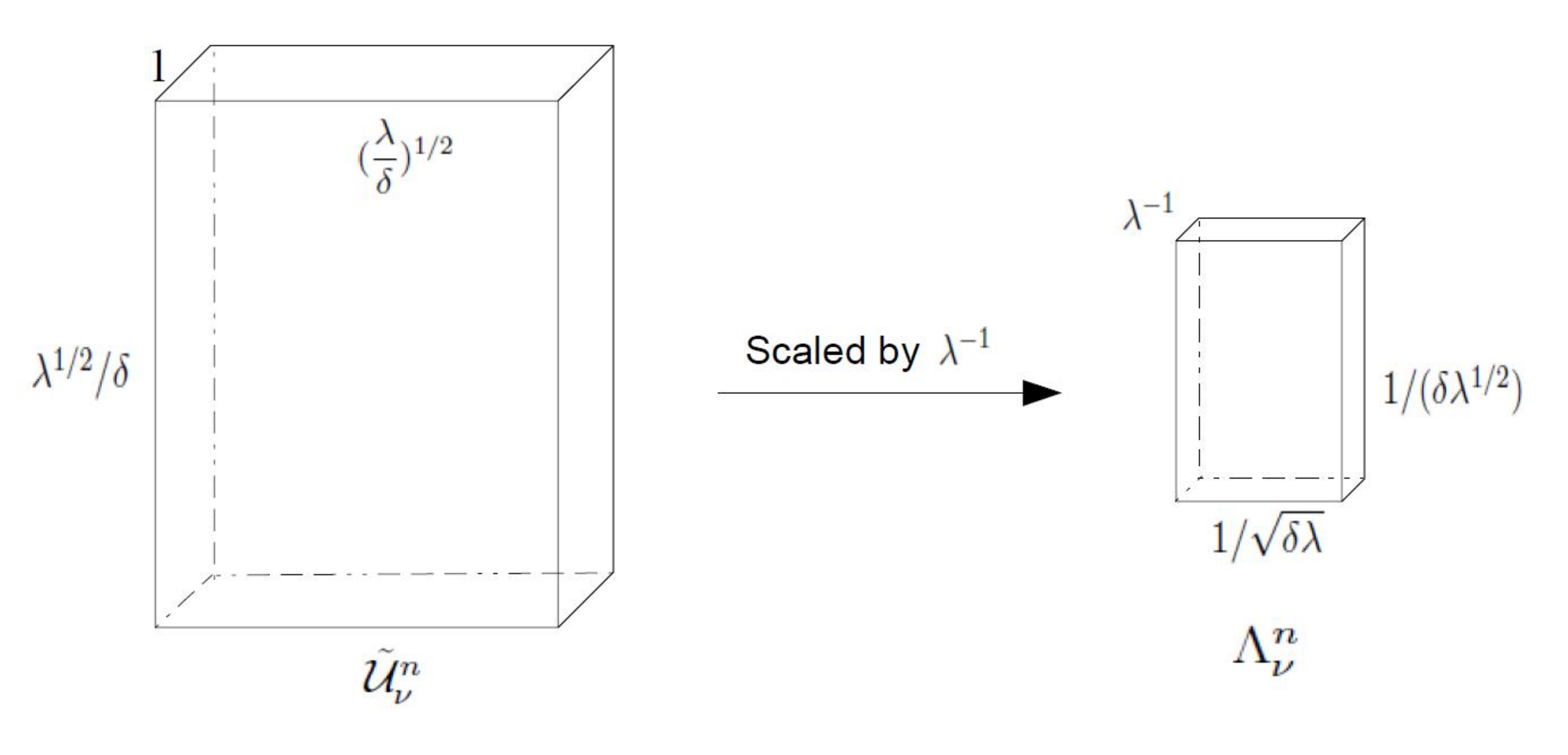}
\end{center}

\begin{center}
Figure 4
\end{center}

The scaled version of the overlap lemma is that there is a constant $C$ independent of $n$, $n'$, $\lambda$ and $\delta$ such that
\begin{equation}\label{coverlemma2}
\sum_{\nu,\nu'}\chi_{\Lambda_{\nu}^{n}+\Lambda_{\nu'}^{n'}}(\eta,\tau)\leq C\log_2(\frac{\lambda}{\delta})^{1/2}.
\end{equation}

However, inequality (\ref{coverlemma2}) follows from
\begin{equation}\label{coverlemma3}
\sum_{|\nu-\nu'|\thickapprox 2^{\ell}}\chi_{\Lambda_{\nu}^{n}+\Lambda_{\nu'}^{n'}}(\eta,\tau)\leq C,\hspace{0.2cm}0\leq \ell\leq \log_2(\frac{\lambda}{\delta})^{1/2}.
\end{equation}

To deduce inequality (\ref{coverlemma3}), we give a simple argument to show that for $\bar{\xi}\in \Gamma_{\nu}$ and $\tilde{\xi}\in \Gamma_{\nu'}$, the angle between the normals to the cone at $(\bar{\xi}, q_{\delta}(\bar{\xi}))$ and $(\tilde{\xi},q_{\delta}(\tilde{\xi}))$ is $\approx 2^{\ell}(\frac{\delta}{\lambda})^{1/2}$ if $|\nu-\nu'|\approx 2^{\ell}$.  It will follow from the claim that $dist(P,A)\approx dist(P,D)\approx 1/\delta$ (See Figure 5 for $q_{\delta}(\xi)=1$). Next, we will first prove this claim.

In Figure 5,  $AE$ and $DE$ are the tangent line of the curve $q_{\delta}(\xi)= 1$ at point $(\bar{\xi}_1, 0)$, $(\tilde{\xi}_1, 0)$, where $\bar{\xi}_1,\tilde{\xi}_1 \approx 1/\delta$. $PA$ and $PD$ are perpendicular respectively $AE$ and $DE$.

From the expression of the curve $q_{\delta}(\xi)=1$, it is easy to know that
$\frac{\pi}{2}-c_2\leq \tan  \angle AOB\approx\tan\angle B, \tan\angle C\approx\tan \angle DOB \leq \frac{\pi}{2}-c_1$.

Since $\tan (\angle DOB-\angle AOB)\approx 2^{\ell}(\frac{\delta}{\lambda})^{1/2}$, then
\begin{align*}
 &\tan \angle C-\tan\angle B\approx \delta^{1/2}(\tilde{\xi}_1^{1/2}-\bar{\xi}_1^{1/2})\approx \tan \angle DOB-\tan \angle AOB\approx 2^{\ell}(\frac{\delta}{\lambda})^{1/2},
\end{align*}
which implies $\tan(\angle C-\angle B)\approx 2^{\ell}(\frac{\delta}{\lambda})^{1/2}.$
As a result of $\angle APD+\angle AED=\angle DEI+\angle AED=180^\circ$, then $\angle APD=\angle DEI=\angle C-\angle B\approx 2^{\ell}(\frac{\delta}{\lambda})^{1/2}$.
\begin{center}
\includegraphics[height=14cm]{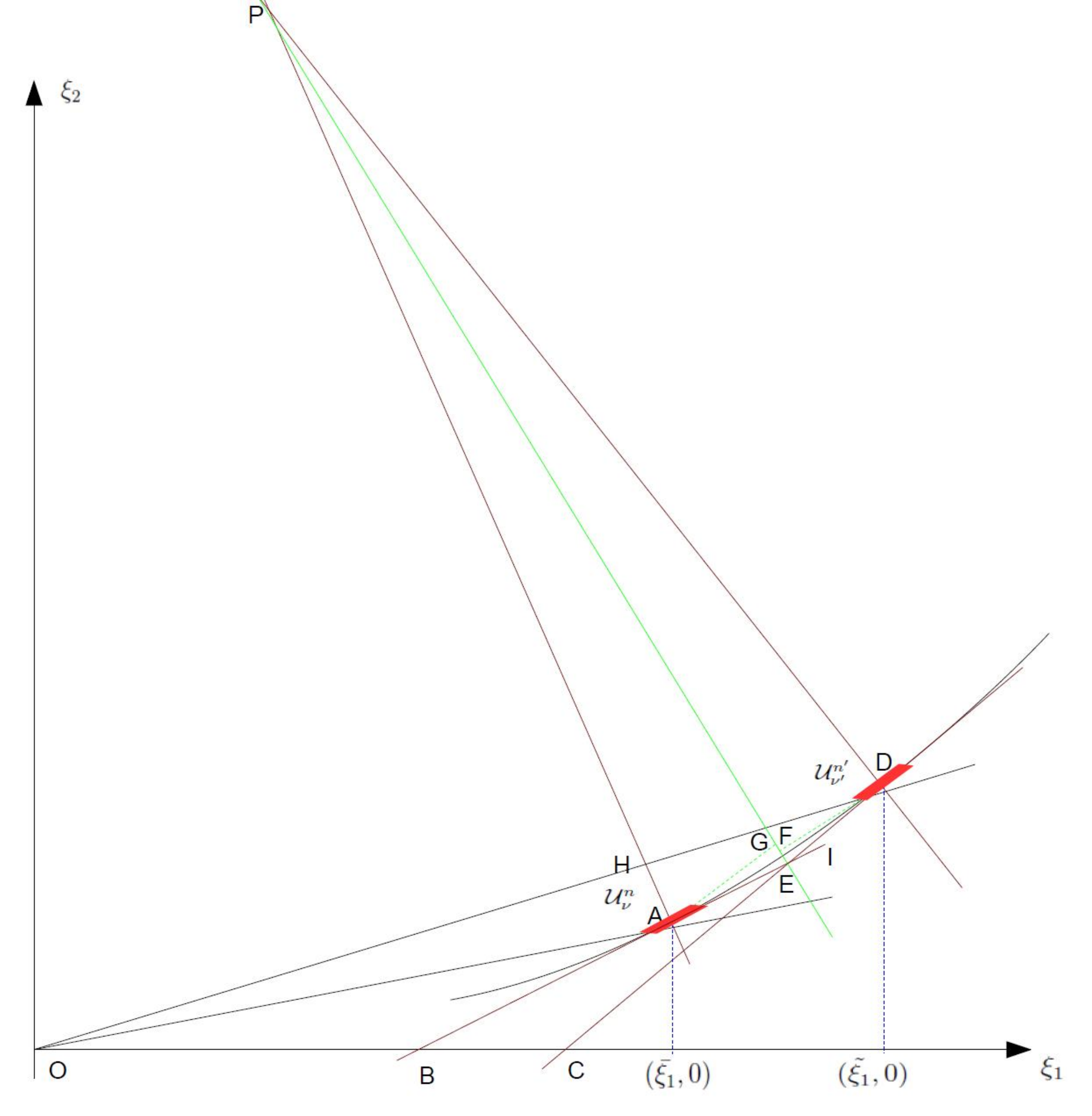}
\end{center}

\begin{center}
Figure 5: Curve $q_{\delta}(\xi)=1$
\end{center}

Since $q_{\delta}(\xi)=1$, then $\delta^{-1/2}\frac{\xi_2}{\xi_1}=\xi_1^{1/2}$, and the slope of the line $PA$ (the line $PD$) is equal to $\frac{3}{2}\delta^{1/2}\bar{\xi}_1^{1/2}$ ($\frac{3}{2}\delta^{1/2}\tilde{\xi}_1^{1/2}$). So
$\tilde{\xi}_1^{1/2}-\bar{\xi}_1^{1/2}=\delta^{-1/2}(\tan\angle DOB-\tan\angle AOB)\approx \delta^{-1/2}2^{\ell}(\frac{\delta}{\lambda})^{1/2}= 2^{\ell}/\lambda^{1/2}$.
Then the line $PA$ can be written as
\begin{equation*}
\xi_2= -\frac{\xi_1}{\frac{3}{2}\delta^{1/2}\bar{\xi}_1^{1/2}}+\delta^{1/2}\bar{\xi}_1^{3/2}+\frac{2}{3}\delta^{-1/2}\bar{\xi}_1^{1/2}
\end{equation*}
and the line $PD$ is
\begin{equation*}
\xi_2= -\frac{\xi_1}{\frac{3}{2}\delta^{1/2}\tilde{\xi}_1^{1/2}}+\delta^{1/2}\tilde{\xi}_1^{3/2}+\frac{2}{3}\delta^{-1/2}\tilde{\xi}_1^{1/2}.
\end{equation*}
It's easy to compute the point $P\approx (-\frac{1}{\delta},\frac{1}{\delta})$. The claim was proved.

Since $AG\bot PG$, $DF\bot PF$, $PA\bot AE$ and $PD\bot DE$, then $\angle GAE=\angle APG\approx 2^{\ell}(\frac{\delta}{\lambda})^{1/2}$ and $\angle EDF=\angle DPF\approx 2^{\ell}(\frac{\delta}{\lambda})^{1/2}$.

The projection on $PG$ of the width of $\mathcal{U}_{\nu}^{n}$ is equal to $\sin\angle GAE \times \frac{1}{\sqrt{\delta\lambda}}\approx \frac{2^{\ell}}{\lambda}\geq{\lambda}^{-1}$. For $\nu$, $\nu'$ with $|\nu-\nu'|\approx 2^{\ell}$,  $\Lambda_{\nu}^n+\Lambda_{\nu'}^{n'}$ is comparable to a rectangle of size $2^\ell\lambda^{-1}\times \frac{1}{\sqrt{\delta\lambda}}\times 1/(\delta\lambda^{1/2})$. In other words, $vol(\Lambda_{\nu}^n+\Lambda_{\nu'}^{n'})\approx 2^{\ell}\delta^{-1/2}\lambda^{-2}$.

Next we define
\begin{equation*}
\Omega=\{(\eta,\tau):\sum_{|\nu-\nu'|\thickapprox 2^{\ell}}\chi_{\Lambda_{\nu}^{n}+\Lambda_{\nu'}^{n'}}(\eta,\tau)\geq \frac{1}{2}\max_{(\eta,\tau)}\sum_{|\nu-\nu'|\thickapprox 2^{\ell}}\chi_{\Lambda_{\nu}^{n}+\Lambda_{\nu'}^{n'}}(\eta,\tau)\}.
\end{equation*}
Since the sum is equally distributed on $\Omega$, one can check that $$vol(\Omega)\approx \underbrace{(2^{\ell}(\frac{\delta}{\lambda})^{1/2}\times\frac{2^{\ell}}{\sqrt{\delta\lambda}})}_d\times \lambda^{-1/2}\times \delta^{-1}=\frac{2^{2\ell}}{\delta^2\lambda^{3/2}}.$$

\begin{center}
\includegraphics[height=7cm]{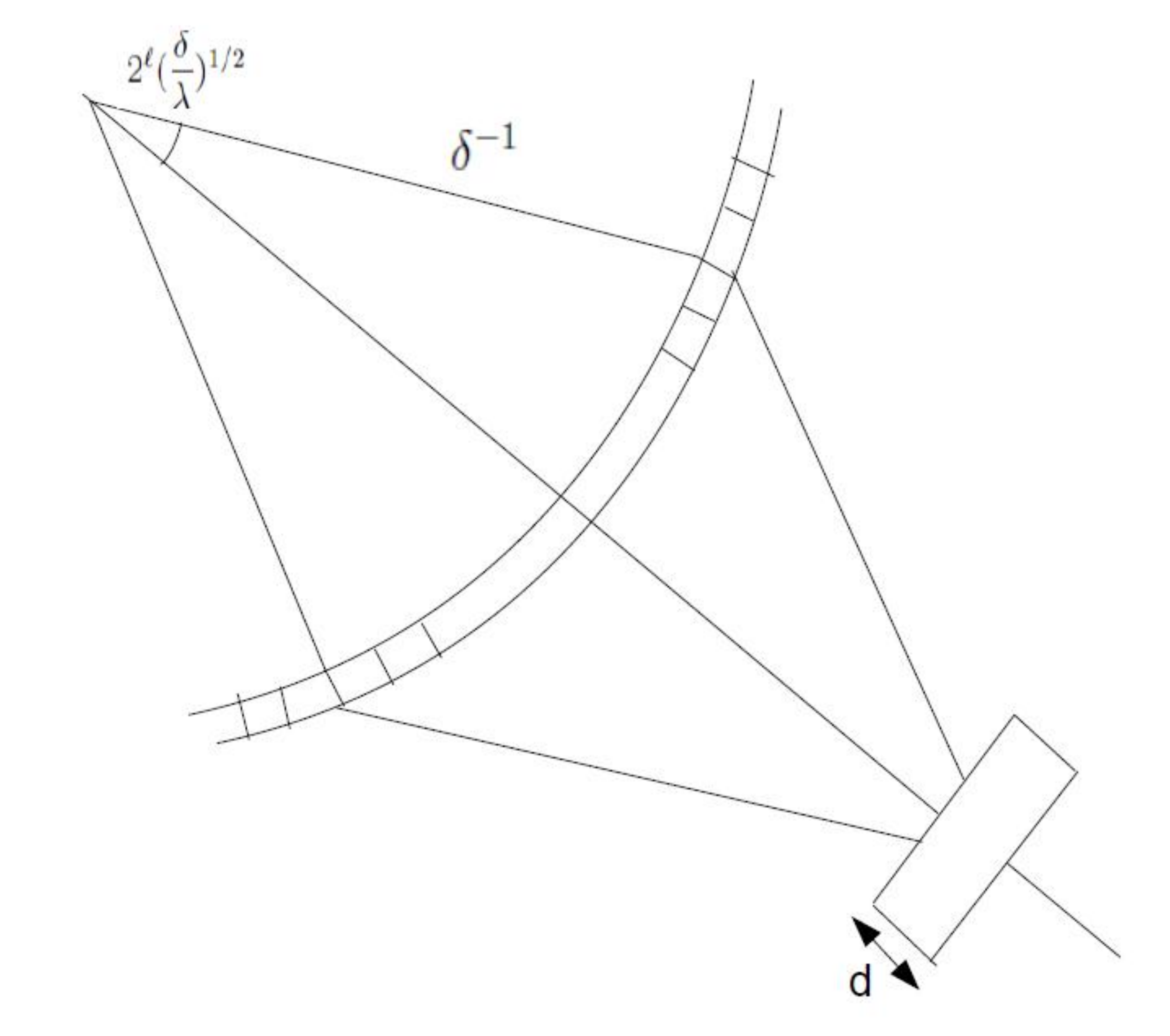}
\end{center}

\begin{center}
Figure 6 : $\Lambda_{\nu}^{n}+\Lambda_{\nu'}^{n'}$ for $|\nu-\nu'|\thickapprox 2^{\ell}$
\end{center}
\begin{center}
\includegraphics[height=3cm]{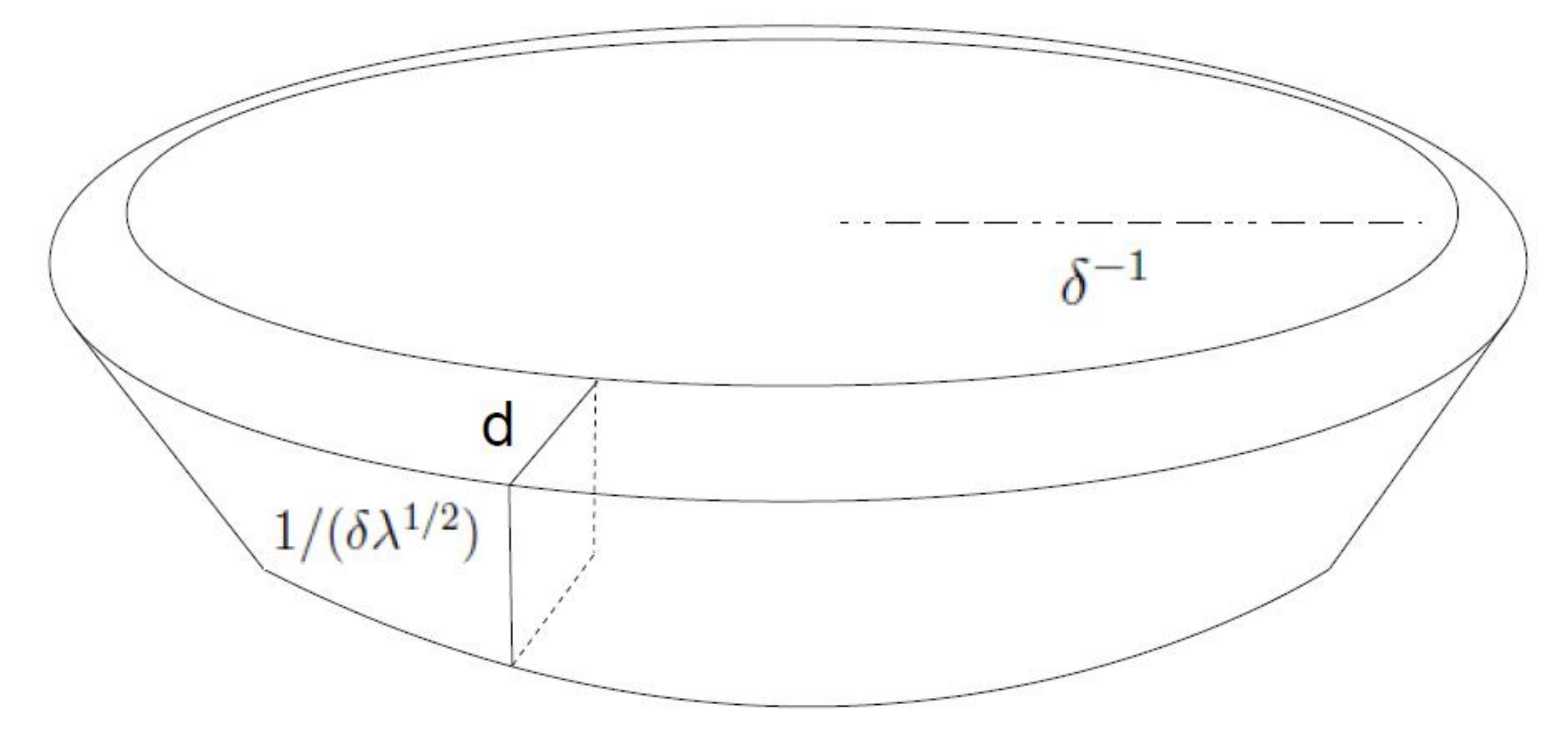}
\end{center}

\begin{center}
Figure 7 : $\Omega$
\end{center}

Set $\bar{C}=\max_{(\eta,\tau)}\sum_{|\nu-\nu'|\thickapprox 2^{\ell}}\chi_{\Lambda_{\nu}^{n}+\Lambda_{\nu'}^{n'}}(\eta,\tau)$, then
\begin{align*}
\bar{C}vol(\Omega)=\int_{\Omega}\bar{C}d\eta d\tau & \leq 2\int_{\mathbb{R}^3}\sum_{|\nu-\nu'|\thickapprox 2^{\ell}}\chi_{\Lambda_{\nu}^{n}+\Lambda_{\nu'}^{n'}}(\eta,\tau)d\eta d\tau\\
& =2\sum_{|\nu-\nu'|\thickapprox 2^{\ell}}vol({\Lambda_{\nu}^{n}+\Lambda_{\nu'}^{n'}})\\
& \lesssim 2^{\ell}(\frac{\lambda}{\delta})^{1/2}2^{\ell}\delta^{-3/2}\lambda^{-2}\lesssim 2^{2\ell}\frac{1}{\delta^2\lambda^{3/2}},
\end{align*}
together with the estimate of $vol(\Omega)$, we know $\bar{C}\leq C$, i.e.
$\sum_{|\nu-\nu'|\thickapprox 2^{\ell}}\chi_{\Lambda_{\nu}^{n}+\Lambda_{\nu'}^{n'}}(\eta,\tau)\leq C$.
\end{proof}

Let us continue our estimate for  $\biggl\|\biggl(\sum_n|P^n\sum_{\nu}Q_{\nu}
(\mathcal{F}_{\lambda,\nu}^{\delta}f_n)|^2\biggl)^{1/2}\biggl\|_{L^4}$. By Plancherel's theorem and the above overlap lemma, we have
\begin{align*}
&\biggl\|\biggl(\sum_n|P^n\sum_{\nu}Q_{\nu}
(\mathcal{F}_{\lambda,\nu}^{\delta}f_n)|^2\biggl)^{1/2}\biggl\|_{L^4}^4\\
&=\int_{\mathbb{R}^3}\biggl(\sum_n|P^n\sum_{\nu}Q_{\nu}
(\mathcal{F}_{\lambda,\nu}^{\delta}f_n)|^2\biggl)^2dydt\\
& =\int_{\mathbb{R}^3} \sum_n\left|P^n\sum_{\nu}Q_{\nu}
(\mathcal{F}_{\lambda,\nu}^{\delta}f_n)\right|^2\sum_{n'}\left|P^{n'}\sum_{\nu'}Q_{\nu'}
(\mathcal{F}_{\lambda,\nu'}^{\delta}f_{n'})\right|^2 dydt\\
& =\int_{\mathbb{R}^3} \sum_{n,n'}\left|\sum_{\nu,\nu'}P^nQ_{\nu}
(\mathcal{F}_{\lambda,\nu}^{\delta}f_n) P^{n'}Q_{\nu'}
(\mathcal{F}_{\lambda,\nu'}^{\delta}f_{n'})\right|^2 dydt\\
& =\int_{\mathbb{R}^3} \sum_{n,n'}\left|\sum_{\nu,\nu'}\chi_{U_{\nu}^n+U_{\nu'}^{n'}}(\eta, \tau)(P^nQ_{\nu}
(\mathcal{F}_{\lambda,\nu}^{\delta}f_n))^{\wedge}*(P^{n'}Q_{\nu'}
(\mathcal{F}_{\lambda,\nu'}^{\delta}f_{n'}))^{\wedge}(\eta, \tau)\right|^2 d\eta d\tau\\
&\leq \int_{\mathbb{R}^3} \sum_{n,n'}\sum_{\nu,\nu'}\chi^2_{U_{\nu}^n+U_{\nu'}^{n'}}(\eta, \tau)\sum_{\nu,\nu'}\left|(P^nQ_{\nu}
(\mathcal{F}_{\lambda,\nu}^{\delta}f_n))^{\wedge}*(P^{n'}Q_{\nu'}
(\mathcal{F}_{\lambda,\nu'}^{\delta}f_{n'}))^{\wedge}(\eta, \tau)\right|^2 d\eta d\tau\\
&\leq C(\frac{\lambda}{\delta})^{2\epsilon}\log_2(\frac{\lambda}{\delta})^{1/2}\int_{\mathbb{R}^3} \sum_{n,n'}\sum_{\nu,\nu'}\left|P^nQ_{\nu}
(\mathcal{F}_{\lambda,\nu}^{\delta}f_n) P^{n'}Q_{\nu'}
(\mathcal{F}_{\lambda,\nu'}^{\delta}f_{n'})\right|^2 dy dt\\
&=C(\frac{\lambda}{\delta})^{2\epsilon}\log_2(\frac{\lambda}{\delta})^{1/2}\biggl\|\biggl(\sum_{n,\nu}|P^nQ_{\nu}
(\mathcal{F}_{\lambda,\nu}^{\delta}f_n)|^2\biggl)^{1/2}\biggl\|_{L^4}^4.
\end{align*}

Hence by Lemma \ref{planelemma3}, it suffices to estimate $\biggl\|\biggl(\sum_{n,\nu}|P^n
\mathcal{F}_{\lambda,\nu}^{\delta}f_n|^2\biggl)^{1/2}\biggl\|_{L^4}$.

Furthermore,
\begin{align*}
&\biggl\|\biggl(\sum_{n,\nu}|P^n
(\mathcal{F}_{\lambda,\nu}^{\delta}f_n)|^2\biggl)^{1/2}\biggl\|_{L^4}^4\\
&=\biggl\|\biggl(\sum_{n,\nu}|\lambda^{1/2}e^{-i\lambda^{1/2}nt}\check{\varphi}(\lambda^{1/2}\cdot)*_t
\mathcal{F}_{\lambda,\nu}^{\delta}f_n(x,t)|^2\biggl)^{1/2}\biggl\|_{L^4}^4\\
&\leq C_N' \Biggl\|\biggl(\sum_{n,\nu}|\lambda^{1/2}\int_{\mathbb{R}}\frac{|\mathcal{F}_{\lambda,\nu}^{\delta}f_n(y,t-s)|}{(1+\lambda^{1/2}|s|)^N}ds
|^2\biggl)^{1/2}\Biggl\|_{L^4}^4\\
&\leq C_N'\int_{\mathbb{R}^2}\int_{\mathbb{R}}\biggl[\sum_{n,\nu}\lambda^{1/2}\int_{\mathbb{R}}\frac{1}
{(1+\lambda^{1/2}|s|)^{N}}ds\lambda^{1/2}\int_{\mathbb{R}}\frac{|\mathcal{F}_{\lambda,\nu}^{\delta}f_n(y,t-s)|^2}
{(1+\lambda^{1/2}|s|)^{N}}ds\biggl]^2dydt\\
&\leq C_N''\int_{\mathbb{R}^2}\int_{\mathbb{R}}\biggl(\sum_{n,\nu}\lambda^{1/2}\int_{\mathbb{R}}\frac{|\mathcal{F}_{\lambda,\nu}^{\delta}f_n(y,t-s)|^2}
{(1+\lambda^{1/2}|s|)^{N}}ds\biggl)^2dydt\\
&\leq C_N''\biggl\|\sum_{n,\nu}\lambda^{1/2}\int_{\mathbb{R}}\frac{|\mathcal{F}_{\lambda,\nu}^{\delta}f_n(y,t-s)|^2}
{(1+\lambda^{1/2}|s|)^{N}}ds\biggl\|_{L^2}^2\\
&\leq C_N''\biggl(\int_{\mathbb{R}}\lambda^{1/2}(1+\lambda^{1/2}|s|)^{-N}\|\sum_{n,\nu}|\mathcal{F}_{\lambda,\nu}^{\delta}f_n(y,t-s)|^2\|_{L^2}ds\biggl)^2\\
&\leq C_N'''\biggl\|\biggl(\sum_{n,\nu}|\mathcal{F}_{\lambda,\nu}^{\delta}f_n(y,t)|^2\biggl)^{1/2}\biggl\|_{L^4}^4,\\
\end{align*}
where $*_t$ means that convolution acts only on the variable $t$. At present we have
\begin{align*}
\|\mathcal{F}_{\lambda}^{\delta}f\|_{L^4}
&\leq C\lambda^{1/8}(\frac{\lambda}{\delta})^{\epsilon/2}[\log_2(\frac{\lambda}{\delta})^{1/2}]^{1/4}
\biggl\|\biggl(\sum_{n,\nu}|\mathcal{F}_{\lambda,\nu}^{\delta}f_n(x,t)|^2\biggl)^{1/2}\biggl\|_{L^4}+C_N\lambda^{-N}\|f\|_{L^4}.
\end{align*}

It remains to estimate $\biggl\|\biggl(\sum_{n,\nu}|\mathcal{F}_{\lambda,\nu}^{\delta}f_n(x,t)|^2\biggl)^{1/2}\biggl\|_{L^4}$. For $m=(m_1,m_2)\in \mathbb{Z}^2$, we define the operator $P_m^{\delta}$ acting on the functions in ${\mathbb{R}}^2$ by
\begin{equation}
(P_m^{\delta}f)^{\wedge}(\xi)=\varphi(\frac{\delta}{\lambda^{1/2}}\xi_1-m_1)\varphi(\frac{\delta}{\lambda^{1/2}}\xi_2-m_2)\widehat{f}(\xi),
\end{equation}
then $f=\sum_{m\in \mathbb{Z}^2}P_m^{\delta}f$.

\begin{center}
\includegraphics[height=7cm]{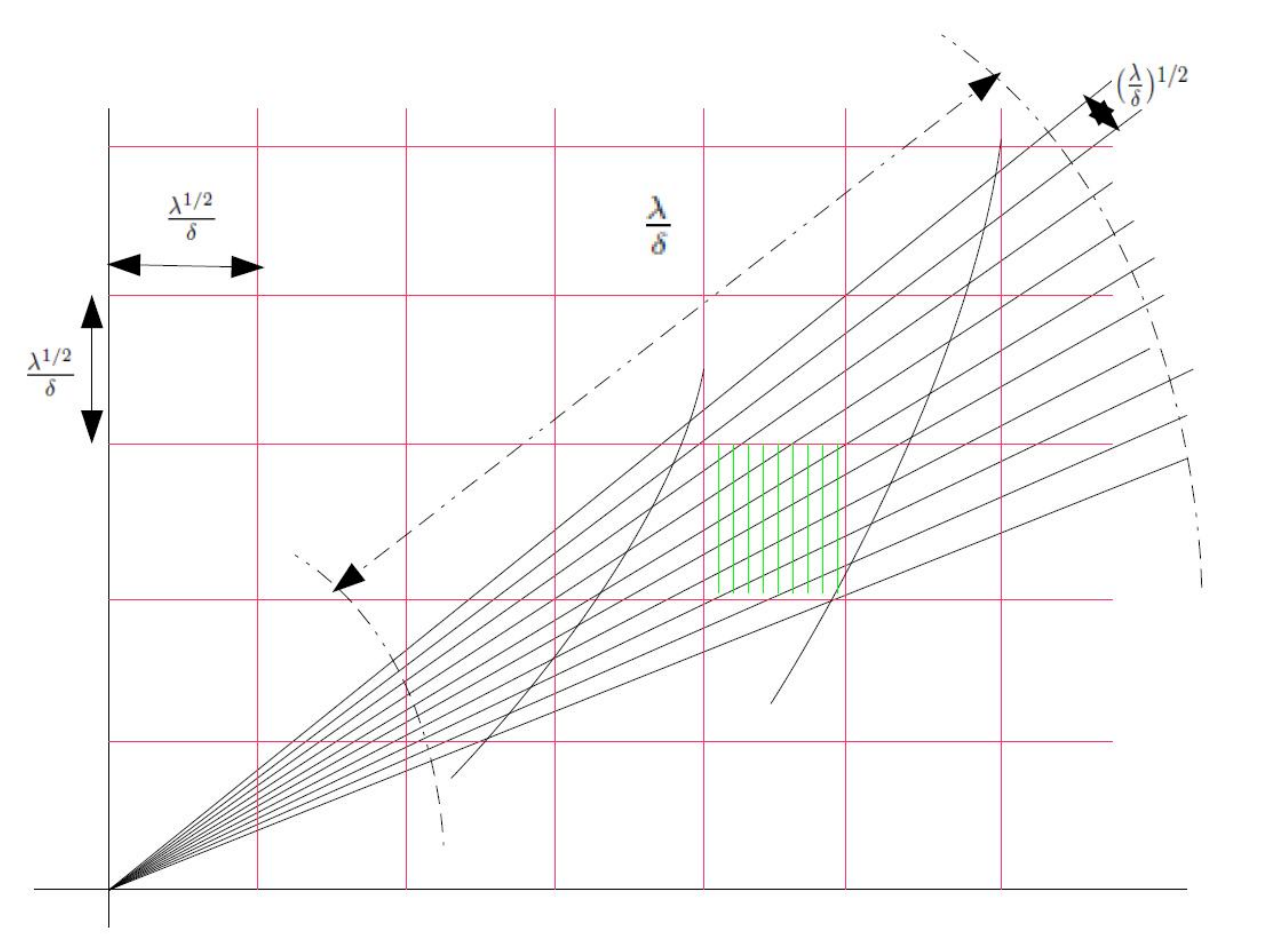}
\end{center}

\begin{center}
Figure 8 : Card $\{(n,\nu):m\in J_{n,\nu}\}\lesssim \delta^{-1/2}$
\end{center}

The definition of $f_n$ and the support properties of $\chi_{\nu}$ imply
\begin{equation*}
\mathcal{F}_{\lambda,\nu}^{\delta}f_n=\mathcal{F}_{\lambda,\nu}^{\delta}f_{n,\nu},
\end{equation*}
where $f_{n,\nu}=\sum_{m\in J_{n,\nu}}P_m^{\delta}f$ and the  index set $J_{n,\nu}$ is contained in the set of all $m\in \mathbb{Z}^2$  such that $\chi_{\nu}\widehat{P_m^{\delta}f_n}$ is not identically zero. It follows from the above discussion that
\begin{equation}\label{cardina1}
card \{J_{n,\nu}\}\leq C.
\end{equation}

Meanwhile, we also get
\begin{equation}\label{cardina2}
card \{(n,\nu):m\in J_{n,\nu}\}\lesssim \delta^{-1/2}.
\end{equation}

Let the kernel $K_{\lambda,\nu,n}^{\delta}(y,t;x)$ of $\mathcal{F}_{\lambda,\nu}^{\delta}f_{n,\nu}$ be
\begin{equation}\label{kernels}
\rho_1(y,t)\int_{\mathbb{R}^2}e^{i((y-x)\cdot \xi+E(\xi)+tq_{\delta}(\xi))}\rho_0(\frac{\delta}{\lambda}|\xi|)\tilde{\chi}(\frac{\xi_1}{\xi_2})
a(\xi,t)\chi_{\nu}(\xi)\varphi ((\lambda^{-1/2}q_{\delta}(\xi)-n )/10)d\xi.
\end{equation}

We need the following estimates. Notice that (\ref{equ:planekerna22}) is an estimate of a Kakeya-type maximal operator.
\begin{lem}\label{lemma index2}
For fixed $n$, $\nu$ and $g\in L^2(\mathbb{R}^2)$, then
\begin{equation}\label{equ:planekerna12}
\int_{\mathbb{R}^2}|K_{\lambda,\nu,n}^{\delta}(y,t;x)|dx\leq C;
\end{equation}
\begin{equation}\label{equ:planekerna22}
\biggl(\int_{\mathbb{R}^2}\sup_{\nu,n}\left\{\int_{\mathbb{R}^2}\int_{\mathbb{R}}|K_{\lambda,\nu,n}^{\delta}(y,t;x)||g(y,t)|dydt\right\}^2dx\biggl)^{1/2}\leq C\delta^{-1/2}(\log_2\frac{\lambda^{1/2}}{\delta})^2\|g\|_{L^2}.
\end{equation}
\end{lem}

\begin{lem}\cite{mss}\label{lemma index3}
For $2\leq p\leq \infty$,
\begin{equation*}
\biggl\|\biggl(\sum_{m\in \mathbb{Z}^2}|P_m^{\delta}f(y)|^2\biggl)^{1/2}\biggl\|_{L^p(\mathbb{R}^2)}\leq C \|f\|_{L^p(\mathbb{R}^2)}.
\end{equation*}
\end{lem}

Supposing that Lemma \ref{lemma index2} holds true, together with (\ref{cardina1}), (\ref{cardina2}) and Lemma \ref{lemma index3}, a duality argument and H\"{o}lder's inequality give
\begin{align*}
&\biggl\|\biggl(\sum_{n,\nu}|\mathcal{F}_{\lambda,\nu}^{\delta}f_n(y,t)|^2\biggl)^{1/2}\biggl\|_{L^4}^2\\
& =\sup_{\|g\|_{L^2}=1}\left|\int_{\mathbb{R}^2}\int_{\mathbb{R}}\sum_{n,\nu}|\mathcal{F}_{\lambda,\nu}^{\delta}f_n(y,t)|^2g(y,t)dydt\right|\\
&=\sup_{\|g\|_{L^2}=1}\sum_{n,\nu}\int_{\mathbb{R}^2}\int_{\mathbb{R}}\left|\int_{\mathbb{R}^2}
(K_{\lambda,\nu,n}^{\delta}(y,t;x))^{1/2}[(K_{\lambda,\nu,n}^{\delta}(y,t;x))^{1/2}f_{n,\nu}(x)]dx\right|^2|g(y,t)|dydt\\
&\leq  \sup_{\|g\|_{L^2}=1}\sum_{n,\nu}\int_{\mathbb{R}^2}\int_{\mathbb{R}}\int_{\mathbb{R}^2}
|K_{\lambda,\nu,n}^{\delta}(y,t;x)|dx\int_{\mathbb{R}^2}|K_{\lambda,\nu,n}^{\delta}(y,t;x)||f_{n,\nu}(x)|^2dx|g(y,t)|dydt\\
&\leq  C\sup_{\|g\|_{L^2}=1}\sum_{n,\nu}\int_{\mathbb{R}^2}\int_{\mathbb{R}}\int_{\mathbb{R}^2}
|K_{\lambda,\nu,n}^{\delta}(y,t;x)||f_{n,\nu}(x)|^2dx|g(y,t)|dydt\\
&\leq  C\sup_{\|g\|_{L^2}=1}\sum_{n,\nu}\int_{\mathbb{R}^2}|f_{n,\nu}(x)|^2\int_{\mathbb{R}^2}\int_{\mathbb{R}}
|K_{\lambda,\nu,n}^{\delta}(y,t;x)||g(y,t)|dydtdx\\
&\leq  C'\sup_{\|g\|_{L^2}=1}\int_{\mathbb{R}^2}\sum_{n,\nu}\sum_{m\in J_{n,\nu}}|P_m^{\delta}f(x)|^2\sup_{n,\nu}\int_{\mathbb{R}^2}\int_{\mathbb{R}}
|K_{\lambda,\nu,n}^{\delta}(y,t;x)||g(y,t)|dydtdx\\
&\leq  C'\sup_{\|g\|_{L^2}=1}\int_{\mathbb{R}^2}\sum_{m\in \mathbb{Z}^2}\sum_{\{(n,\nu):m\in J_{n,\nu}\}}|P_m^{\delta}f(x)|^2
\sup_{n,\nu}\int_{\mathbb{R}^2}\int_{\mathbb{R}}
|K_{\lambda,\nu,n}^{\delta}(y,t;x)||g(y,t)|dydtdx\\
&\leq  C''\delta^{-1/2}\sup_{\|g\|_{L^2}=1}\biggl\|\biggl(\sum_{m\in \mathbb{Z}^2}|P_m^{\delta}f(x)|^2\biggl)^{1/2}\biggl\|_{L^4}^2\\
&\quad\times\biggl(\int_{\mathbb{R}^2}\sup_{\nu,n}\left\{\int_{\mathbb{R}^2}\int_{\mathbb{R}}
|K_{\lambda,\nu,n}^{\delta}(y,t;x)||g(y,t)|dydt\right\}^2dx\biggl)^{1/2}\\
&\leq  C'''\delta^{-1}(\log_2\frac{\lambda^{1/2}}{\delta})^2\|f\|_{L^4}^2.
\end{align*}

Therefore,
\begin{align*}
\|\mathcal{F}_{\lambda}^{\delta}f\|_{L^4}
&\leq C\delta^{-1/2}\lambda^{1/8}(\frac{\lambda}{\delta})^{\epsilon/2}[\log_2(\frac{\lambda}{\delta})^{1/2}]^{1/4}\log_2\frac{\lambda^{1/2}}{\delta}
\|f\|_{L^4}+C_N\lambda^{-N}\|f\|_{L^4}\\
&\leq C\lambda^{1/8+\epsilon_1}\delta^{-(1/2+\epsilon_2)}\|f\|_{L^4},
\end{align*}
where $\epsilon_1$ and $\epsilon_2$ are very small.

\begin{center}
\includegraphics[height=8cm]{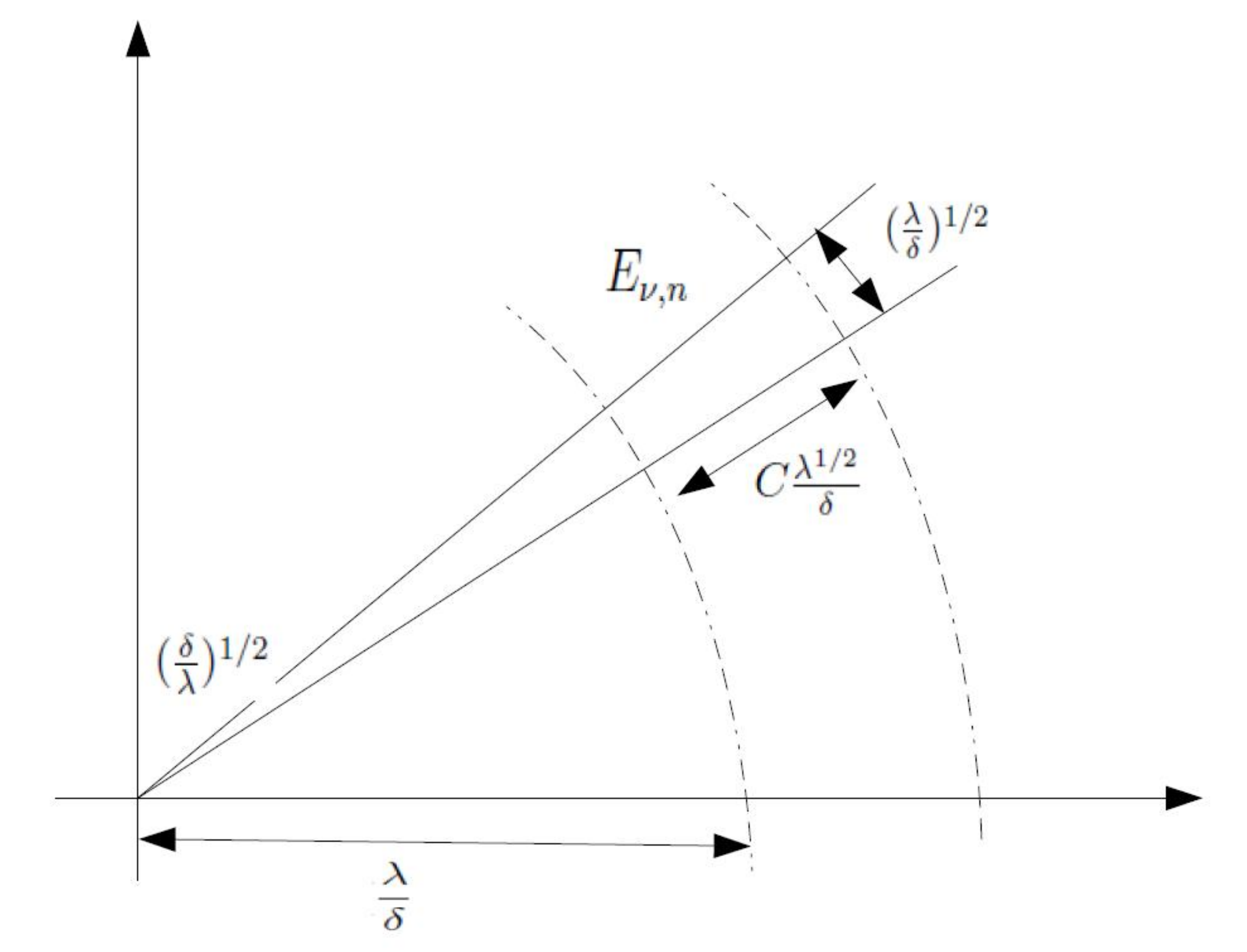}
\end{center}

\begin{center}
Figure 9 : $E_{\nu,n}$
\end{center}

We turn to prove Lemma \ref{lemma index2}. By (\ref{kernels}), we put
\begin{equation}
E_{\nu,n}:=\{\xi:\rho_0(\frac{\delta}{\lambda}|\xi|)\tilde{\chi}(\frac{\xi_1}{\xi_2})
a(\xi,t)\chi_{\nu}(\xi)\varphi ((\lambda^{-1/2}q_{\delta}(\xi)-n )/10)\neq 0\}
\end{equation}
and consider the size of $E_{\nu,n}$ (see Figure 9), first we can adopt some ideas from the proof of inequality (\ref{kernalL^infty}) to prove inequality (\ref{equ:planekerna12}), here we still use similar notations.

 After a rotation by the transformation $\tilde{T}$,  we put $\tilde{\tilde{\Psi}}(y,\xi)=\tilde{T}\xi\cdot y+E(\tilde{T}\xi)+tq_{\delta}(\tilde{T}\xi)$ and $h(\xi)=\tilde{\tilde{\Psi}}(y,\xi)-\tilde{\tilde{\Psi}}_{\xi}(y,\bar{\xi})\cdot\xi$, where $\bar{\xi}=(1,0)$. Since $\tilde{\tilde{\Psi}}$ is homogeneous of degree one, $|\xi|\approx \frac{\lambda}{\delta}$ and $|\xi_2|\leq C(\frac{\lambda}{\delta})^{1/2}$, then we can get some similar results with (\ref{h1}), (\ref{h2}), (\ref{anghomofunct2impr}),
\begin{equation}\label{h11}
\left|\left(\frac{\partial}{\partial\xi_1}\right)^Nh(\xi)\right|\lesssim (\frac{\lambda}{\delta})^{-N},\hspace{0.2cm}\left|\left(\frac{\partial}{\partial\xi_2}\right)^Nh(\xi)\right|\lesssim (\frac{\lambda}{\delta})^{-N/2},
\end{equation}
\begin{equation}\label{h12}
\left|\left(\frac{\partial}{\partial\xi_1}\right)^N(\chi_{\nu}\circ \tilde{T})(\xi)\right|\lesssim (\frac{\lambda}{\delta})^{-N},\hspace{0.2cm}\left|\left(\frac{\partial}{\partial\xi_2}\right)^N(\chi_{\nu}\circ \tilde{T})(\xi)\right|\lesssim (\frac{\lambda}{\delta})^{-N/2}, \hspace{0.2cm}N\geq 1.
\end{equation}

Furthermore,  set $\mathcal{L}:= I-\frac{\lambda}{\delta^2}\frac{\partial^2}{\partial\xi_1^2}-\frac{\lambda}{\delta}\frac{\partial^2}{\partial\xi_2^2}$ and $D_t(z):=(t^2z_1,tz_2)$. Because of (\ref{symbola}), (\ref{h11}), (\ref{h12}) and
$\mathcal{L}^N\left(\varphi((\lambda^{-1/2}q_{\delta}(T\xi)-n)/10)\right)\leq C$, then
$$\mathcal{L}^N\left(e^{ih(\xi)}\rho_0(\frac{\delta}{\lambda}|\xi|)\tilde{\chi}(\frac{\xi_1}{\xi_2})
a(\xi,t)\chi_{\nu}(\xi)\varphi ((\lambda^{-1/2}q_{\delta}(\xi)-n )/10)\right)\leq C.$$

Noting that the area of the region $E_{\nu,n}$ is at most $C\frac{\lambda}{\delta^{3/2}}$,  so by integration by parts in $\xi$, we obtain
\begin{equation}\label{equk1}
\begin{aligned}
|K_{\lambda,\nu,n}^{\delta}(y,t;x)|&\leq \frac{C_N'\lambda/\delta^{3/2}}{\left(1+\frac{\lambda}{\delta^2}|(\tilde{\tilde{\Psi}}_{\xi}(y-x,\bar{\xi}))_1|^2+\frac{\lambda}{\delta}|(\tilde{\tilde{\Psi}}_{\xi}(y-x,\bar{\xi}))_2|^2\right)^N}\\
&\leq\frac{C_N\lambda/\delta^{3/2}}{\left(1+| D_{\delta^{-1/2}}\lambda^{1/2}(\tilde{\tilde{\Psi}}_{\xi}(y-x,\bar{\xi}))|\right)^{2N}},
\end{aligned}
\end{equation}
where
$$|(\tilde{\tilde{\Psi}}_{\xi}(y-x,\bar{\xi}))_1|=\left|\langle y-x+\nabla E(\xi_{\nu})+t\nabla q_{\delta}(\xi_{\nu}),\xi_{\nu}\rangle\right|$$
and
$$|(\tilde{\tilde{\Psi}}_{\xi}(y-x,\bar{\xi}))_2|=\left|y-x+\nabla E(\xi_{\nu})+t\nabla q_{\delta}(\xi_{\nu})-\langle y-x+\nabla E(\xi_{\nu})+t\nabla q_{\delta}(\xi_{\nu}),\xi_{\nu}\rangle\xi_{\nu}\right|.$$

Inequality (\ref{equk1}) implies inequality  (\ref{equ:planekerna12}).

We introduce some notations. Given a direction $\xi_{\nu}=(\cos\theta_{\nu},\sin\theta_{\nu})\in S^1$, let $\gamma_{\theta_{\nu}}\subset\mathbb{R}^3$ be the ray defined by
\begin{equation}\label{gamma}
\gamma_{\theta_{\nu}}=\{(y,t):y+\nabla E(\cos\theta_{\nu},\sin\theta_{\nu})+t\nabla q_{\delta}(\cos\theta_{\nu},\sin\theta_{\nu})=0\}.
\end{equation}

From inequality (\ref{equk1}), we note that for fixed $x$, the kernels $K_{\lambda,\nu,n}^{\delta}(y,t;x)$ are essentially supported in a rectangle of size $\frac{\delta}{\lambda^{1/2}}\times(\frac{\delta}{\lambda})^{1/2}\times 1$ around $\gamma_{\theta_{\nu}}+(x,0)$, see Figure 10. This sheds some light on the proof of inequality (\ref{equ:planekerna22}).

We may assume $g\geq 0$ and choose non-negative functions $\varrho \in C_0^{\infty}([\pi/2-c_2,\pi/2-c_1])$ and $\beta_0 \in C_0^{\infty}([0,2])$ which satisfies $\beta_0(r)+\sum_{\ell=1}^{\infty}\beta(2^{-\ell}r)=1$, for $r>0$. Let $|\cdot|_{D}$ be a homogeneous norm under $D_t(z)=(t^2z_1,tz_2)$, i.e. $|D_t(z)|_{D}=t|z|_{D}$. Let $h_0(\xi)=\beta_0(|\xi|_{D})\in C_0^{\infty}(\mathbb{R}^2)$ and  $h(\xi)=\beta(|\xi|_{D})\in C_0^{\infty}(\mathbb{R}^2)$, then $h_0(\xi)+\sum_{\ell=1}^{\infty}h(D_{2^{-\ell}}\xi)=1$. It is clear that for $\ell>1$, $\xi\neq 0$,
$h(D_{2^{-\ell}}\xi)\neq 0 \Rightarrow 2^{\ell-1}\leq |\xi|_{D}\leq 2^{\ell+1}\Rightarrow |\xi|\gtrsim 2^{\ell} $.

Since the right side of inequality (\ref{equk1}) does not depend on $n$, then the left side of inequality (\ref{equ:planekerna22}) can be controlled by
\begin{align*}
&\biggl(\int_{\mathbb{R}^2}\sup_{\nu,n}\left\{\int_{\mathbb{R}^2}\int_{\mathbb{R}}|K_{\lambda,\nu,n}^{\delta}(y,t;x)||g(y,t)|dydt\right\}^2dx\biggl)^{1/2}\\
&\leq C_N'\biggl(\biggl\|\sup_{\nu}\frac{1}{\delta^{3/2}/\lambda}\int_{\mathbb{R}^2\times [1/2,4]}h_0( D_{\delta^{-1/2}} \lambda^{1/2}\tilde{\tilde{\Psi}}_{\xi}(y,\bar{\xi}))g(x+y,t) dydt\biggl\|_{L^2}\\
&\quad +
\biggl\|\sup_{\nu}\sum_{\ell=1}^{\infty}\int_{\mathbb{R}^2\times [1/2,4]}h(D_{2^{-\ell}}(D_{\delta^{-1/2}}\lambda^{1/2} \tilde{\tilde{\Psi}}_{\xi}(y,\bar{\xi})))\frac{\lambda/\delta^{3/2}g(x+y,t)}{\left(1+| D_{\delta^{-1/2}} \lambda^{1/2}\tilde{\tilde{\Psi}}_{\xi}(y,\bar{\xi})|\right)^{2N}} dydt\biggl\|_{L^2}\biggl)
\end{align*}
\begin{align*}
&\leq C''_N\sum_{\ell=0}^{\infty}2^{-2(N-3/2)\ell}\biggl\|\sup_{\nu}\frac{1}{2^{3(\ell+1)}\delta^{3/2}/\lambda}\int_{\{(y,t)\in\mathbb{R}^2\times [0,1]:|\lambda^{1/2}(y+\nabla E(\xi_{\nu})+t\nabla q_{\delta}(\xi_{\nu}))|_D\leq 2^{\ell+1}\delta^{1/2}\}}\\
&\quad \quad\quad\quad\quad\quad\quad\quad\times g(x+y,t)dydt\biggl\|_{L^2}\\
&\leq C''_N\sum_{\ell=0}^{\infty}2^{-2(N-3/2)\ell}\biggl\|\sup_{\nu}\frac{1}{|R_{\nu}^{\ell}|}\int_{R_{\nu}^{\ell}}g(x+y,t)dydt\biggl\|_{L^2},
\end{align*}
where
\begin{equation*}
R_{\nu}^{\ell}=\{(y,t)\in\mathbb{R}^2\times [1/2,4]:|\lambda^{1/2}(y+\nabla E(\xi_{\nu})+t\nabla q_{\delta}(\xi_{\nu}))|_D\leq 2^{\ell+1}\delta^{1/2}\}.
\end{equation*}

From the above argument, in order to prove inequality (\ref{equ:planekerna22}),
we need to prove an $L^2(\mathbb{R}^3)\rightarrow L^2(\mathbb{R}^2)$ maximal estimate involving averages over
cuboids $R_{\nu}^0$ of dimensions $\frac{\delta}{\lambda^{1/2}}\times(\frac{\delta}{\lambda})^{1/2}\times 1$,
which is basically tangential to the cone $\{(y,t)\in \gamma_{\theta}:\gamma_{\theta} \hspace{0.1cm}\textmd{is} \hspace{0.1cm} \textmd{defined}\hspace{0.1cm} \textmd{as } \hspace{0.1cm} (\ref{gamma}) \hspace{0.1cm}\textmd{and} \hspace{0.1cm}\theta\in \textmd{supp} \hspace{0.1cm}\rho\}$ (see Figure 10),
because in a similar way one can obtain $L^2(\mathbb{R}^3)\rightarrow L^2(\mathbb{R}^2)$ maximal estimates involving averages over $\{R_{\nu}^{\ell}\}_{\nu}$ which are $\frac{2^{2\ell}\delta}{\lambda^{1/2}}\times2^{\ell}(\frac{\delta}{\lambda})^{1/2}\times 1$ cuboids for $\ell=1,2,\cdots$.
However, we can split every $R_{\nu}^0$ along its longer side of size  $(\frac{\delta}{\lambda})^{1/2}$ into $\delta^{-1/2}$ pieces,
i.e. $\{R_{\nu,i}^0\}_{i=1}^{\delta^{-1/2}}$, where $R_{\nu,i}^0$ is  a $\frac{\delta}{\lambda^{1/2}}\times\frac{\delta}{\lambda^{1/2}}\times 1$ tube around $\gamma_{\theta_{\nu,i}}+(x,0)$ (see Figure 10). Then next we will prove a stronger maximal estimate involving averages over tubes $R_{\theta}$ defined by
\begin{equation}
R_{\theta}:=\{(y,t)\in\mathbb{R}^2\times [0,1]: dist\{(y,t),\gamma_{\theta}\}< \frac{\delta}{\lambda^{1/2}}\}.
\end{equation}
That is because
 \begin{align*}
\biggl\|\sup_{\nu}\frac{1}{|R_{\nu}^0|}\int_{R_{\nu}^0}g(x+y,t)dydt\biggl\|_{L^2}&\leq \delta^{1/2}\sum_{i=1}^{\delta^{-1/2}}\biggl\|\sup_{\nu}\frac{1}{|R_{\nu,i}^0|}\int_{R_{\nu,i}^0}g(x+y,t)dydt\biggl\|_{L^2}\\
&\leq \delta^{1/2}\sum_{i=1}^{\delta^{-1/2}}\biggl\|\sup_{\theta}\frac{1}{|R_{\theta}|}
\int_{R_{\theta}}\varrho(\theta)g(x+y,t)
dydt\biggl\|_{L^2}.
\end{align*}
Hence, we would be done if we can obtain the following Kakeya type estimate.

\begin{lem}
\begin{equation}\label{planelemma4}
\biggl\|\sup_{\theta}\frac{1}{|R_{\theta}|}
\int_{R_{\theta}}\varrho(\theta)g(x-y,t)
dydt\biggl\|_{L^2}\leq C\delta^{-1/2}(\log_2\frac{\lambda^{1/2}}{\delta})^2\|g\|_{L^2}.
\end{equation}
\end{lem}

\begin{center}
\includegraphics[height=17  cm]{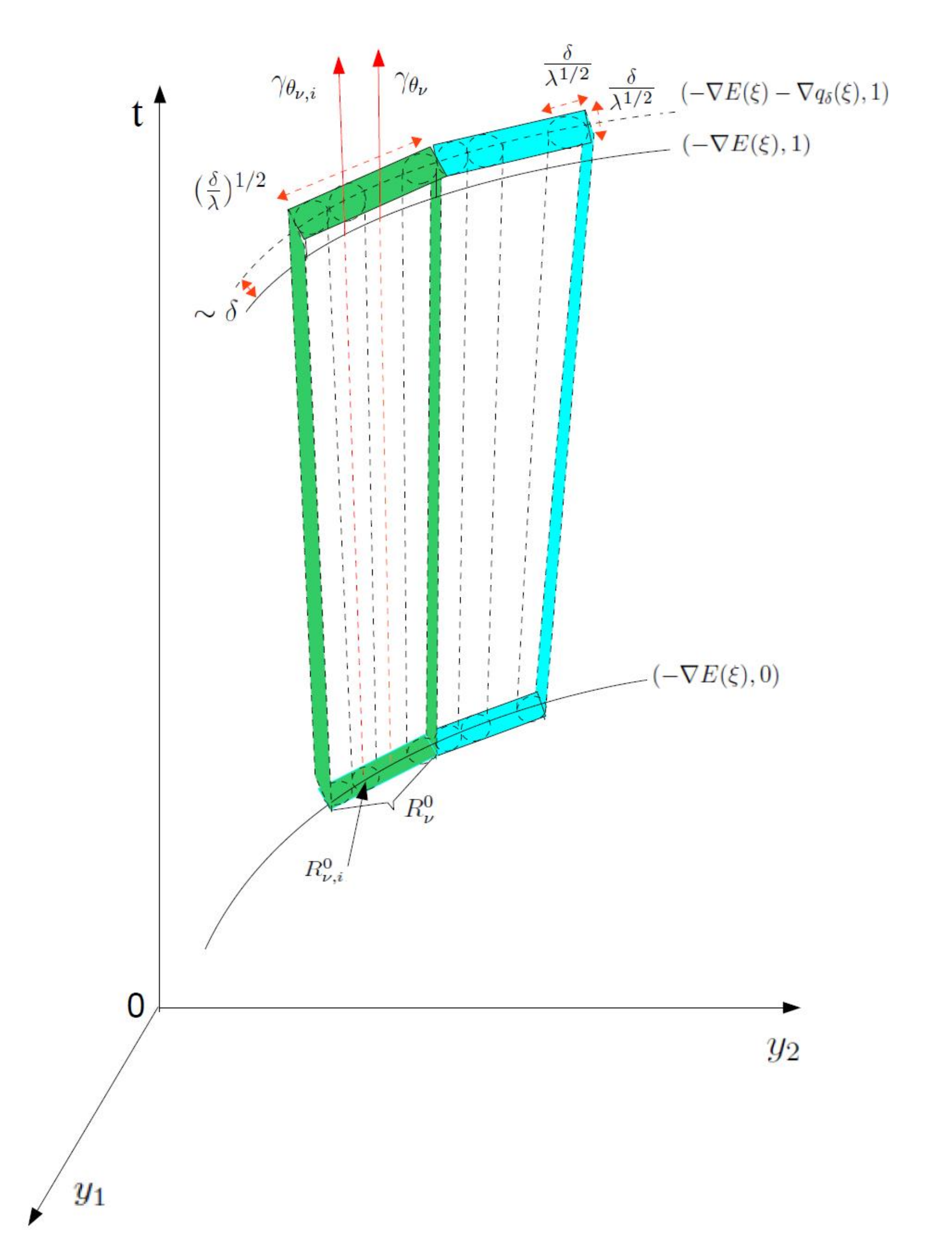}
\end{center}
\begin{center}
Figure 10 : cuboid $R_{\nu}^0$ and tube $R_{\nu,i}^0$
\end{center}

\begin{rem}

The constant $\delta^{-1/2}$ in inequality (\ref{planelemma4}) is sharp if we take $E(\xi)=|\xi|$ and $q_{\delta}(\xi)=\delta|\xi|$. Let
$$M_{\delta,\lambda}g(x)=\sup_{\theta}\frac{1}{|R_{\theta}|}
\int_{R_{\theta}}\varrho(\theta)g(x-y,t)
dydt.$$
In particular, we take $g(z,t)=\chi_{P_{\delta}}(z,t)$, where $P_{\delta}$ is a tubular neighborhood of width $\delta$ around the $t$-axis with height $1$. Assume that $x\in A_{\delta}:=\{x\in \mathbb{R}^2:1\leq |x|\leq 1+\delta,\hspace{0.3cm} c_1\leq |x_1/x_2|\leq c_2 \}$. Now taking $\theta$ so that $(\cos\theta,\sin\theta)=\frac{x}{|x|}$, then we note that
\begin{equation*}
 g*\chi_{R_{\theta}}(x)=|R_{\theta}\cap(P_{\delta}-(x,0))|\approx |R_{\theta}|,
 \end{equation*}
which implies $M_{\delta,\lambda}g(x) \gtrsim 1$ on $A_{\delta}$. Let $C_{\delta}:=\|M_{\delta,\lambda}\|_{L^2\rightarrow L^2}$. Since $|A_{\delta}|\approx \delta$, then
\begin{equation*}
\delta^{1/2}\lesssim\|M_{\delta,\lambda}g\|_{L^2}\leq C_{\delta}\delta \Rightarrow C_{\delta}\gtrsim\delta^{-1/2}.
 \end{equation*}

In fact, if we take $E(\xi)=\frac{\xi_1^2}{\xi_2}$ and $q_{\delta}(\xi)=\delta\frac{\xi_1^3}{\xi_2^2}$, in the similar way as above, we still get the same conclusion, regarding complicated computations, we omit here.
\end{rem}

\begin{proof}
We choose a suitable $\sigma\in C_0^{\infty}(\mathbb{R}^2)$ satisfying $\check{\sigma}\geq0$. Let $\bar{\delta}=\frac{\delta}{\lambda^{1/2}}$ and $\sigma_{\bar{\delta}}(\theta,t,\xi)=\varrho(\theta)\sigma(\bar{\delta}\xi)\chi_{[0,1]}(t)$. Then We have
\begin{align*}
&\int_{\{(y,t)\in \mathbb{R}^2\times [0,1]: dist\{(y,t),\gamma_{\theta}\} <\bar{\delta}\}}g(x-y,t)
dydt\\
&\leq  \int_{\mathbb{R}^2}\int_{\mathbb{R}}\varrho(\theta)\check{\sigma}\biggl(\frac{y+t\nabla q_{\delta}(\cos{\theta},\sin{\theta})+\nabla E(\cos{\theta},\sin{\theta})}{\bar{\delta}}\biggl)g(x-y,t)
\chi_{[0,1]}(t)dydt\\
&= {\bar{\delta}}^2\int_{\mathbb{R}^2}\int_{\mathbb{R}}\varrho(\theta)[\sigma(\bar{\delta}\cdot)]^{\vee}(y+t\nabla q_{\delta}(\cos{\theta},\sin{\theta})+\nabla E(\cos{\theta},\sin{\theta}))g(x-y,t)\\
&\quad \times
\chi_{[0,1]}(t)dydt\\
&= {\bar{\delta}}^2\int_{\mathbb{R}^2}\int_{\mathbb{R}}\int_{\mathbb{R}^2}e^{i\xi\cdot[x-y+t\nabla q_{\delta}(\cos{\theta},\sin{\theta})+\nabla E(\cos{\theta},\sin{\theta})]}\sigma_{\bar{\delta}}(\theta,t,\xi)d\xi g(y,t)dydt\\
&= {\bar{\delta}}^2\int_{\mathbb{R}^2}\int_{\mathbb{R}}e^{i\xi\cdot[x+t\nabla q_{\delta}(\cos{\theta},\sin{\theta})+\nabla E(\cos{\theta},\sin{\theta})]}\sigma_{\bar{\delta}}(\theta,t,\xi)g(\hat{\xi},t)d\xi dt\\
&=:{\bar{\delta}}^2A_{\theta}g(x),
\end{align*}
where $g(\hat{\xi},t)$ denotes the partial Fourier transform of $g$ with respect to the $\xi$-variables.

Therefore, we may further reduce to prove that
\begin{equation}
\|\sup_{\theta}|A_{\theta}g|\|_{L^2}\leq C\delta^{-1/2}|\log_2\bar{\delta}|^2\|g\|_{L^2}.
\end{equation}

In order to prove the above inequality, we need to break up the operator $A_{\theta}$. Just as before, we take $\beta \in C_0^{\infty}(\mathbb{R})$ and define the dyadic operators $A_{\theta}^{\tau}$ by
\begin{equation*}
A_{\theta}^{\tau}g(x)=\int_{\mathbb{R}^2}\int_{\mathbb{R}}e^{i\xi\cdot[x+t\nabla q_{\delta}(\cos{\theta},\sin{\theta})+\nabla E(\cos{\theta},\sin{\theta})]}\sigma_{\bar{\delta}}(\theta,t,\xi)\beta(\frac{|\xi|}{\tau})g(\hat{\xi},t)d\xi dt,
\end{equation*}
so it suffices to prove that
\begin{equation}\label{equ:planeA_theta}
\|\sup_{\theta}|A_{\theta}^{\tau}g|\|_{L^2}\leq C\delta^{-1/2}\log_2{\tau}^{1/2}\|g\|_{L^2},\hspace{0.2cm}\tau>2.
\end{equation}
This is because $A_{\theta}=\sum_{1<\kappa<|\log_2{\bar{\delta}}|+C}A_{\theta}^{2^{\kappa}}+\tilde{R}_{\theta}$, where $C$ is a fixed constant and  the kernel of $\tilde{R}_{\theta}$ defined by
\begin{equation*}
\int_{\mathbb{R}^2}e^{i\xi\cdot[(x-y)+t\nabla q_{\delta}(\cos{\theta},\sin{\theta})+\nabla E(\cos{\theta},\sin{\theta})]}\sigma_{\bar{\delta}}(\theta,t,\xi)\beta_0(|\xi|)d\xi
\end{equation*}
is controlled by  $\mathcal{O}((1+|y-x|)^{-N})$ for any $N$ with bounds independent of $\theta$.

Let $\theta_{\xi}=\arg\xi$. We need to make one final reduction based on the following observation:
\begin{equation}\label{compu}
\begin{aligned}
&\frac{\partial}{\partial \theta}\langle\xi,\nabla q_{\delta}(\cos{\theta},\sin{\theta})\rangle\\
&
=\left\langle\xi,(-\sin\theta,\cos\theta)\left(
\begin{array}{lcr}
\partial_{11}q_{\delta}(\cos\theta,\sin\theta)&\partial_{12}q_{\delta}(\cos\theta,\sin\theta)\\
\partial_{21}q_{\delta}(\cos\theta,\sin\theta)&\partial_{22}q_{\delta}(\cos\theta,\sin\theta)
\end{array}\right)\right\rangle\\
&=6\delta|\xi|\frac{\cos\theta}{\sin^4\theta}\langle\frac{\xi}{|\xi|},(-\sin\theta,\cos\theta)\rangle=6\delta|\xi|\frac{\cos\theta}{\sin^4\theta}\sin(\theta_{\xi}-\theta)=0\Leftrightarrow \theta=\theta_{\xi}.
\end{aligned}
\end{equation}

For $\ell>0$, put
\begin{align*}
A_{\theta}^{\tau,\ell}g(x)=&\int_{\mathbb{R}^2}\int_{\mathbb{R}}e^{i\xi\cdot[x+t\nabla q_{\delta}(\cos{\theta},\sin{\theta})+\nabla E(\cos{\theta},\sin{\theta})]}\sigma_{\bar{\delta}}(\theta,t,\xi)\beta(\frac{|\xi|}{\tau})\beta_{\ell,\tau}(\xi,\theta)g(\hat{\xi},t)d\xi dt,
\end{align*}
where $\beta_{\ell,\tau}(\xi,\theta)=\beta\left(2^{-\ell}\tau^{1/2}|\langle(-\sin\theta,\cos\theta),\frac{\xi}{|\xi|}\rangle|\right)$. In fact, $|\langle(-\sin\theta,\cos\theta),\frac{\xi}{|\xi|}\rangle|\leq C$, $C$ is very small, which implies that $\ell\leq C\log_2\tau^{1/2}$. We define $A_{\theta}^{\tau,0}=A_{\theta}^{\tau}-\sum_{0<\ell\leq C\log_2\tau^{1/2}}A_{\theta}^{\tau,\ell}$.

Again we use Lemma \ref{lem:Lemma3}, for  $\ell\geq 0$,
\begin{equation*}
\biggl\|\sup_{\theta}|A_{\theta}^{\tau,\ell}g|\biggl\|_{L^2}^2\leq C\biggl(\int\int|A_{\theta}^{\tau,\ell}g(x)|^2d\theta dx\biggl)^{1/2}
\biggl(\int\int|\frac{\partial}{\partial \theta}A_{\theta}^{\tau,\ell}g(x)|^2d\theta dx\biggl)^{1/2}.
\end{equation*}

Inequality (\ref{equ:planeA_theta}) will be obtained if we can show that
\begin{equation}\label{equ:planeAthetaell}
\biggl(\int\int|A_{\theta}^{\tau,\ell}g(x)|^2d\theta dx\biggl)^{1/2}\leq C\delta^{-1/2}2^{-\ell/2}\tau^{-1/4}\|g\|_{L^2},
\end{equation}
\begin{equation}\label{equ:planeAthetael2}
\biggl(\int\int|\frac{\partial}{\partial \theta}A_{\theta}^{\tau,\ell}g(x)|^2d\theta dx\biggl)^{1/2}\leq C\delta^{-1/2}2^{\ell/2}\tau^{1/4}\|g\|_{L^2}.
\end{equation}

From $\beta\left(2^{-\ell}\tau^{1/2}|\langle(-\sin\theta,\cos\theta),\frac{\xi}{|\xi|}\rangle|\right)\neq 0$ and $\beta(\frac{|\xi|}{\tau})\neq 0$, we know that on the support of the symbol of the operator $A_{\theta}^{\tau,\ell}$, $|\langle(-\sin\theta,\cos\theta),\frac{\xi}{|\xi|}\rangle|\approx 2^{\ell}\tau^{-1/2}$ and $|\xi|\approx \tau$.

Since
\begin{align*}
\frac{\partial}{\partial \theta}\xi\cdot\nabla E(\cos{\theta},\sin{\theta})&=\langle\xi,2(-\sin\theta,\cos\theta)\rangle \left(
\begin{array}{lcr}
\partial_{11}E(\cos\theta,\sin\theta)&\partial_{12}E(\cos\theta,\sin\theta)\\
\partial_{21}E(\cos\theta,\sin\theta)&\partial_{22}E(\cos\theta,\sin\theta)
\end{array}\right)\\
&=\langle\xi,2(-\sin\theta,\cos\theta)\rangle \left(
\begin{array}{lcr}
\frac{1}{\sin\theta}&-\frac{\cos\theta}{\sin^2\theta}\\
-\frac{\cos\theta}{\sin^2\theta}&\frac{\cos^2\theta}{\sin^3\theta}
\end{array}\right)\\
&=\frac{2}{\sin\theta}\langle\xi,(-\sin\theta,\cos\theta)\rangle \left(
\begin{array}{lcr}
1&-\frac{\cos\theta}{\sin\theta}\\
-\frac{\cos\theta}{\sin\theta}&\frac{\cos^2\theta}{\sin^2\theta}
\end{array}\right)\\
&=|\xi|\frac{2}{\sin^3\theta}\langle\frac{\xi}{|\xi|},(-\sin\theta,\cos\theta)\rangle,
\end{align*}
then
$$\left|\frac{\partial}{\partial \theta}\xi\cdot\biggl[t\nabla q_{\delta}(\cos{\theta},\sin{\theta})+\nabla E(\cos{\theta},\sin{\theta})\biggl]\right|\leq C2^{\ell}\tau^{1/2}.$$
It is easy to check that
$$\left|\frac{\partial}{\partial \theta}\beta\left(2^{-\ell}\tau^{1/2}|\langle(-\sin\theta,\cos\theta),\frac{\xi}{|\xi|}\rangle|\right)\right|\leq C2^{-\ell}\tau^{1/2},$$
one can see that $\frac{\partial}{\partial \theta}A_{\theta}^{\tau,\ell}$ behaves like $2^{\ell}\tau^{1/2}A_{\theta}^{\tau,\ell}$, then we only prove inequality (\ref{equ:planeAthetaell}).

Employing Plancherel's theorem, We have that
\begin{align*}
&\int\int|A_{\theta}^{\tau,\ell}g(x)|^2d\theta dx \\
&=\int_{\mathbb{R}^2}\int_{\mathbb{R}}A_{\theta}^{\tau,\ell}g(x)\overline{A_{\theta}^{\tau,\ell}g(x)}d\theta dx\\
&=\int_{\mathbb{R}^2}\int_{\mathbb{R}}\left(\int_{\mathbb{R}^2}\int_{\mathbb{R}}e^{i\xi\cdot[x+t\nabla q_{\delta}(\cos{\theta},\sin{\theta})+\nabla E(\cos{\theta},\sin{\theta})]}\sigma_{\bar{\delta}}(\theta,t,\xi)\beta(\frac{|\xi|}{\tau})\beta_{\ell,\tau}(\xi,\theta)g(\hat{\xi},t)d\xi dt\right)\\
&\quad\times \left(\int_{\mathbb{R}^2}\int_{\mathbb{R}}e^{-i\xi'\cdot[x+t'\nabla q_{\delta}(\cos{\theta},\sin{\theta})+\nabla E(\cos{\theta},\sin{\theta})]}\overline{\sigma_{\bar{\delta}}(\theta,t',\xi')\beta(\frac{|\xi'|}{\tau})\beta_{\ell,\tau}(\xi',\theta)g(\hat{\xi'},t')}d\xi' dt'\right)\\
&\quad\times d\theta dx\\
&=\int_{\mathbb{R}^2}\int_{\mathbb{R}}\int_{\mathbb{R}^2}\int_{\mathbb{R}}
\biggl(\int_{\mathbb{R}}e^{i[(t\xi-t'\xi')\cdot\nabla q_{\delta}(\cos{\theta},\sin{\theta})+(\xi-\xi')\cdot\nabla E(\cos{\theta},\sin{\theta})]}
\sigma_{\bar{\delta}}(\theta,t,\xi)\overline{\sigma_{\bar{\delta}}(\theta,t',\xi')}\\
&\quad\times\beta_{\ell,\tau}(\xi,\theta)\overline{\beta_{\ell,\tau}(\xi',\theta)}
d\theta\biggl)
\delta_0(\xi'-\xi)\beta(\frac{|\xi|}{\tau})\overline{\beta(\frac{|\xi'|}{\tau})}g(\hat{\xi},t)\overline{g(\hat{\xi'},t')
}d\xi dt d\xi' dt'\\
&=\int_{\mathbb{R}^2}\int_0^1\int_0^1
\left(\int_{\mathbb{R}}e^{i[(t-t')\xi\cdot\nabla q_{\delta}(\cos{\theta},\sin{\theta})]}|\rho(\theta)\beta_{\ell,\tau}(\xi,\theta)|^2
d\theta\right)|\sigma(\bar{\delta}\xi)\beta(\frac{|\xi|}{\tau})|^2g(\hat{\xi},t)\\
&\quad\times \overline{g(\hat{\xi},t')}
d\xi dt dt',
\end{align*}
where $\delta_0$ denotes the two-dimensional Dirac delta function.

Hence
\begin{align*}
\int\int|A_{\theta}^{\tau,\ell}g(x)|^2d\theta dx=\int_{\mathbb{R}^2}\biggl\{\int_0^1\int_0^1 H^{\ell,\tau}(t,t',\xi)\left|\beta(\frac{|\xi|}{\tau})\sigma(\bar{\delta}\xi)\right|^2g(\hat{\xi},t)\overline{g(\hat{\xi},t')}dtdt'\biggl\}d\xi,
\end{align*}
where
\begin{align*}
H^{\ell,\tau}(t,t',\xi)=\int_{\mathbb{R}}e^{i(t-t')\langle\xi,\nabla q_{\delta}(\cos{\theta},\sin{\theta})\rangle}|\varrho(\theta)\beta_{\ell,\tau}(\xi,\theta)|^2d\theta.
\end{align*}

First we claim that for $\ell>0$,
\begin{equation}\label{claim1}
|H^{\ell,\tau}(t,t',\xi)|\leq C2^{\ell}\tau^{-1/2}(1+\delta 2^{2\ell}|t-t'|)^{-N},\hspace{0.2cm}|\xi|\approx \tau.
\end{equation}

Since  $\langle(-\sin\theta,\cos\theta),\frac{\xi}{|\xi|}\rangle=\sin(\theta_{\xi}-\theta)$ and $|\theta_{\xi}-\theta|\approx 2^{\ell}\tau^{-1/2}$ on supp $\beta_{\ell,\tau}$, we make the change of variable $\theta=\theta_{\xi}+\arcsin(2^{\ell}\tau^{-1/2}\omega)$, then
\begin{align*}
H^{\ell,\tau}(t,t',\xi)&=\int_{\mathbb{R}}e^{i(t-t')\left\langle\xi,\nabla q_{\delta}\left(\cos({\theta_{\xi}+\arcsin(2^{\ell}\tau^{-1/2}\omega)}),\sin(\theta_{\xi}+\arcsin(2^{\ell}\tau^{-1/2}\omega)\right)\right\rangle}\\
&\quad\times |\varrho(\theta_{\xi}+\arcsin(2^{\ell}\tau^{-1/2}\omega))
\beta(|\omega|)|^2\frac{2^{\ell}\tau^{-1/2}}{\sqrt{1-2^{2\ell}\tau^{-1}\omega^2}}d\omega.
\end{align*}

In a similar way as (\ref{compu}) we have that
\begin{align*}
&\frac{\partial}{\partial\omega}\left\langle\xi,\nabla q_{\delta}\left(\cos(\theta_{\xi}+\arcsin(2^{\ell}\tau^{-1/2}\omega)),\sin(\theta_{\xi}+\arcsin(2^{\ell}\tau^{-1/2}\omega)\right)\right\rangle\\
&=\frac{62^{\ell}\tau^{-1/2}}{\sqrt{1-2^{2\ell}\tau^{-1}\omega^2}}\delta|\xi|\frac{\cos(\theta_{\xi}+\arcsin(2^{\ell}\tau^{-1/2}\omega))}{\sin^4(\theta_{\xi}+\arcsin(2^{\ell}\tau^{-1/2}\omega))}\\
&\quad\times
\biggl\langle\biggl(-\sin(\theta_{\xi}+\arcsin(2^{\ell}\tau^{-1/2}\omega)), \cos(\theta_{\xi}+\arcsin(2^{\ell}\tau^{-1/2}\omega))\biggl),\frac{\xi}{|\xi|}\biggl\rangle
\\
&=\frac{62^{\ell}\tau^{-1/2}\delta|\xi|}{\sqrt{1-2^{2\ell}\tau^{-1}\omega^2}}\frac{\cos(\theta_{\xi}+\arcsin(2^{\ell}\tau^{-1/2}\omega))}{\sin^4(\theta_{\xi}+\arcsin(2^{\ell}\tau^{-1/2}\omega))}
\langle\sin(\theta_{\xi}-(\theta_{\xi}+\arcsin(2^{\ell}\tau^{-1/2}\omega)))\rangle
\\
&=\frac{-62^{2\ell}\tau^{-1}\delta|\xi|\omega}{\sqrt{1-2^{2\ell}\tau^{-1}\omega^2}}\frac{\cos(\theta_{\xi}+\arcsin(2^{\ell}\tau^{-1/2}\omega))}{\sin^4(\theta_{\xi}+\arcsin(2^{\ell}\tau^{-1/2}\omega))}.
\end{align*}
Since the support of $\rho$ and $\beta$ gives that
$\theta_{\xi}+\arcsin(2^{\ell}\tau^{-1/2}\omega)\approx 1$ and $|\omega|\approx 1$, in addition to  $|\xi|\approx \tau$ and $\ell\leq C \log_2\tau^{1/2}$, then
\begin{align*}
&\frac{\partial}{\partial\omega}\left\langle\xi,\nabla q_{\delta}\left(\cos(\theta_{\xi}+\arcsin(2^{\ell}\tau^{-1/2}\omega)),\sin(\theta_{\xi}+\arcsin(2^{\ell}\tau^{-1/2}\omega)\right)\right\rangle\approx 2^{2\ell}\delta.
\end{align*}

The claim (\ref{claim1}) will follow from integration by parts. Based on the above claim,  H\"{o}lder's inequality and Plancherel's theorem, we have
\begin{align*}
&\int\int|A_{\theta}^{\tau,\ell}g(y)|^2d\theta dy\\
&\leq C_N2^{\ell}\tau^{-1/2}\int_{\mathbb{R}^2}\biggl\{\int_0^1\int_0^1 \frac{|\overline{g(\hat{\xi},t')}|}{(1+\delta2^{2\ell}|t-t'|)^N}dt' |g(\hat{\xi},t)|dt\biggl\} |\beta(\frac{|\xi|}{\tau})\sigma(\bar{\delta}\xi)|^2d\xi\\
&\leq C_N2^{\ell}\tau^{-1/2}\int_{\mathbb{R}^2}\biggl\{\int_0^1|\overline{g(\hat{\xi},\cdot)}|*(1+\delta2^{2\ell}|\cdot|)^{-N}(t) |g(\hat{\xi},t)|dt\biggl\}
|\beta(\frac{|\xi|}{\tau})\sigma(\bar{\delta}\xi)|^2d\xi\\
 &\leq C_N2^{\ell}\tau^{-1/2}\int_{\mathbb{R}^2}\biggl\||\overline{g(\hat{\xi},\cdot)}|*(1+\delta2^{2\ell}|\cdot|)^{-N}\biggl\|_{(L^2,dt)} \|g(\hat{\xi},t)\|_{(L^2,dt)} |\beta(\frac{|\xi|}{\tau})\sigma(\bar{\delta}\xi)|^2d\xi\\
 &\leq C_N2^{-\ell}\tau^{-1/2}\delta^{-1}\|g\|_{L^2}^2,
\end{align*}
and inequality (\ref{equ:planeAthetaell}) has been proved for $\ell>0$.

For $\ell=0$, the proof is simpler. We still make the change of variable  $\theta\rightarrow \theta_{\xi}+\arcsin(\tau^{-1/2}\omega)$, then the support of the symbol of the operator $H^{0,\tau}$ gives that $0$ is the non-degenerate critical point, since
\begin{align*}
&\frac{\partial^2}{\partial\omega^2}\langle\xi,\nabla q_{\delta}(\cos(\theta_{\xi}+\arcsin(\tau^{-1/2}\omega)),\sin(\theta_{\xi}+\arcsin(\tau^{-1/2}\omega)))\rangle\biggl|_{\omega=0}\\
&=-6\delta|\xi|\tau^{-1}\frac{\cos\theta_{\xi}}{\sin^4\theta_{\xi}}\approx \delta.
\end{align*}
Finally, integration by parts implies that
\begin{equation}
|H^{0,\tau}(t,t',\xi)|\leq C\tau^{-1/2}(1+\delta|t-t'|)^{-1/2}.
\end{equation}
We still use H\"{o}lder's inequality and Plancherel's theorem to get inequality (\ref{equ:planeAthetaell}) for $\ell=0$.
\end{proof}

\section[More general Fourier integral operators not satisfying the  cinematic curvature condition uniformly]{More general Fourier integral operators not satisfying the  cinematic curvature condition uniformly.}

The purpose of  this section is to prove Theorem \ref{L^4therom}. The main idea follows from the proof of Theorem 6.1 given in \cite{mss}. In fact, the proof has a similar structure with Section 3. Instead of repeating the proof here, we give a brief overview, and only a detailed argument partly. In this section, we still assume $\delta\lambda>1$ and the $\xi$-support of the symbol $a$ is in the first quadrant.

Since $t\approx 1$, we can replace $t$ by $1/t$ in (\ref{estimate:L4}). By (\ref{phase function}),  we write
\begin{equation*}
-t^{-2}\xi_2\Phi(s,\tilde{q}(s,\delta),\delta):=E(\xi)+q_{\delta}(\xi,t),\hspace{0.3cm}s=s(\xi,t^{-1}),
\end{equation*}
where
\begin{equation}\label{phasef}
E(\xi)=\frac{\xi_1^2}{2\xi_2}, \hspace{0.5cm}q_{\delta}(\xi,t)=t\delta\phi'(0)\frac{\xi_1^3}{\xi_2^2}+\delta^2 R(t,\delta,\xi),
\end{equation}
by abuse of notation, we have written $R(t,\delta,\xi)$ in place of $R(t^{-1},\delta,\xi)$ in (\ref{phasef}).

Then
\begin{equation}
 \tilde{F}_{\lambda}f(y,t):=\tilde{F}_{\lambda}^{\delta}f(y,t)=\int_{{\mathbb{R}}^2}e^{i(\xi \cdot y+E(\xi)+q_{\delta}(\xi,t))}\tilde{a}_{\lambda}(y,t,\xi)\hat{f}(\xi)d\xi,
\end{equation}
where $\tilde{a}_{\lambda}(y,t,\xi)=\rho_1(y,t)a(\xi,t)\rho_0(\frac{\delta}{\lambda}|\xi|)\tilde{\chi}(\frac{\xi_1}{\xi_2})$, which is a symbol of order zero in $\xi$.

The support of $\tilde{\chi}$ implies that  we are working in a fixed, small conic region $c_1\leq |\xi_1/\xi_2|\leq c_2$. Using the angular decomposition appeared in Chapter 3,  we can write  $\tilde{F}_{\lambda}=\sum_{\nu}F_{\lambda}^{\nu}$, where the sum runs only over $(\frac{\pi}{2}-c_2)\sqrt{\frac{\lambda}{\delta}}\lesssim \nu \lesssim (\frac{\pi}{2}-c_1)\sqrt{\frac{\lambda}{\delta}}$ and
\begin{equation}
 \tilde{F}_{\lambda}^{\nu}f(y,t):=\int_{{\mathbb{R}}^2}e^{i(\xi \cdot y+E(\xi)+q_{\delta}(\xi,t))}\tilde{a}_{\lambda}(y,t,\xi)\chi_{\nu}(\xi)\hat{f}(\xi)d\xi.
\end{equation}

By Minkowski's inequality, we have
\begin{align*}
\|\tilde{F}_{\lambda}f\|_{L^4}^2&=\|\sum_{\nu,\mu}\tilde{F}_{\lambda}^{\nu}f\tilde{F}_{\lambda}^{\mu}f\|_{L^2}\\
& =\|\sum_{\ell}\sum_{|\nu-\mu|\approx2^{\ell}}\tilde{F}_{\lambda}^{\nu}f\tilde{F}_{\lambda}^{\mu}f\|_{L^2}\\
&\leq C\sum _{\ell}\biggl\|\sum_{|\nu-\mu|\approx2^{\ell},\nu\leq \mu}\tilde{F}_{\lambda}^{\nu}f\tilde{F}_{\lambda}^{\mu}f\biggl\|_{L^2},
\end{align*}
where $2^{\ell}\leq (c_2-c_1)\sqrt{\frac{\lambda}{\delta}}$.

Now we make a further decomposition so that the symbol becomes
\begin{equation}
\tilde{a}_{\lambda,\ell}^{\nu,j}(y,t,\xi)=\tilde{a}_{\lambda}(y,t,\xi)\chi_{\nu}(\xi)\varphi(2^{\ell}\lambda^{-1}q_{\delta}'(\xi,t)-j),
\end{equation}
where $\varphi$ is defined as in Chapter 3 and $q_{\delta}'(\xi,t)=\partial_tq_{\delta}(\xi,t)$. Since $|\xi|\approx \frac{\lambda}{\delta}$, then we have $\tilde{F}_{\lambda}^{\nu}=\sum_{j\approx 2^{\ell}}\tilde{F}_{\lambda,\ell}^{\nu,j}$, where
\begin{equation*}
\tilde{F}_{\lambda,\ell}^{\nu,j}f(y,t)=\int_{{\mathbb{R}}^2}e^{i(\xi \cdot y+E(\xi)+q_{\delta}(\xi,t))}\tilde{a}_{\lambda,\ell}^{\nu,j}(y,t,\xi)\hat{f}(\xi)d\xi.
\end{equation*}

Since $c_1\leq |\xi_1/\xi_2|\leq c_2$ and $q_{\delta}'(\xi,t)=\delta\phi'(0)\frac{\xi_1^3}{\xi_2^2}+\delta^2\partial_tR(t,\delta,\xi)$, then $2^{-\ell}\lambda(j-1)\lesssim \delta |\xi|\lesssim 2^{-\ell}\lambda(j+1)$. The support of $\chi_{\nu}$ implies that  the $\xi$-support of the symbol is comparable to a $ 2^{-\ell}\frac{\lambda}{\delta}\times {(\frac{\lambda}{\delta})}^{1/2}$ rectangle.  Meanwhile, by (\ref{II}) and (\ref{phasef}), we notice that
\begin{equation}\label{symbol1}
|\partial_{\xi}^{\alpha}\partial_{y,t}^{\beta}\tilde{a}_{\lambda,\ell}^{\nu,j}(y,t,\xi)|\leq C_{\alpha,\beta}(1+|\xi|)^{-|\alpha|/2}(\delta\lambda)^{|\beta|/2}.
\end{equation}

Now we use two almost orthogonal lemmas. One of them is as follows
\begin{lem}\label{othogonal1}
Suppose that $|j+k-j'-k'|\geq \lambda^{\varepsilon}$. Then for any $N>0$,
\begin{equation}
\left|\int_{\mathbb{R}^3}\tilde{F}_{\lambda,\ell}^{\nu,j}f\tilde{F}_{\lambda,\ell}^{\mu,k}f\overline{\tilde{F}_{\lambda,\ell}^{\nu',j'}f\tilde{F}_{\lambda,\ell}^{\mu',k'}f}dydt\right|
\leq C_{\varepsilon,N}\lambda^{-N}\|f\|_{L^4}^4.
\end{equation}
\end{lem}
\begin{proof}
\begin{equation}
\int_{\mathbb{R}^3}\tilde{F}_{\lambda,\ell}^{\nu,j}f\tilde{F}_{\lambda,\ell}^{\mu,k}f\overline{\tilde{F}_{\lambda,\ell}^{\nu',j'}f\tilde{F}_{\lambda,\ell}^{\mu',k'}f}dydt
=\int H_{j,k,j',k'}^{\nu,\mu,\nu',\mu'}(\eta,\xi,\eta',\xi')\hat{f}(\eta)\hat{f}(\xi)\overline{\hat{f}(\eta')\hat{f}(\xi')}d\eta d\xi d\eta' d\xi',
\end{equation}
where
\begin{equation}H_{j,k,j',k'}^{\nu,\mu,\nu',\mu'}(\eta,\xi,\eta',\xi')
=\int_{\mathbb{R}^3}e^{i\Psi(y,t,\eta,\xi,\eta',\xi')}b_{j,k,j',k'}^{\nu,\mu,\nu',\mu'}(y,t,\eta,\xi,\eta',\xi')dydt,
\end{equation}
$\Psi(y,t,\eta,\xi,\eta',\xi')=y\cdot(\eta+\xi-\eta'-\xi')+\frac{1}{2}(\frac{\eta_1^2}{\eta_2}+
\frac{\xi_1^2}{\xi_2}-\frac{\eta_1'^2}{\eta_2'}-\frac{\eta_1'^2}{\eta_2'})+F(t,\eta,\xi,\eta',\xi')$, $F(t,\eta,\xi,\eta',\xi')=q_{\delta}(\eta,t)+q_{\delta}(\xi,t)-q_{\delta}(\eta',t)-q_{\delta}(\xi',t)$ and
\begin{equation}
b_{j,k,j',k'}^{\nu,\mu,\nu',\mu'}(y,t,\eta,\xi,\eta',\xi')=\tilde{a}_{\lambda,\ell}^{\nu,j}(y,t,\eta)
\tilde{a}_{\lambda,\ell}^{\mu,k}(y,t,\xi)\overline{\tilde{a}_{\lambda,\ell}^{\nu',j'}(y,t,\eta')\tilde{a}_{\lambda,\ell}^{\mu',k'}(y,t,\xi')}.
\end{equation}

Let $\mathcal{L}g =\frac{\partial}{\partial_t}(\frac{g}{\partial_t\Psi})$.  By integration by parts, we have
\begin{equation}
|H_{j,k,j',k'}^{\nu,\mu,\nu',\mu'}(\eta,\xi,\eta',\xi')|
\leq \int_{\mathbb{R}^3}\left|\mathcal{L}^Nb_{j,k,j',k'}^{\nu,\mu,\nu',\mu'}(y,t,\eta,\xi,\eta',\xi')\right|dydt.
\end{equation}

Since $\partial_t\Psi=\partial_tF$, $\mathcal{L}g=\frac{\partial}{\partial_t}(\frac{g}{\partial_tF})$, then the inner integral can be written as the sum of expressions of the form
\begin{equation}\label{form}
\frac{\prod_{i=0}^{N_2}\left(\Psi_i(t,\eta)+\Psi_i(t,\xi)-\Psi_i(t,\eta')-\Psi_i(t,\xi')\right)}{(\partial_tF)^{N+N_1}}\partial_t^{\beta}b_{j,k,j',k'}^{\nu,\mu,\nu',\mu'}(y,t,\eta,\xi,\eta',\xi'),
\end{equation}
where $0\leq N_1\leq N$, $0\leq N_2\leq N_1$, $\beta\leq N-N_2$, and $\Psi_i$ are partial derivative of $\partial_tF$ with respect to $t$-variable. Note that $\Psi_i$ are still homogeneous of degree one in $\xi$.

Now from the estimate (\ref{symbol1}) and the support of $\varphi$, it is easy to see that
\begin{align*}
|\partial_t F(t,\eta,\xi,\eta',\xi')|&\geq C2^{-\ell}\lambda|j+k-j'-k'|\geq C (\delta\lambda)^{1/2}|j+k-j'-k'|.
\end{align*}

Since $|\Psi_i(t,\cdot)|\leq |\partial_t^{2+\alpha}q_{\delta}(\cdot,t)|\leq \delta^2\frac{\lambda}{\delta}=\delta\lambda$ for all $\alpha\geq 0$, then $$\left|\prod_{i=0}^{N_2}(\Psi_i(t,\eta)+\Psi_i(t,\xi)-\Psi_i(t,\eta')-\Psi_i(t,\xi'))\right|\leq C(\delta\lambda)^{N_2+1},$$
in addition to the fact that
\begin{equation}\label{b}
|\partial_tb_{j,k,j',k'}^{\nu,\mu,\nu',\mu'}|\leq 2^{\ell}\lambda^{-1}\delta^2\frac{\lambda}{\delta}=2^{\ell}\delta\leq (\frac{\lambda}{\delta})^{1/2}\delta=(\delta\lambda)^{1/2},
\end{equation}
we have
\begin{align*}
\left|H_{j,k,j',k'}^{\nu,\mu,\nu',\mu'}(y,\eta,\xi,\eta',\xi')\right|&\leq C\frac{(\delta\lambda)^{N_2+1}(\delta\lambda)^{(N-N_2)/2}}
{\left((\delta\lambda)^{1/2}|j+k-j'-k'|\right)^{N+N_1}}\leq C \lambda^{-\varepsilon N}(\delta\lambda),
\end{align*}
provided that  $|j+k-j'-k'|\geq \lambda^{\varepsilon}$. Since $b_{j,k,j',k'}^{\nu,\mu,\nu',\mu'}$ is supported on a set of measure $2^{-4\ell}(\frac{\lambda}{\delta})^6$ and $\delta^{-1}<\lambda$, then the integral is bounded by $C_{\varepsilon, N}\lambda^{-N}\|f\|_{L^1}^4$, in addition to the compact support of $f$, we get the desired result.
\end{proof}

The above lemma implies that
\begin{equation}\label{1}
\begin{aligned}
&\int\biggl|\sum_{|\nu-\mu|\approx2^{\ell}}\tilde{F}_{\lambda}^{\nu}f\tilde{F}_{\lambda}^{\mu}f\biggl|^2dydt\\
&\leq \int\biggl|\sum_{|\nu-\mu|\approx2^{\ell}}\sum_{j,k\leq 2^{\ell}}\tilde{F}_{\lambda,\ell}^{\nu,j}f\tilde{F}_{\lambda,\ell}^{\mu,k}f\biggl|^2dydt\\
&=\int\biggl|\sum_{s\leq 2^{\ell+1}}\sum_{j+k=s}\sum_{|\nu-\mu|\approx2^{\ell}}\tilde{F}_{\lambda,\ell}^{\nu,j}f\tilde{F}_{\lambda,\ell}^{\mu,k}f\biggl|^2dydt\\
&=\int\biggl|\sum_{\substack{s\leq 2^{\ell+1}\\s'\leq 2^{\ell+1}}}\sum_{\substack{j+k=s\\j'+k'=s'}}\sum_{\substack{|\nu-\mu|\approx2^{\ell}\\|\nu'-\mu'|\approx2^{\ell}}}\tilde{F}_{\lambda,\ell}^{\nu,j}f
\tilde{F}_{\lambda,\ell}^{\mu,k}f\overline{\tilde{F}_{\lambda,\ell}^{\nu',j'}f\tilde{F}_{\lambda,\ell}^{\mu',k'}f}\biggl|dydt\\
&\leq \int\biggl|\sum_{s,s',|s-s'|\leq \lambda^{\varepsilon}}
\sum_{\substack{j+k=s\\j'+k'=s'}}\sum_{\substack{|\nu-\mu|\approx2^{\ell}\\|\nu'-\mu'|\approx2^{\ell}}}\tilde{F}_{\lambda,\ell}^{\nu,j}f
\tilde{F}_{\lambda,\ell}^{\mu,k}f\overline{\tilde{F}_{\lambda,\ell}^{\nu',j'}f\tilde{F}_{\lambda,\ell}^{\mu',k'}f}\biggl|dydt+C\lambda^{-N}\|f\|_{L^4}^4.
\end{aligned}
\end{equation}
By the Cauchy-Schwarz inequality, we have that
\begin{equation}\label{2}
\begin{aligned}
&\int\biggl|\sum_{s,s',|s-s'|\leq \lambda^{\varepsilon}}
\sum_{\substack{j+k=s\\j'+k'=s'}}\sum_{\substack{|\nu-\mu|\approx2^{\ell}\\|\nu'-\mu'|\approx2^{\ell}}}\tilde{F}_{\lambda,\ell}^{\nu,j}f
\tilde{F}_{\lambda,\ell}^{\mu,k}f\overline{\tilde{F}_{\lambda,\ell}^{\nu',j'}f\tilde{F}_{\lambda,\ell}^{\mu',k'}f}\biggl|dydt\\
&\leq \lambda^{\varepsilon}\sum_{s\leq 2^{\ell+1}}\int\biggl|\sum_{\substack{j+k=s\\j'+k'=s}}\sum_{\substack{|\nu-\mu|\approx2^{\ell}\\|\nu'-\mu'|\approx2^{\ell}}}\tilde{F}_{\lambda,\ell}^{\nu,j}f
\tilde{F}_{\lambda,\ell}^{\mu,k}f\overline{\tilde{F}_{\lambda,\ell}^{\nu',j'}f\tilde{F}_{\lambda,\ell}^{\mu',k'}f}\biggl|dydt\\
&=\lambda^{\varepsilon}\sum_{s\leq 2^{\ell+1}}\int\biggl|\sum_{j+k=s}\sum_{|\nu-\mu|\approx2^{\ell}}\tilde{F}_{\lambda,\ell}^{\nu,j}f
\tilde{F}_{\lambda,\ell}^{\mu,k}f\biggl|^2dydt\\
&\leq C\lambda^{\varepsilon}2^{\ell+1}\sum_{s\leq 2^{\ell+1}}\int\sum_{j+k=s}\biggl|\sum_{|\nu-\mu|\approx2^{\ell}}\tilde{F}_{\lambda,\ell}^{\nu,j}f
\tilde{F}_{\lambda,\ell}^{\mu,k}f\biggl|^2dydt\\
&\leq C\lambda^{\varepsilon}2^{\ell+1}\biggl\|\left(\sum_{j,k}|\sum_{|\nu-\mu|\approx2^{\ell}}\tilde{F}_{\lambda,\ell}^{\nu,j}f
\tilde{F}_{\lambda,\ell}^{\mu,k}f|^2\right)^{1/2}\biggl\|_{L^2}^2.
\end{aligned}
\end{equation}

Now we use another almost orthogonality lemma to bring the sum in $\mu$ and $\nu$ outside of the square function.
\begin{lem}\label{othogonal2}
Suppose that $|\nu-\mu|\approx2^{\ell}$, $|\nu'-\mu'|\approx2^{\ell}$, and $|\nu-\nu'|+|\mu-\mu'|\geq \lambda^{\varepsilon}$. Then for any $N>0$,
\begin{equation}
\left|\int_{\mathbb{R}^3}\tilde{F}_{\lambda,\ell}^{\nu,j}f\tilde{F}_{\lambda,\ell}^{\mu,k}f\overline{\tilde{F}_{\lambda,\ell}^{\nu',j'}f\tilde{F}_{\lambda,\ell}^{\mu',k'}f}dydt\right|
\leq C_{\varepsilon,N}\lambda^{-N}\|f\|_{L^4}^4.
\end{equation}
\end{lem}
\begin{proof}
First
we will introduce  Lemma 6.8 in \cite{mss} as follows.
\begin{lem}\cite{mss}
Let $\tilde{\Psi}\in C^4(\mathbb{R}^2\backslash 0)$ be homogeneous of degree one. Let $\varsigma<\pi/4$, $a_0<1/4$, $A_0\geq 1$. Let $\mathcal{S}_{\lambda}$ be the intersection of a sector which subtends an angle of size $\varsigma$ with the annulus $\{\eta:(1-a_0)\lambda\leq |\eta|\leq (1+a_0)\lambda\}$. Let $h\in C^1(\mathbb{R}^2\backslash 0)$ be homogeneous of degree one such that $b_0|\eta|\leq h(\eta)\leq b_1|\eta|$, $|\nabla_{\eta}h(\eta)|\leq b_2$ for some positive constants $b_0$, $b_1$, $b_2$.

Suppose that $a_0^{-1}\leq 2^n$, $n\leq \ell$, $2^{\ell}\leq \varsigma \lambda^{1/2}$, and that $\eta$, $\xi$, $\eta'$, $\xi'\in  \mathcal{S}_{\lambda}$ are chosen such that for given integers $\nu$, $\mu$, $\nu'$, $\mu'$\\
(1) $|\arg(\eta)-\nu\lambda^{-1/2}|\leq \lambda^{-1/2}$; $|\arg(\eta')-\nu'\lambda^{-1/2}|\leq \lambda^{-1/2}$;\\
(2) $|\arg(\eta)-\nu\lambda^{-1/2}|\leq \lambda^{-1/2}$; $|\arg(\eta')-\nu'\lambda^{-1/2}|\leq \lambda^{-1/2}$;\\
(3) $2^{\ell-1}\lambda^{-1/2}\leq \max\{|\arg(\eta)-\arg(\xi)|, |\arg(\eta')-\arg(\xi')|\}\leq 2^{\ell+1}\lambda^{-1/2}$;\\
(4) $|h(\eta)-h(\eta')|\leq 2^{-n}\lambda$;\\
(5) $|h(\xi)-h(\xi')|\leq 2^{-n}\lambda$.

Then one can choose $\varsigma$, $a_0$ sufficiently small, and $A_0$ sufficiently large (only depending on $\tilde{\Psi}$, $b_0$, $b_1$, $b_2$) such that for all $\nu$, $\mu$, $\nu'$, $\mu'$ with $A_02^{\ell-n}\leq |\nu-\nu'|+|\mu-\mu'|\leq \varsigma\lambda^{1/2}$
\begin{equation}\label{contorl1}
\begin{aligned}
&|\tilde{\Psi}(\eta)+\tilde{\Psi}(\xi)-\tilde{\Psi}(\eta')-\tilde{\Psi}(\xi')|\\
&\quad\quad\quad\leq C[(2^{\ell}|\nu-\nu'|+|\nu-\nu'|^2)+(2^{\ell}|\mu-\mu'|+|\mu-\mu'|^2)+|\eta+\xi-\eta'-\xi'|].
\end{aligned}
\end{equation}
Suppose now that $\tilde{\Psi}$ satisfies the additional assumption rank $\tilde{\Psi}_{\eta\eta}''=1$. Then if either $\mu\leq \nu$ and $\mu'\leq \nu'$ or $\nu \leq \mu$ and
$\nu'\leq \mu'$ and if $A_02^{\ell-n}\leq |\nu-\nu'|+|\mu-\mu'|\leq \varsigma\lambda^{1/2}$ we have also with suitable positive constants $c_0$, $C_0$
\begin{equation}\label{contorl2}
\begin{aligned}
&|\tilde{\Psi}(\eta)+\tilde{\Psi}(\xi)-\tilde{\Psi}(\eta')-\tilde{\Psi}(\xi')|\\
&\quad\quad\quad\geq c_0[(2^{\ell}|\nu-\nu'|+|\nu-\nu'|^2)+(2^{\ell}|\mu-\mu'|+|\mu-\mu'|^2)]-C_0|\eta+\xi-\eta'-\xi'|.
\end{aligned}
\end{equation}
\end{lem}

Now in order to apply the above lemma to evaluate (\ref{form}), we rewrite (\ref{form}) as $\frac{\delta^{2(N_2+1)}}{\delta^{N+N_1}}$ times
\begin{equation}\label{form2}
\frac{\prod_{i=0}^{N_2}\left(\Psi_i(t,\eta)/\delta^2+\Psi_i(t,\xi)/\delta^2-\Psi_i(t,\eta')/\delta^2-\Psi_i(t,\xi')/\delta^2\right)}
{(\partial_tF/\delta)^{N+N_1}}\partial_t^{\beta}b_{j,k,j',k'}^{\nu,\mu,\nu',\mu'}(y,t,\eta,\xi,\eta',\xi').
\end{equation}

The worst case to bound the numerator of (\ref{form2}) from above is that we apply inequality (\ref{contorl1}) in the above lemma with $\tilde{\Psi}=\frac{1}{\delta^2}\partial_t^2 q_{\delta}(\cdot,t)$, at the same time, we will replace $\varsigma$ by $c_2-c_1$, $h(\eta)$ by $\frac{1}{\delta^2}\partial_t^2q_{\delta}(\eta,t)$ and $\lambda$ by $\lambda/\delta$ in the above lemma. In order to bound the denominator of (\ref{form2}) from below,  we apply inequality (\ref{contorl2}) in the above lemma with $\tilde{\Psi}=\frac{1}{\delta}q_{\delta}'$, in this case, we will replace $\varsigma$ by $c_2-c_1$, $h(\eta)$ by $\frac{1}{\delta}q_{\delta}'(\eta,t)$ and $\lambda$ by $\lambda/\delta$ in the above lemma.  Note that  $\partial_y\Psi(y,t,\eta,\xi,\eta',\xi')=\eta+\xi-\eta'-\xi'$,  by (\ref{b}), $\frac{\delta^{2(N_2+1)}}{\delta^{N+N_1}} \times$ (\ref{form2}) can be controlled by
\begin{equation}\label{N2}
\frac{C\delta^{N_2+1}(\delta\lambda)^{(N-N_2)/2}\left[\delta\left((2^{\ell}|\nu-\nu'|+|\nu-\nu'|^2)+(2^{\ell}|\mu-\mu'|+|\mu-\mu'|^2)+|\partial_y\Psi|\right)\right]^{N_2+1}}
{\biggl(\delta\max\left\{|\partial_y\Psi|,\left|(2^{\ell}|\nu-\nu'|+|\nu-\nu'|^2)+(2^{\ell}|\mu-\mu'|+|\mu-\mu'|^2)-|\partial_y\Psi|\right|\right\}\biggl)^{N+N_1}}
\end{equation}

If $|\nu-\nu'|+|\mu-\mu'|\geq \lambda^{\varepsilon}$, together with the fact that $(\frac{\pi}{2}-c_2)\sqrt{\frac{\lambda}{\delta}}\lesssim \nu \lesssim (\frac{\pi}{2}-c_1)\sqrt{\frac{\lambda}{\delta}}$, then
\begin{align*}
&\frac{1}{\delta\left((2^{\ell}|\nu-\nu'|+|\nu-\nu'|^2)+(2^{\ell}|\mu-\mu'|+|\mu-\mu'|^2)\right)}\\
&\leq C\delta^{-1} \left(2^{\ell}\lambda^{\epsilon}+\sqrt{\frac{\lambda}{\delta}}\lambda^{\epsilon}\right)^{-1}\\
&\leq C (\delta\lambda)^{-1/2}\lambda^{-\epsilon}.
\end{align*}

Since $\delta\lambda>1$, (\ref{N2}) can be controlled by
$\mathcal{O}\biggl(\frac{\delta^{N_2+1}(\delta\lambda)^{(N-N_2)/2}}{(\delta\lambda)^{(N-1)/2}{\lambda}^{\varepsilon(N-1)}}\biggl)
=\mathcal{O}(\frac{1}{\lambda^{\varepsilon (N-1)-1/2}})$. Now if we choose $N$ sufficiently large, the similar argument with Lemma \ref{othogonal1} gives the desired result.
\end{proof}

Let us continue the proof of our theorem. Based on Lemma \ref{othogonal2} and the Cauchy-Schwarz inequality, then we have that
\begin{align*}
&\biggl\|\biggl(\sum_{j,k}|\sum_{|\nu-\mu|\approx2^{\ell}}\tilde{F}_{\lambda,\ell}^{\nu,j}f
\tilde{F}_{\lambda,\ell}^{\mu,k}f|^2\biggl)^{1/2}\biggl\|_{L^2}^2\\
&=\int\sum_{j,k}\biggl|\sum_{\substack{|\nu-\mu|\approx2^{\ell}\\|\nu'-\mu'|\approx2^{\ell}}}\tilde{F}_{\lambda,\ell}^{\nu,j}f
\tilde{F}_{\lambda,\ell}^{\mu,k}f\overline{\tilde{F}_{\lambda,\ell}^{\nu',j}f
\tilde{F}_{\lambda,\ell}^{\mu',k}f}\biggl|dydt\\
&\leq \int\sum_{j,k}\biggl|\sum_{\substack{|\nu-\mu|,|\nu'-\mu'|\approx2^{\ell}\\|\nu-\nu'|+|\mu-\mu'|\leq \lambda^{\varepsilon} }}\tilde{F}_{\lambda,\ell}^{\nu,j}f
\tilde{F}_{\lambda,\ell}^{\mu,k}f\overline{\tilde{F}_{\lambda,\ell}^{\nu',j}f
\tilde{F}_{\lambda,\ell}^{\mu',k}f}\biggl|dydt+C\lambda^{-N}\|f\|_{L^4}^4\\
&\leq \lambda^{2\varepsilon}\int\sum_{j,k}\sum_{|\nu-\mu|\approx2^{\ell} }\biggl|\tilde{F}_{\lambda,\ell}^{\nu,j}f
\tilde{F}_{\lambda,\ell}^{\mu,k}f\biggl|^2dydt+C\lambda^{-N}\|f\|_{L^4}^4\\
&\leq \lambda^{2\varepsilon}\biggl\|\biggl(\sum_{j,\nu}|\tilde{F}_{\lambda,\ell}^{\nu,j}f|^2\biggl)^{1/2}\biggl\|_{L^4}^4+C\lambda^{-N}\|f\|_{L^4}^4.
\end{align*}

At present, we put all estimates together and get
\begin{align*}
\|\tilde{F}_{\lambda}f\|_{L^4}^2\leq C \lambda^{\varepsilon}\sum_{\ell}2^{\ell/2}\biggl\|\biggl(\sum_{j,\nu}|\tilde{F}_{\lambda,\ell}^{\nu,j}f|^2\biggl)^{1/2}\biggl\|_{L^4}^2+C\lambda^{-N}\|f\|_{L^4}^2.
\end{align*}

Furthermore, we introduce a finer decomposition of the operators $\tilde{F}_{\lambda,\ell}^{\nu,j}$, namely, we set
\begin{equation}
\tilde{F}_{\lambda,\ell}^{\nu,j,n}f(y,t):=\int_{{\mathbb{R}}^2}e^{i(\xi \cdot y+E(\xi)+q_{\delta}(\xi,t))}\tilde{a}_{\lambda,\ell}^{\nu,j,n}(y,t,\xi)\hat{f}(\xi)d\xi,
\end{equation}
where
\begin{equation}
\tilde{a}_{\lambda,\ell}^{\nu,j,n}(y,t,\xi)=\varphi(\lambda^{-1/2}q_{\delta}'(\xi,t)-n)\tilde{a}_{\lambda,\ell}^{\nu,j}(y,t,\xi).
\end{equation}

For fixed $y$, $t$, supp $\xi\rightarrow \tilde{a}_{\lambda,\ell}^{\nu,j,n}$ is comparable to a $\frac{\lambda^{1/2}}{\delta}\times(\frac{\lambda
}{\delta})^{1/2}$ rectangle and  $\tilde{F}_{\lambda,\ell}^{\nu,j}=\sum_{n}\tilde{F}_{\lambda,\ell}^{\nu,j,n}$, where the sum involves less than $2^{-\ell}\lambda^{1/2}$. Since
 on the support of the symbols $\tilde{a}_{\lambda,\ell}^{\nu,j,n}$, $\lambda^{1/2}(n-1) \leq q_{\delta}'(\xi,t)\leq \lambda^{1/2}(n+1)$ and on the support of the symbols $\tilde{a}_{\lambda,\ell}^{\nu,j}$, $2^{-\ell}\lambda(j-1) \leq q_{\delta}'(\xi,t)\leq 2^{-\ell}\lambda(j+1)$,  for fixed $j$,  then which implies that if $n$ satisfies $|n-n(j)|\leq 2^{-\ell}\lambda^{1/2}$ with $n(j)=[2^{-\ell}\lambda^{1/2}j]$, then $\tilde{F}_{\lambda,\ell}^{\nu,j,n}\neq 0$, where $[\cdot]$ denotes the nearest integer function.

One can follow the proof of Lemma \ref{othogonal1} and get an almost orthogonality lemma similar to Lemma \ref{othogonal1}:
\begin{lem}\label{othogonal3}
Suppose that $|m+n-m'-n'|\geq \lambda^{\varepsilon}$. Then for any $N>0$,
\begin{equation}
\left|\int_{\mathbb{R}^3}\tilde{F}_{\lambda,\ell}^{\nu,j,m}f\tilde{F}_{\lambda,\ell}^{\mu,k,n}f\overline{\tilde{F}_{\lambda,\ell}^{\nu,j,m'}f\tilde{F}_{\lambda,\ell}^{\mu,k,n'}f}dydt\right|
\leq C_{\varepsilon,N}\lambda^{-N}\|f\|_{L^4}^4.
\end{equation}
\end{lem}

Using this lemma and the Cauchy-Schwarz inequality, we obtain

\begin{align*}
&\biggl\|\biggl(\sum_{j,\nu}|\tilde{F}_{\lambda,\ell}^{\nu,j}f|^2\biggl)^{1/2}\biggl\|_{L^4}^4\\
&=\sum_{\substack{\nu,\mu\\j,k}}\int\biggl|\sum_{m,n}\tilde{F}_{\lambda,\ell}^{\nu,j,m}f\tilde{F}_{\lambda,\ell}^{\mu,k,n}f\biggl|^2dydt\\
&=\sum_{\substack{\nu,\mu\\j,k}}\int\biggl|\sum_{\substack{s\leq 2^{-\ell}\lambda^{1/2}\\s'\leq 2^{-\ell}\lambda^{1/2}}}\sum_{\substack{m-m'=s\\n'-n=s'}}\tilde{F}_{\lambda,\ell}^{\nu,j,m}f\tilde{F}_{\lambda,\ell}^{\mu,k,n}f
\overline{\tilde{F}_{\lambda,\ell}^{\nu,j,m'}f\tilde{F}_{\lambda,\ell}^{\mu,k,n'}f}\biggl|dydt\\
&\leq \sum_{\substack{\nu,\mu\\j,k}}\int\biggl|\sum_{\substack{s\leq 2^{-\ell}\lambda^{1/2}\\s'\leq 2^{-\ell}\lambda^{1/2}\\|s-s'|\leq \lambda^{\epsilon}}}\sum_{\substack{m-m'=s\\n'-n=s'}}\tilde{F}_{\lambda,\ell}^{\nu,j,m}f\tilde{F}_{\lambda,\ell}^{\mu,k,n}f
\overline{\tilde{F}_{\lambda,\ell}^{\nu,j,m'}f\tilde{F}_{\lambda,\ell}^{\mu,k,n'}f}\biggl|dydt\\
&\quad +C_N\lambda^{-N}\|f\|_{L^4}^4\\
& \leq \lambda^{\epsilon}\sum_{\substack{\nu,\mu\\j,k}}\sum_{s\leq 2^{-\ell}\lambda^{1/2}}\int\biggl|\sum_{\substack{m-m'=s\\n'-n=s}}\tilde{F}_{\lambda,\ell}^{\nu,j,m}f\tilde{F}_{\lambda,\ell}^{\mu,k,n}f
\overline{\tilde{F}_{\lambda,\ell}^{\nu,j,m'}f\tilde{F}_{\lambda,\ell}^{\mu,k,n'}f}\biggl|dydt\\
& = \lambda^{\epsilon}\sum_{\substack{\nu,\mu\\j,k}}\sum_{s\leq 2^{-\ell}\lambda^{1/2}}\int\biggl|\sum_{m-m'=s}\tilde{F}_{\lambda,\ell}^{\nu,j,m}f\overline{\tilde{F}_{\lambda,\ell}^{\nu,j,m'}f}
\sum_{n-n'=s}\overline{\tilde{F}_{\lambda,\ell}^{\mu,k,n}f\overline{\tilde{F}_{\lambda,\ell}^{\mu,k,n'}f}}\biggl|dydt\\
& \leq \lambda^{\epsilon}\sum_{\substack{\nu,\mu\\j,k}}\int\Biggl(\sum_{s\leq 2^{-\ell}\lambda^{1/2}}\biggl|\sum_{m-m'=s}\tilde{F}_{\lambda,\ell}^{\nu,j,m}f\overline{\tilde{F}_{\lambda,\ell}^{\nu,j,m'}f}\biggl|^2\Biggl)^{1/2}\\
&\quad \times
\Biggl(\sum_{s\leq 2^{-\ell}\lambda^{1/2}}\biggl|\sum_{n-n'=s}\tilde{F}_{\lambda,\ell}^{\mu,k,n}f\overline{\tilde{F}_{\lambda,\ell}^{\mu,k,n'}f}\biggl|^2\Biggl)^{1/2}dydt\\
& = \lambda^{\varepsilon}\int\biggl(\sum_{j,\nu}\biggl(\sum_{s\leq 2^{-\ell}\lambda^{1/2}}\biggl|\sum_{n-n'=s}\tilde{F}_{\lambda,\ell}^{\nu,j,n}f\overline{\tilde{F}_{\lambda,\ell}^{\nu,j,n'}f}\biggl|^2\biggl)^{1/2}\biggl)^2
dydt\\
& \leq C \lambda^{\varepsilon}2^{-\ell}\lambda^{1/2}\int\Biggl(\sum_{j,\nu}\biggl(\sum_{s\leq 2^{-\ell}\lambda^{1/2}}\sum_{n-n'=s}\biggl|\tilde{F}_{\lambda,\ell}^{j,\nu,n}f\overline{\tilde{F}_{\lambda,\ell}^{\nu,j,n'}f}\biggl|^2\biggl)^{1/2}\Biggl)^2
dydt\\
& \leq C \lambda^{\varepsilon}2^{-\ell}\lambda^{1/2}\int\Biggl(\sum_{j,\nu}\biggl(\sum_{n}|\tilde{F}_{\lambda,\ell}^{\nu,j,n}f|^4\biggl)^{1/2}\Biggl)^2
dydt\\
&\leq C\lambda^{\varepsilon}2^{-\ell}\lambda^{1/2}\biggl\|\biggl(\sum_{j,\nu,n}|\tilde{F}_{\lambda,\ell,n}^{\nu,j}f|^2\biggl)^{1/2}\biggl\|_{L^4}^4.
\end{align*}

We would be done if we could prove that
\begin{equation}\label{P^n}
\biggl\|\biggl(\sum_{j,\nu,n}|\tilde{F}_{\lambda,\ell}^{\nu,j,n}f|^2\biggl)^{1/2}\biggl\|_{L^4}\leq C\delta^{-(1/2+\varepsilon_2)}\lambda^{\varepsilon_1}\|f\|_{L^4}.
\end{equation}

Now comparing expressions of $\tilde{F}_{\lambda,\ell}^{\nu,j,n}f$ and $\mathcal{F}_{\lambda,\nu}^{\delta}f_n$ which appeared in Section 3, we see that they have a similar support in $\xi$, see Figure 8. Since all estimates after (\ref{lemma index2}) in Section 3 are valid under small smooth perturbations $q_{\delta}'$, then from now on, we can combine the idea from Section 3 and the proof of Theorem 6.1 in \cite{mss}, the $L^1$ boundness on the kernel of the remaining operator, Carleson's square function estimate, and a Kakeya-type maximal estimate yield the factor $\delta^{-(1/2+\varepsilon_2)}\lambda^{\varepsilon_1}$. We thus finish the proof for the more general phase function.


\section[Maximal functions associated with nonisotropic dilations of
some classes of hypersurfaces in ${\mathbb{R}}^3$ ]{Maximal functions associated with nonisotropic dilations of
some classes of hypersurfaces in ${\mathbb{R}}^3$.}

\subsection{Proofs for surfaces with one non-vanishing principal curvature}

\subsubsection{Maximal theorem  with $2a_2\neq a_3$}\label{nonvanishing}
In this section we give the proof for Theorem \ref{3nonvanish}.

We can always choose non-negative functions $\eta_1$,
$\eta_2$ $\in C_0^{\infty}({\mathbb{R}})$ so that $\eta(x)\leq \eta_1(x_1)\eta_2(x_2)$.
Since
\begin{equation*}
\left|\int_{\mathbb{R}^2}f(y-\delta_t(x_1,x_2,\Phi(x_1,x_2)))\eta(x)dx\right |\leq \int_{\mathbb{R}^2}|f|(y-\delta_t(x_1,x_2,\Phi(x_1,x_2)))\eta_1(x_1)\eta_2(x_2)dx,
\end{equation*}
then we may assume $\eta(x)=\eta_1(x_1)\eta_2(x_2)$ and  $f\geq
0$, $a_1=1$. Set $(y_2, y_3)=y'$ and $(\xi_2,\xi_3)=\xi'$. Denote $1+a_2+a_3$ by $Q$ and $(t^{a_2}\xi_2, t^{a_3}\xi_3)$ by
$\delta'_t\xi'$.

Put
\begin{equation*}
A_tf(y):=\int_{\mathbb{R}^2}f(y-\delta_t(x_1,x_2,\Phi(x_1,x_2)))\eta_1(x_1)\eta_2(x_2)dx.
\end{equation*}

By means of the Fourier inversion formula, we can write
\begin{equation*}
 A_tf(y)=\frac{1}{(2\pi)^3}\int_{{\mathbb{R}}^3}e^{i\xi\cdot
  y}\int_{\mathbb{R}}e^{-it\xi_1x_1}\eta_1(x_1)\widehat{d\mu_{x_1}}(\delta'_t\xi')dx_1\hat{f}(\xi)d\xi,
\end{equation*}
where
\begin{equation}\label{symbolwe}
 \widehat{d\mu_{x_1}}(\xi')=\int_{\mathbb{R}}e^{-i(\xi_2x_2+\xi_3\Phi(x_1,x_2))}\eta_2(x_2)dx_2.
\end{equation}

Choose a non-negative function $\beta\in C_0^{\infty}(\mathbb{R})$
such that
\begin{equation*}
\textrm{supp }\hspace{0.2cm}\beta \subset[1/2,2] \hspace{0.5cm}\textrm{and}
\hspace{0.5cm}\sum_{j\in \mathbb{Z}}\beta(2^{-j}r)=1 \hspace{0.5cm}
\textrm{for} \hspace{0.2cm} r>0.
\end{equation*}

Put
\begin{equation*}
A_{t,j}f(y):= \frac{1}{(2\pi)^3}\int_{{\mathbb{R}}^3}e^{i\xi\cdot
  y}\int_{\mathbb{R}}e^{-it\xi_1x_1}\eta_1(x_1)\widehat{d\mu_{x_1}}(\delta'_t\xi')dx_1\beta(2^{-j}|\delta'_t\xi'|)\hat{f}(\xi)d\xi,
\end{equation*}
and $M_j$ is the corresponding maximal operator.

\begin{align*}
 {}A_{t}^0f(y)
:&=A_{t}f(y)-\sum_{j=1}^{\infty}A_{t,j}f(y)
\\
&=\frac{1}{(2\pi)^3}\int_{{\mathbb{R}}^3}e^{i\xi\cdot
  y}\int_{\mathbb{R}}e^{-it\xi_1x_1}\eta_1(x_1)\widehat{d\mu_{x_1}}(\delta'_t\xi')dx_1\rho(|\delta'_t\xi'|)\hat{f}(\xi)d\xi,
\end{align*}
where $\rho$ is supported in a neighborhood of the origin. Since
\begin{align*}
{}&\frac{1}{(2\pi)^3}\int_{{\mathbb{R}}^3}e^{i\xi\cdot
  y}\int_{\mathbb{R}}e^{-it\xi_1x_1}\eta_1(x_1)\widehat{d\mu_{x_1}}(\delta'_t\xi')dx_1\rho(|\delta'_t\xi'|)d\xi
\\
&=\frac{1}{(2\pi)^2}\int_{\mathbb{R}}\eta_1(x_1)\int_{{\mathbb{R}}^2}e^{i\xi'\cdot
  y'}\widehat{d\mu_{x_1}}(\delta'_t\xi')\rho(|\delta'_t\xi'|)d\xi'\delta_0(y_1-tx_1)dx_1
\\
&=\frac{1}{(2\pi)^2}\frac{1}{t}\eta_1(\frac{y_1}{t})\int_{{\mathbb{R}}^2}e^{i\xi'\cdot
  y'}\widehat{d\mu_{\frac{y_1}{t}}}(\delta'_t\xi')\rho(|\delta'_t\xi'|)d\xi'
\\
&=\frac{1}{(2\pi)^2}t^{-Q}\eta_1(\frac{y_1}{t})\int_{{\mathbb{R}}^2}e^{i\delta'_{t^{-1}}\xi'\cdot
  y'}\widehat{d\mu_{\frac{y_1}{t}}}(\xi')\rho(|\xi'|)d\xi',
\end{align*}
then  $A_{t}^0f(y)=t^{-Q}f*K(t^{-1}y_1,t^{-a_2}y_2,t^{-a_3}y_3)=f*K_{\delta_{t^{-1}}}(y)$,
where
\begin{equation*}
K(y)=\frac{1}{(2\pi)^2}\eta_1(y_1)\int_{{\mathbb{R}}^2}e^{i\xi'\cdot
  y'}\widehat{d\mu_{y_1}}(\xi')\rho(|\xi'|)d\xi'.
\end{equation*}

Since $\Phi$ satisfies (\ref{conditionnonvani}) and $y_1 \in$ supp $\eta_1$, for $N\in \mathbb{N}$ and multi-index $\alpha$ with $|\alpha|=N$, then Theorem \ref{lem:Lemma1} implies that
\begin{equation}
|D_{\xi'}^{\alpha}\widehat{d\mu_{y_1}}(\xi')|\leq C_N'(1+|\xi'|)^{-1/2}.
\end{equation}

By integration by parts in $\xi'$, we get
\begin{align*}
|K(y)|&\leq C''_N\frac{1}{(1+|y_1|)^{N}}\frac{1}{(1+|y'|)^{N}}\int_{{\mathbb{R}}^2}\left|D_{\xi'}^{\alpha}\left(\widehat{d\mu_{y_1}}(\xi')\rho(|\xi'|)\right)\right|d\xi'\\
&\leq C_N\prod_{i=1}^3\frac{1}{(1+|y_i|)^N}.
\end{align*}

Now we choose $N$ sufficiently large, then
\begin{equation}\label{M1}
\sup_{t>0}|A_t^0f(y)|\leq C_N Mf(y),
\end{equation}
where $M$ is the non-isotropic Hardy-Littlewood maximal operator defined by
\begin{equation}\label{maximalfuxntion}
Mf(y):=\sup_{t>0}\frac{1}{|B(y,t)|}\int_{B(y,t)}f(x)dx,
\end{equation}
where $B(y,t)=\{x:|y-x|_{\delta}<t\}$ and $|x|_{\delta}=\max_{i=1}^3|x_i|^{1/{a_i}}$.

So it suffices to prove that
\begin{equation}
\|M_jf\|_{L^p}\leq
C_p2^{-j\epsilon(p)}\|f\|_{L^p},\hspace{0.2cm} j\geq 1, \hspace{0.2cm}2<p<\infty,
\hspace{0.2cm} \textrm{some }\hspace{0.2cm}\epsilon(p)>0.
\end{equation}

Since  $A_{t,j}$ is localized to frequencies $|\delta_t\xi|\approx 2^j$, we can employ Lemma \ref{Appendix} in a similar way as in Section \ref{nonvanishingcur}
to prove that
\begin{equation*}
\|M_j\|_{L^p\rightarrow L^p}\lesssim \|M_{j,loc}\|_{L^p\rightarrow
L^p},
\end{equation*}
where $M_{j,loc}f(y):=\sup_{t\in[1,2]}|A_{t,j}f(y)|$.

Indeed, for fixed $j\geq 1$ and  all $\ell \in \mathbb{Z}$, let $\bigtriangleup^j_{\ell}$ be Littlewood-Paley
operator in ${\mathbb{R}}^2$ defined by
$\widehat{\bigtriangleup^j_{\ell}f}(\xi)=\tilde{\beta}(2^{-j}|\delta'_{2^{\ell}}\xi'|)\hat{f}(\xi)$, where $\tilde{\beta}\in C_0^{\infty}({\mathbb{R}})$ is nonnegative and satisfies $\beta(|\delta'_{t}\xi'|)=\beta(|\delta'_{t}\xi'|)\tilde{\beta}(|\xi'|)$.
$M_jf(y)$ is equal to

\begin{align*}
&\sup_{\ell\in \mathbb{Z}}\sup_{t\in[1,2]}\left|\int_{{\mathbb{R}}^3}e^{i\xi\cdot y}\int_{\mathbb{R}}e^{-i2^{\ell}t\xi_1x_1}\eta_1(x_1)\widehat{d\mu_{x_1}}(\delta'_{2^{\ell}t}\xi')dx_1\beta(2^{-j}
|\delta'_{2^{\ell}t}\xi'|)\tilde{\beta}(2^{-j}|\delta'_{2^{\ell}}\xi'|)
\hat{f}(\xi)d\xi\right|
\\
&= \biggl(\sum_{\ell\in \mathbb{Z}}\sup_{t\in[1,2]}\left|\int_{{\mathbb{R}}^3}e^{i\xi\cdot y}\int_{\mathbb{R}}e^{-i2^{\ell}t\xi_1x_1}\eta_1(x_1)\widehat{d\mu_{x_1}}(\delta'_{2^{\ell}t}\xi')dx_1\beta(2^{-j}
|\delta'_{2^{\ell}t}\xi'|)\widehat{\bigtriangleup^j_{\ell}f}(\xi)d\xi\right|^p\biggl)^{1/p}
\\
&=\biggl(\sum_{\ell\in \mathbb{Z}}\sup_{t\in[1,2]}\biggl|2^{-\ell Q}\int_{{\mathbb{R}}^3}e^{i\xi\cdot \delta_{2^{-\ell}}y}\int_{\mathbb{R}}e^{-it\xi_1x_1}\eta_1(x_1)\widehat{d\mu_{x_1}}(\delta'_{t}\xi')dx_1\beta(2^{-j}
|\delta'_{t}\xi'|) \\
&\quad\quad\quad\times \widehat{\bigtriangleup^j_{\ell}f}(\delta_{2^{-\ell}}\xi)d\xi\biggl|^p\biggl)^{1/p}
\\
&=\biggl(\sum_{\ell\in \mathbb{Z}}\biggl|\mathcal{M}_{j,loc}(\bigtriangleup^j_{\ell}f\circ\delta_{2^{\ell}})(\delta_{2^{-\ell}}y)\biggl|^p\biggl)^{1/p}.
\end{align*}

Since $p>2$, then Lemma \ref{Appendix} implies that
\begin{align*}
\|M_jf(y)\|_{L^p}^p
&= \sum_{\ell\in \mathbb{Z}}\int_{{\mathbb{R}}^3}\biggl|\mathcal{M}_{j,loc}(\bigtriangleup^j_{\ell}f\circ\delta_{2^{\ell}})(\delta_{2^{-\ell}}y)\biggl|^pdy
\\
& = \sum_{\ell\in \mathbb{Z}}2^{\ell Q}\biggl\|\mathcal{M}_{j,loc}(\bigtriangleup^j_{\ell}f\circ\delta_{2^{\ell}})\biggl\|_{L^p}^p
\\
& \leq \|\mathcal{M}_{j,loc}\|_{L^p\rightarrow L^p}^p \sum_{\ell\in \mathbb{Z}}2^{\ell Q}\int_{{\mathbb{R}}^3}|\bigtriangleup^j_{\ell}f(\delta_{2^{\ell}}y)|^pdy
\\
& =\|\mathcal{M}_{j,loc}\|_{L^p\rightarrow L^p}^p \sum_{\ell\in \mathbb{Z}}\int_{{\mathbb{R}}^3}|\bigtriangleup^j_{\ell}f(y)|^pdy\\
&\leq \|\mathcal{M}_{j,loc}\|_{L^p\rightarrow L^p}^p \int_{\mathbb{R}}\int_{{\mathbb{R}}^2}\biggl(\sum_{\ell\in \mathbb{Z}}|\bigtriangleup^j_{\ell}f(y)|^2\biggl)^{p/2}dy\\
&=\|\mathcal{M}_{j,loc}\|_{L^p\rightarrow L^p}^p \Biggl\|\biggl\|\biggl(\sum_{\ell\in \mathbb{Z}}|\bigtriangleup^j_{\ell}f(y)|^2\biggl)^{1/2}\biggl\|_{L^p({\mathbb{R}}^2)}\Biggl\|_{L^p(\mathbb{R})}^p
\\
&\leq C_p\|\mathcal{M}_{j,loc}\|_{L^p\rightarrow L^p}^p\|f\|_{L^p({\mathbb{R}}^3)}.
\end{align*}

For fixed $t\in[1,2]$, let us estimate $\widehat{d\mu_{x_1}}(\delta'_t\xi')$.

Set
\begin{equation}
s:=s(\xi',t)=-\frac{t^{a_2}\xi_2}{t^{a_3}\xi_3}, \hspace{0.3cm}\textrm{for} \hspace{0.2cm} \xi_3\neq 0,
\end{equation}
and
\begin{equation}
\Psi(x_1,x_2,s):=-sx_2+\Phi(x_1,x_2).
\end{equation}

We observe that
\begin{equation}
\partial_2\Psi(0,0,0)=0 \hspace{0.5cm}\textrm{and}\hspace{0.5cm}\partial_2^2\Psi(0,0,0)\neq0,
\end{equation}
since $\Phi$ satisfies the condition (\ref{conditionnonvani}). The implicit function theorem implies that there must be a smooth solution $x_2=\psi(x_1,s)$ to the equation
\begin{equation}
\partial_2\Phi(x_1,x_2)=s,
\end{equation}
where $x_1$ and $s$ are enough small. Here if $t\in [\frac{1}{2},4]$, we can choose $x_1$ and $s$ sufficiently small such that
\begin{equation}\label{equation}
\partial_2\Phi(\frac{x_1}{t},\psi(\frac{x_1}{t},s))=s.
\end{equation}

For above sufficiently small $x_1$, a standard application of the method of stationary phase in $x_2$  yields that
\begin{equation*}
\widehat{d\mu_{x_1}}(\delta'_t\xi')=e^{-it^{a_3}\xi_3\tilde{\Psi}(x_1,s)}
\frac{\chi_{x_1}(t^{a_2}\xi_2/t^{a_3}\xi_3)}{{(1+|\delta'_t\xi'|)^{1/2}}}
A_{x_1}(\delta'_t\xi')+B_{x_1}(\delta'_t\xi'),
\end{equation*}
where $\tilde{\Psi}(x_1,s):=\Psi(x_1,\psi(x_1,s),s)$ and  $\chi_{x_1}$ is a smooth function supported on the set $\{z:|z|<\epsilon_{x_1}\}$, where $\epsilon_{x_1}$ can be controlled by a small positive constant independent on $x_1$.
Moreover, $t^{a_3}\xi_3\tilde{\Psi}(x_1,s)$
is a smooth function which is homogeneous of degree one in $\xi'$.  Meanwhile, $A_{x_1}$ is a
symbol of order zero such that
\begin{equation}\label{symbio}
|D^{\alpha}_{\xi'}A_{x_1}(\xi')|\leq
C_{\alpha}(1+|\xi'|)^{-|\alpha|},
\end{equation}
where $\alpha$ is a multi-index and $C_{\alpha}$ are admissible
constants.  $B_{x_1}$ is a smooth function and  satisfies
\begin{equation}\label{B1}
|D_{\xi'}^{\alpha}B_{x_1}(\xi')|\leq C_{\alpha,N}(1+|\xi'|)^{-N},
\end{equation}
again with admissible constants $C_{\alpha,N}$ and $N\in \mathbb{N}$.

For fixed $t\in [1,2]$,
\begin{align*}
A_{t,j}^0f(y):&=\frac{1}{(2\pi)^3}\int
_{{\mathbb{R}}^3}e^{i\xi\cdot y
}\int_{\mathbb{R}}e^{-it\xi_1x_1}\eta_1(x_1)B_{x_1}(\delta'_t\xi')dx_1\beta(2^{-j}|\delta'_t\xi'|)\hat{f}(\xi)d\xi
\\
&=\mathcal{F}^{-1}\biggl\{\int
_{{\mathbb{R}}^3}e^{i\xi\cdot y
}\biggl(\int_{\mathbb{R}}e^{-it\xi_1x_1}\eta_1(x_1)B_{x_1}(\delta'_t\xi')dx_1\biggl)\beta(2^{-j}|\delta'_t\xi'|)\biggl\}*f(y),
\end{align*}
and
\begin{align*}
&\mathcal{F}^{-1}\biggl\{\int
_{{\mathbb{R}}^3}e^{i\xi\cdot y
}\biggl(\int_{\mathbb{R}}e^{-it\xi_1x_1}\eta_1(x_1)B_{x_1}(\delta'_t\xi')dx_1\biggl)\beta(2^{-j}|\delta'_t\xi'|)\biggl\}(y)
\\
&=\frac{1}{(2\pi)^2}\int_{{\mathbb{R}}^2}e^{i\xi'\cdot y'}\int_{\mathbb{R}}\eta_1(x_1)B_{x_1}(\delta'_t\xi')\delta_0(y_1-tx_1)dx_1\beta(2^{-j}|\delta'_t\xi'|)d\xi'
\\
&=\frac{1}{(2\pi)^2}t^{-Q}\eta_1(\frac{y_1}{t})\int_{{\mathbb{R}}^2}e^{i\xi'\cdot \delta'_{t^{-1}}y'}B_{y_1/t}(\xi')\beta(2^{-j}|\xi'|)d\xi'
\\
&=\tilde{K}_{\delta_{t^{-1}}}(y),
\end{align*}
where $\tilde{K}(y)=\frac{1}{(2\pi)^2}\eta_1(y_1)\int_{{\mathbb{R}}^2}e^{i\xi'\cdot y'}B_{y_1}(\xi')\beta(2^{-j}|\xi'|)d\xi'$. By (\ref{B1}) and the support of $\beta$, it is easy to get
\begin{equation*}
|\tilde{K}(y)|\leq C_N2^{-jN}\prod_{i=1}^3\frac{1}{(1+|y_i|)^N},
\end{equation*}
and
\begin{equation}\label{M2}
\sup_{t\in[1,2]}|A_{t,j}^0f(y)|\leq C_N2^{-jN}Mf(y),
\end{equation}
where $M$ denotes the Hardy-Littlewood maximal operator defined by (\ref{maximalfuxntion}).

Put
\begin{align*}
A_{t,j}^1f(y):=&\frac{1}{(2\pi)^3}\int
_{{\mathbb{R}}^3}e^{i\xi\cdot y
}\int_{\mathbb{R}}e^{-it\xi_1x_1}\eta_1(x_1)e^{-it^{a_3}\xi_3\tilde{\Psi}(x_1,s)}E_{x_1}(\delta'_t\xi')dx_1\beta(2^{-j}|\delta'_t\xi'|)\hat{f}(\xi)d\xi
\end{align*}
where
\begin{equation*}
E_{x_1}(\delta'_t\xi'):=
\frac{\chi_{x_1}(t^{a_2}\xi_2/t^{a_3}\xi_3)}{(1+|\delta'_t\xi'|)^{1/2}}
A_{x_1}(\delta'_t\xi'),
\end{equation*}
and denote by $\mathcal{M}_{j,loc}^1$ the corresponding maximal operator. It remains to prove that
\begin{equation}
\|\mathcal{M}_{j,loc}^1f\|_{L^p}\leq
C_p2^{-j\epsilon(p)}\|f\|_{L^p},\hspace{0.2cm} j\geq 1, \hspace{0.2cm}2<p<\infty,
\hspace{0.2cm} \textrm{some }\hspace{0.2cm}\epsilon(p)>0.
\end{equation}

\begin{align*}
\|\mathcal{M}_{j,loc}^1f\|_{L^p}&=\biggl\|\sup_{t\in[1,2]}\frac{1}{(2\pi)^2}\biggl|\int_{\mathbb{R}}\eta_1(x_1)\int
_{{\mathbb{R}}^2}e^{i(\xi'\cdot y'-t^{a_3}\xi_3\tilde{\Psi}(x_1,s))}E_{x_1}(\delta'_t\xi')
\\& \quad\times \beta(2^{-j}|\delta'_t\xi'|)f(y_1-tx_1,\widehat{\xi'})d\xi'dx_1\biggl|\biggl\|_{L^p(dy)}
\\
&\leq C\biggl \|\sup_{t\in[1,2]}\biggl|\int_{\mathbb{R}}\eta_1(\frac{x_1}{t})\int
_{{\mathbb{R}}^2}e^{i(\xi'\cdot y'-t^{a_3}\xi_3\tilde{\Psi}(\frac{x_1}{t},s))}E_{x_1/t}(\delta'_t\xi')
\\& \quad \times \beta(2^{-j}|\delta'_t\xi'|)f(y_1-x_1,\widehat{\xi'})d\xi'dx_1\biggl|\biggl\|_{L^p(dy)}
\\
& =C \biggl\|\sup_{t\in[1,2]}|\widetilde{A_{t,j}^1}f|\biggl\|_{L^p(dy)}=C\|\widetilde{\mathcal{M}_{j,loc}^1}f\|_{L^p},
\end{align*}
where
\begin{equation}\widetilde{A_{t,j}^1}f(y):=\int_{\mathbb{R}}\eta_1(\frac{x_1}{t})\int
_{{\mathbb{R}}^2}e^{i(\xi'\cdot y'-t^{a_3}\xi_3\tilde{\Psi}(\frac{x_1}{t},s))}E_{x_1/t}(\delta'_t\xi')
\beta(2^{-j}|\delta'_t\xi'|)f(y_1-x_1,\widehat{\xi'})d\xi'dx_1,\
\end{equation}
and
\begin{equation}
\widetilde{\mathcal{M}_{j,loc}^1}f(y):=\sup_{t\in[1,2]}|\widetilde{A_{t,j}^1}f(y)|.
\end{equation}

Set
\begin{equation}
Q_{x_1}(y',t,\xi')=\xi'\cdot y'-t^{a_3}\xi_3\tilde{\Psi}(\frac{x_1}{t},s).
\end{equation}
Choose a bump function $\tilde{\rho} \in C_0^{\infty}(\mathbb{R})$ supported in $[1/2,4]$ such that $\tilde{\rho}(t)=1$ if $1\leq t\leq 2$.  By Lemma \ref{lem:Lemma3}, we have
\begin{equation}\label{wellknoweest}
\begin{aligned}
\|\widetilde{\mathcal{M}_{j,loc}^1}f(y)\|^p_{L^p}&\leq C_p\biggl(\int
_{{\mathbb{R}}^3}\int_{1/2}^4\left|\tilde{\rho}(t)\widetilde{A_{t,j}^1}f(y)\right|^p dtdy\biggl )^{1/p'}\\
&\quad \times \biggl(\int
_{{\mathbb{R}}^3}\int_{1/2}^4\left|\frac{\partial}{\partial t}\left(\tilde{\rho}(t)\widetilde{A_{t,j}^1}f(y)\right)\right|^pdtdy\biggl)^{1/p}.
\end{aligned}
\end{equation}

Moreover, we can choose $\tilde{\eta}_1\in C_0^{\infty}(\mathbb{R})$ such that $\tilde{\eta}_1=1$  on the support of $\eta_1$, then we have
\begin{equation*}
\frac{\partial}{\partial t}\widetilde{A_{t,j}^1}f(y)  =\int_{\mathbb{R}}\tilde{\eta}_1(\frac{x_1}{t})\int
_{{\mathbb{R}}^2}e^{iQ_{x_1}(y',t,\xi')}h(t,j,x_1,\xi')f(y_1-x_1,\widehat{\xi'})d\xi'dx_1,
\end{equation*}
where
\begin{align*}
h(t,j,x_1,\xi')&=\left(-t^{-2}x_1\eta_1'(\frac{x_1}{t})+\frac{\partial}{\partial t}Q_{x_1}(y',t,\xi')\right)E_{x_1/t}(\delta'_t\xi')\beta(2^{-j}|\delta'_t\xi'|)
\\
& +\eta_1(\frac{x_1}{t})\frac{\partial}{\partial t}E_{x_1/t}(\delta'_t\xi')\beta(2^{-j}|\delta'_t\xi'|)
+\eta_1(\frac{x_1}{t})E_{x_1/t}(\delta'_t\xi')\frac{\partial}{\partial t}\beta(2^{-j}|\delta'_t\xi'|),
\end{align*}
and
\begin{equation*}
\left|\frac{\partial}{\partial t}Q_{x_1}(y',t,\xi'))\right|=\left|\xi_3\frac{\partial}{\partial t}(t^{a_3}\Psi(\frac{x_1}{t},\psi(\frac{x_1}{t},s),s))\right|\leq C|\xi_3|,
\end{equation*}
since $t\approx 1$, $x_1$ and the support of $\chi_{x_1}$ are sufficiently small.  Note that $A_{x_1}$ satisfies (\ref{symbio}), then it is easy to see that
$\frac{\partial}{\partial t}(\tilde{\rho}(t)\widetilde{A_{t,j}^1})$ behaves like $2^{j}\widetilde{A_{t,j}^1}$. Clearly we only need to show the $L^p$-boundedness of the operator $\widetilde{A_{t,j}^1}$.

Furthermore, choose a function $\eta_0\in C_0^{\infty}(\mathbb{R})$, non-negative so that  for arbitrary $t\in [1/2,4]$, $\eta_1(\frac{x_1}{t})\leq \eta_0(x_1)$, then
\begin{align*}
&\biggl(\int
_{{\mathbb{R}}^3}\int_{1/2}^4|\tilde{\rho}(t)\widetilde{A_{t,j}^1}f(y)|^pdtdy\biggl)^{1/p}
\\
&\lesssim\biggl\|\int_{\mathbb{R}}\eta_0(x_1)\biggl|\tilde{\rho}(t)\int
_{{\mathbb{R}}^2}e^{iQ_{x_1}(y',t,\xi')}E_{x_1/t}(\delta'_t\xi')\beta(2^{-j}|\delta'_t\xi'|)\\
&\quad\quad\quad \times
f(y_1-x_1,\widehat{\xi'})d\xi'\biggl| dx_1\biggl\|_{L^p([1/2,4]\times \mathbb{R}^3, dtdy'dy_1)}
\\
&\lesssim\Biggl\|\int_{\mathbb{R}}\eta_0(x_1)\biggl\|\tilde{\rho}(t)\int
_{{\mathbb{R}}^2}e^{iQ_{x_1}(y',t,\xi')}E_{x_1/t}(\delta'_t\xi')\beta(2^{-j}|\delta'_t\xi'|)\\
&\quad\quad\quad\times
f(y_1-x_1,\widehat{\xi'})d\xi'
\biggl\|_{L^p([1/2,4]\times \mathbb{R}^2,dtdy')}dx_1
\Biggl\|_{L^p(\mathbb{R}, dy_1)}.
\end{align*}

Now for the inner norm we would apply the local smoothing estimate from Theorem \ref{lem:lemma4}, so we should verify that $Q_{x_1}(y',t,\xi')$
satisfies the non-degeneracy condition (\ref{nondegeneracy}) and the cone condition (\ref{cone:condition}). Apparently, the non-degeneracy condition will follow from
\begin{equation*}
\hspace{0.2cm}\textrm{rank}\hspace{0.2cm} \partial^2_{(y',t),\xi'}Q_{x_1}(y',t,\xi')=\textrm{rank}
\left(
\begin{array} {lcr}
1 & 0  \\
0 & 1  \\
* & *
\end{array}
\right)
=2.
\end{equation*}

Since $\partial_2^2\Phi(0,0)\neq 0$, we can choose $U$ sufficiently small such that
\begin{equation}\label{phi1}
C_1|\partial_2^2\Phi(0,0)|\geq |\partial_2^2\Phi(x_1,x_2)|\geq C_2|\partial_2^2\Phi(0,0)|.
\end{equation}

Also since
\begin{equation}\label{equatins1}
\partial_2\Phi(\frac{x_1}{t},\psi(\frac{x_1}{t},s))=s,
\end{equation}
by applying $\frac{\partial}{\partial s}$ on both sides,  then we have
$$\partial_2^2\Phi(\frac{x_1}{t},\psi(\frac{x_1}{t},s))\partial_2\psi(\frac{x_1}{t},s)=1,$$
which implies that
\begin{equation}\label{equation2}
\partial_2\psi(\frac{x_1}{t},s)=\frac{1}{\partial_2^2\Phi(\frac{x_1}{t},\psi(\frac{x_1}{t},s))}.
\end{equation}
Next, we will use the fact that $|\xi_2|\ll|\xi_3|\approx |\xi'|\approx 2^j$ $(j\geq 1)$.

By the above arguments,  we get
\begin{align*}
\partial_tQ_{x_1}(y'_0,t_0,\xi') & =\xi_3\biggl[a_2t_0^{a_2-1}\frac{\xi_2}{\xi_3}\psi(\frac{x_1}{t_0},s)+{t_0}^{a_2}\frac{\xi_2}{\xi_3}\partial_2
\psi(\frac{x_1}{t_0},s)(a_2-a_3){t_0}^{a_2-a_3-1}\frac{\xi_2}{\xi_3}
\\
&\quad+
a_3{t_0}^{a_3-1}\Phi(\frac{x_1}{t_0},\psi(\frac{x_1}{t_0},s))+
 {t_0}^{a_3}\partial_2\Phi(\frac{x_1}{t_0},\psi(\frac{x_1}{t_0},s))\partial_2
\psi(\frac{x_1}{t_0},s)
\\
&\quad\times (a_2-a_3){t_0}^{a_2-a_3-1}\frac{\xi_2}{\xi_3}+x_1R^1(x_1,t_0,s)\biggl]
\\
& =\xi_3\biggl[a_2t_0^{a_2-1}\frac{\xi_2}{\xi_3}\psi(\frac{x_1}{t_0},s)+
a_3{t_0}^{a_3-1}\Phi(\frac{x_1}{t_0},\psi(\frac{x_1}{t_0},s))+x_1R^1(x_1,t_0,s)\biggl].
\end{align*}

Furthermore, we will prove that
\begin{equation}\label{coneoc}
|\partial_{\xi_2}^2(\partial_tQ_{x_1}(y'_0,t_0,\xi'))|\approx 1.
\end{equation}
Note that $\partial_tQ_{x_1}(y'_0,t_0,\xi')$ is homogeneous of degree one in $\xi'$, in addition to (\ref{coneoc}), then we obtain that
\begin{equation}
\textmd{rank}\hspace{0.2cm}\left(\frac{\partial^2}{\partial \xi_i\partial\xi_j}\right)\left\langle\partial_{(y',t)}Q_{x_1}(y_0',t_0,\xi'),\theta\right\rangle=1,\hspace{0.2cm}i,j=2,3,
\end{equation}
where $\pm\theta$ are the unique directions for which $\nabla_{\xi'}\langle\partial_{(y',t)}Q_{x_1}(y_0',t_0,\xi'),\theta\rangle =0$. This implies the cone condition (\ref{cone:condition}).

Now, let us turn to prove (\ref{coneoc}). Since $x_1$ is sufficiently small, $\left|\left(\frac{\partial}{\partial\xi_2}\right)^{\alpha}R^1(x_1,t_0,s)\right|$, $\alpha\leq 2$ is a remainder term. So we only consider
\begin{equation}\label{Q}
\biggl|\frac{{\partial}^2}{\partial{\xi_2^2}}\biggl(a_2t_0^{a_2-1}\frac{\xi_2}{\xi_3}\psi(\frac{x_1}{t_0},s)+
a_3{t_0}^{a_3-1}\Phi(\frac{x_1}{t_0},\psi(\frac{x_1}{t_0},s))\biggl)\biggl|\approx 1.
\end{equation}
\begin{align*}
{}&\frac{\partial}{\partial\xi_2}\biggl(a_2t_0^{a_2-1}\frac{\xi_2}{\xi_3}\psi(\frac{x_1}{t_0},s)+
a_3{t_0}^{a_3-1}\Phi(\frac{x_1}{t_0},\psi(\frac{x_1}{t_0},s))\biggl)
\\
& =a_2t_0^{a_2-1}\frac{1}{\xi_3}\psi(\frac{x_1}{t_0},s)-a_2t_0^{a_2-1}\frac{\xi_2}{\xi_3}\partial_2\psi(\frac{x_1}{t_0},s){t_0}^{a_2-a_3}\frac{1}{\xi_3}
\\
&\quad-a_3{t_0}^{a_3-1}\partial_2\Phi(\frac{x_1}{t_0},\psi(\frac{x}{t_0},s))\partial_2\psi(\frac{x}{t_0},s){t_0}^{a_2-a_3}\frac{1}{\xi_3}+x_1R^2(x_1,t_0,s)
\\
&=t_0^{a_2-1}\frac{1}{\xi_3}\biggl(a_2\psi(\frac{x_1}{t_0},s)+(a_3-a_2)t_0^{a_2-a_3}\frac{\xi_2}{\xi_3}\partial_2\psi(\frac{x_1}{t_0},s)
+x_1R^2(x_1,t_0,s)\biggl).
\end{align*}

By the same reason as before,  we only estimate
\begin{align*}
&\frac{{\partial}^2}{\partial{\xi_2^2}}\biggl(a_2\psi(\frac{x_1}{t_0},s)+(a_3-a_2)t_0^{a_2-a_3}\frac{\xi_2}{\xi_3}\partial_2\psi(\frac{x_1}{t_0},s)\biggl)
\\
&=-a_2\partial_2\psi(\frac{x_1}{t_0},s){t_0}^{a_2-a_3}\frac{1}{\xi_3}+(a_3-a_2)t_0^{a_2-a_3}\frac{1}{\xi_3}\partial_2
\psi(\frac{x_1}{t_0},s)-
(a_3-a_2)t_0^{a_2-a_3}\frac{\xi_2}{\xi_3}
\\
&\quad \quad\quad \times \partial_2^2\psi(\frac{x_1}{t_0},s)t_0^{a_2-a_3}\frac{1}{\xi_3}
\\
&={t_0}^{a_2-a_3}\frac{1}{\xi_3}\biggl((a_3-2a_2)\partial_2\psi(\frac{x_1}{t_0},s)+
(a_2-a_3)t_0^{a_2-a_3}\frac{\xi_2}{\xi_3}\partial_2^2\psi(\frac{x_1}{t_0},s)\biggl).
\end{align*}

(\ref{equation2}) and  (\ref{phi1}) imply that  $\partial_2\psi(\frac{x_1}{t_0},s)\geq \frac{1}{C_1|\Phi^2(0,0)|}\neq 0$. Moreover, since $2a_2\neq a_3$ and $|\xi_2|\ll |\xi_3|$, then $|(a_3-2a_2)\partial_2\psi(\frac{x_1}{t_0},s)|\approx 1$ and
$|(a_2-a_3)t_0^{a_2-a_3}\frac{\xi_2}{\xi_3}\partial_2^2\psi(\frac{x_1}{t_0},s)|\ll 1$. We finish the proof of  (\ref{Q}) and (\ref{coneoc}).

Hence we can apply Theorem \ref{lem:lemma4} for $\mu=-1/2$, and obtain
\begin{align*}
&\biggl(\int
_{{\mathbb{R}}^3}\int_{1/2}^4|\tilde{\rho}(t)\widetilde{A_{t,j}^1}f(y)|^pdtdy\biggl)^{1/p}\\
&\leq C_p 2^{-j(1/p+\epsilon(p))}\biggl\|\int_{\mathbb{R}}\eta_0(x_1)\|f(y_1-x_1,y')\|_{L^p({\mathbb{R}}^2,dy')}dx_1\biggl\|_{L^p(\mathbb{R},dy_1)}
\\
&\leq C_p 2^{-j(1/p+\epsilon(p))}\|f\|_{L^p}\|\eta_0\|_{L^1},
\end{align*}
and thus by  (\ref{wellknoweest}), we get
\begin{align*}
&\|\widetilde{\mathcal{M}_{j,loc}^1}f(y)\|_{L^p}^p\\
&\leq C_p\biggl( 2^{-j(1/p+\epsilon(p))}\|f\|_{L^p}\|\eta_0\|_{L^1}\biggl)^{p-1}\biggl( 2^{-j(-1+1/p+\epsilon(p))}\|f\|_{L^p}\|\eta_0\|_{L^1}\biggl)
\\
&= C_p2^{-jp\epsilon(p)}\|f\|_{L^p}^p\|\eta_0\|_{L^1}^p.
\end{align*}

Now we have finished the proof of Theorem \ref{3nonvanish}.


\subsubsection{Maximal theorem  with $2a_2=a_3$}\label{dim32a2=a3}

In this section, we will prove Theorem \ref{theocurvanishi}.

First we may assume $\eta(x)=\eta_1(x_1)\eta_2(x_2)$ with non-negative functions $\eta_1$,
$\eta_2$ $\in C_0^{\infty}({\mathbb{R}})$, and $f\geq
0$, $a_2=1$, $a_3=2$. Let $(y_2, y_3)=y'$ and $(\xi_2,\xi_3)=\xi'$. Denote $a_1+3$ by $Q$ and $(t\xi_2, t^2\xi_3)$ by
$\tilde{\delta}_t\xi'$. We choose $B>0$ very small and  $\tilde{\rho}\in C_0^{\infty}(\mathbb{R})$ such that supp $\tilde{\rho}\subset\{x\in \mathbb{R}:B/2\leq|x|\leq 2B\}$  and $\sum_k\tilde{\rho}(2^kx)=1$.

Put
\begin{equation*}
A_tf(y):=\int_{\mathbb{R}^2}f(y_1-t^{a_1}x_1,y_2-tx_2,y_3-t^2x_2^2\phi(x_2))\eta(x)dx=\sum_k\widetilde{A_t^k}f(y),
\end{equation*}
where $\widetilde{A_t^k}f(y)=\int_{\mathbb{R}^2}f(y_1-t^{a_1}x_1,y_2-tx_2,y_3-t^2x_2^2\phi(x_2))\eta_1(x_1)\eta_2(x_2)\tilde{\rho}(2^kx_2)dx$.

Define the isometric operator T on $L^p(\mathbb{R}^3)$ by
\begin{equation*}
Tf(x_1,x_2,x_3)=2^{3k/p}f(x_1,2^kx_2,2^{2k}x_3).
\end{equation*}

One can easily compute that
$$T^{-1}\widetilde{A_t^k}Tf(y)=2^{-k}\int_{\mathbb{R}^2}f(y_1-t^{a_1}x_1,y_2-tx_2,y_3-t^2x_2^2\phi(\frac{x_2}{2^k}))\eta_1(x_1)\eta_2(\frac{x_2}{2^k})\rho(x_2)dx.$$

It suffices to prove the following estimate
\begin{equation*}
\sum_k2^{-k}\|\sup_{t>0}|A_t^k|\|_{L^p\rightarrow L^p}\leq C_p,\hspace{0.2cm}p>2,
\end{equation*}
where
\begin{equation}
A_t^kf(y):=\int_{\mathbb{R}^2}f(y_1-t^{a_1}x_1,y_2-tx_2,y_3-t^2x_2^2\phi(\frac{x_2}{2^k}))\eta_1(x_1)\eta_2(\frac{x_2}{2^k})\rho(x_2)dx.
\end{equation}

By means of the Fourier inversion formula, we can write
\begin{equation*}
 A_t^kf(y)=\frac{1}{(2\pi)^3}\int_{{\mathbb{R}}^3}e^{i\xi\cdot
  y}\widehat{\eta_1}(t^{a_1}\xi_1)\widehat{d\mu_{k,m}}(\tilde{\delta}_t\xi')\hat{f}(\xi)d\xi,
\end{equation*}
where
\begin{align*}
 \widehat{d\mu_{k,m}}(\tilde{\delta}_t\xi')&=\int_{\mathbb{R}}e^{-i(t\xi_2x_2+t^2\xi_3x_2^2\phi(\frac{x_2}{2^k}))}\tilde{\rho}(x_2)\eta_2(2^{-k}x_2)dx_2.
\end{align*}

Choose a non-negative function $\beta\in C_0^{\infty}(\mathbb{R})$
such that
\begin{equation*}
\hspace{0.2cm}\textrm{supp }\hspace{0.2cm} \beta \subset[1/2,2] \hspace{0.5cm}\textrm{and}
\hspace{0.5cm}\sum_{j\in \mathbb{Z}}\beta(2^{-j}r)=1 \hspace{0.5cm}
\textrm{for} \hspace{0.5cm} r>0.
\end{equation*}
and set
\begin{equation}
A_{t,j}^kf(y):=\frac{1}{(2\pi)^3}\int_{{\mathbb{R}}^3}e^{i\xi\cdot y}\widehat{\eta_1}(t^{a_1}\xi_1)\widehat{d\mu_{k,m}}(\tilde{\delta}_t\xi')\beta(2^{-j}|\tilde{\delta}_t\xi'|)\hat{f}(\xi)d\xi,
\end{equation}
and denote by $M_j^k$ the corresponding maximal operator.

From the proof of inequality (\ref{M1}), it is easy to see that the supremum of the absolute value of
the difference between $A_{t}^kf(y)$ and $\sum_{j=1}^{\infty}A_{t,j}^kf(y)$ is dominated by the Hardy-Littlewood maximal function $Mf(y)$ defined by (\ref{maximalfuxntion}). It remains to consider the $L^p$-boundedness ($2<p<\infty$) of the maximal operator $M_j^k$ for $j\geq 1$.

Since $A_{t,j}^k$ is localized to frequencies $|\tilde{\delta}_t\xi'|\approx 2^j$, combining  the method of Section \ref{nonvanishing} and Lemma \ref{Appendix}, then we will have
\begin{equation*}
\|M_j^k\|_{L^p\rightarrow L^p}\lesssim \|M_{j,loc}^k\|_{L^p\rightarrow
L^p},
\end{equation*}
where $M_{j,loc}^kf(y):=\sup_{t\in[1,2]}|A_{t,j}^kf(y)|$.

For fixed $t\in[1,2]$, let us estimate $\widehat{d\mu_{k,m}}(\tilde{\delta}_t\xi')$.

Set
\begin{equation}
\delta:=2^{-k},\hspace{0.3cm}s:=s(\xi',t)=-\frac{\xi_2}{t\xi_3},\hspace{0.2cm}\textrm{for}\hspace{0.2cm}\xi_3\neq 0,
\end{equation}
\begin{equation}
\Phi(s,x_2,\delta):=-sx_2+x_2^2\phi(\delta x_2).
\end{equation}

A similar argument as in Section \ref{nonvanishing} and Section \ref{vanishm>1} shows that we can reduce to considering for $ j\geq 1$, the $L^p$-boundedness of the operator $\widetilde{A_{t,j}^k}$ given by
 \begin{align*}
\widetilde{A_{t,j}^k}f(y)&:=\int_{\mathbb{R}}\eta_1(\frac{x_1}{t^{a_1}})\int
_{{\mathbb{R}}^2}e^{i(\xi'\cdot y'-t^2\xi_3\tilde{\Phi}(s,\delta))}\chi_{k,m}(\frac{\xi_2}{t\xi_3})
\frac{A_{k,m}(\tilde{\delta}_t\xi')}{(1+|\tilde{\delta}_t\xi'|)^{1/2}}
\beta(2^{-j}|\tilde{\delta}_t\xi'|)\\
&\quad \times f(y_1-x_1,\widehat{\xi'})d\xi'dx_1,
\end{align*}
where  $\tilde{\Phi}(s,\delta):=\Phi(s,\tilde{q}(s,\delta),\delta)$ and $x_2=\tilde{q}(s,\delta)$ is the solution of the equation\\
$\partial_2\Phi(s,x_2,\delta)=0$ and smoothly converges to the solution $\tilde{q}(s,0)=s$ of the equation $\partial_2\Phi(s,x,0)=0$  if we assume $\phi(0)=1/2$. The phase function can be written as
\begin{equation}\label{weneed}
-t^2\xi_3\tilde{\Phi}(s,\delta):=\frac{\xi_2^2}{2\xi_3}+(-1)^{m+1}\delta^m\frac{\phi^{(m)}(0)}{m!}\frac{\xi_2^{m+2}}{t^m\xi_3^{m+1}}+R(t,\xi',\delta),
\end{equation}
which is homogeneous of degree one and can be considered as a small perturbation of  $\frac{\xi_2^2}{2\xi_3}+(-1)^{m+1}\delta^m\frac{\phi^{(m)(0)}}{m!}\frac{\xi_2^{m+2}}{t^m\xi_3^{m+1}}$. $\chi_{k,m}$ is a smooth function supported in the interval $[c_{k,m},\tilde{c}_{k,m}]$, for certain non-zero constants $c_{k,m}$ and $\tilde{c}_{k,m}$ dependent only on $k$ and $m$. $A_{k,m}$ is a symbol of order zero and $\{A_{k,m}(\tilde{\delta}_t\xi')\}_k$ is contained in a bounded subset of symbols of order zero. Denote by $\widetilde{\mathcal{M}_{j,loc}^k}$ the corresponding maximal operator.

Following the argument of (\ref{local}), we can
choose a bump function $\rho_1 \in C_0^{\infty}(\mathbb{R}^2\times [1/2,4])$, by  Lemma \ref{lem:Lemma3}, we have
\begin{equation}
\begin{aligned}
\|\widetilde{\mathcal{M}_{j,loc}^k}f\|^p_{L^p}&\leq C_p\biggl(\int
_{{\mathbb{R}}^3}\int_{1/2}^4|\rho_1(y',t)\widetilde{A_{t,j}^k}f(y)|^pdtdy\biggl)^{1/p'}\\
&\quad \times \biggl(\int
_{{\mathbb{R}}^3}\int_{1/2}^4|\frac{\partial}{\partial t}(\rho_1(y',t)\widetilde{A_{t,j}^k}f(y))|^pdtdy\biggl)^{1/p}.
\end{aligned}
\end{equation}

Moreover, we choose a non-negative function $\tilde{\eta}_1\in C_0^{\infty}(\mathbb{R})$ such that $\eta_1=1$ on the support of $\eta_1$, then
\begin{equation*}
\frac{\partial}{\partial t}\left(\widetilde{A_{t,j}^k}f(y)\right)=\int_{\mathbb{R}}\tilde{\eta}_1(\frac{x_1}{t})\int
_{{\mathbb{R}}^2}e^{i(\xi'\cdot y'-t^2\xi_3\tilde{\Phi}(s,\delta))}h_{k,m}(y,t,j,\xi')
f(y_1-x_1,\widehat{\xi'})d\xi'dx_1,
\end{equation*}
where
\begin{align*}
h_{k,m}(y,t,j,\xi')&=\left(\frac{\partial}{\partial t}\rho_1(y',t)\eta_1(\frac{x_1}{t})-\rho_1(y',t)\frac{x_1}{t^2}\eta_1'(\frac{x_1}{t})-\rho_1(y',t)\eta_1(\frac{x_1}{t})\frac{\partial}{\partial t}(t^2\xi_3\tilde{\Phi}(s,\delta))\right)\\
&\quad \times \chi_{k,m}(\frac{\xi_2}{t\xi_3})
\frac{A_{k,m}(\tilde{\delta}_t\xi')}{(1+|\tilde{\delta}_t\xi'|)^{1/2}}
\beta(2^{-j}|\tilde{\delta}_t\xi'|)\\
&\quad -\rho_1(y',t)\eta_1(\frac{x_1}{t})\frac{\xi_2}{t^2\xi_3}\chi_{k,m}'(\frac{\xi_2}{t\xi_3})\frac{A_{k,m}(\tilde{\delta}_t\xi')}{(1+|\tilde{\delta}_t\xi'|)^{1/2}}
\beta(2^{-j}|\tilde{\delta}_t\xi'|)\\
&\quad +\rho_1(y',t)\eta_1(\frac{x_1}{t})\chi_{k,m}(\frac{\xi_2}{t\xi_3})\frac{\partial}{\partial t}\left(\frac{A_{k,m}(\tilde{\delta}_t\xi')}{(1+|\tilde{\delta}_t\xi'|)^{1/2}}\right)
\beta(2^{-j}|\tilde{\delta}_t\xi'|)\\
&\quad +\rho_1(y',t)\eta_1(\frac{x_1}{t})\chi_{k,m}(\frac{\xi_2}{t\xi_3})\frac{A_{k,m}(\tilde{\delta}_t\xi')}{(1+|\tilde{\delta}_t\xi'|)^{1/2}}
\frac{\partial}{\partial t}\left(\beta(2^{-j}|\tilde{\delta}_t\xi'|)\right).
\end{align*}
Since  $t\approx 1$, $x_1$ and support of $\chi_{k,m}$ are sufficiently small, by (\ref{weneed}), then $|t^2\xi_3\tilde{\Phi}(s,\delta)|\approx 2^{j}\delta^m$. Moreover, $\{A_{k,m}(\tilde{\delta}_t\xi')\}_k$ is contained in a bounded subset of symbols of order zero,  so we obtain
\begin{equation}
|h_{k,m}(y,t,j,\xi')|\lesssim 2^{j/2}\delta^m+2^{-j/2}.
\end{equation}
Now it is easy to see that
$(2^{j/2}\delta^m+2^{-j/2})^{-1}\frac{\partial}{\partial t}(\rho_1(y',t)\widetilde{A_{t,j}^k}f)$ behaves like $2^{j/2}\widetilde{A_{t,j}^k}f$. It is sufficient to estimate $\|\rho_1(y',t)\widetilde{A_{t,j}^k}f(y)\|_{L^p(\mathbb{R}^3\times[1/2,4],dt dy)}$.

Furthermore, choosing a function $\eta_0\in C_0^{\infty}(\mathbb{R})$, non-negative, such that for arbitrary $t\in [1,2]$, $a_1>0$, $\eta_1(\frac{x_1}{t^{a_1}})\leq \eta_0(x_1)$,  we get
 \begin{align*}
&\left|\int_{\mathbb{R}}\eta_1(\frac{x_1}{t^{a_1}})\int
_{{\mathbb{R}}^2}e^{i(\xi'\cdot y'-t^2\xi_3\tilde{\Phi}(s,\delta))}\chi_{k,m}(\frac{\xi_2}{t\xi_3})
\frac{A_{k,m}(\tilde{\delta}_t\xi')}{(1+|\tilde{\delta}_t\xi'|)^{1/2}}
\beta(2^{-j}|\tilde{\delta}_t\xi'|)f(y_1-x_1,\widehat{\xi'})d\xi'dx_1\right|\\
&\leq \int_{\mathbb{R}}\eta_0(x_1)\biggl|\int
_{{\mathbb{R}}^2}e^{i(\xi'\cdot y'-t^2\xi_3\tilde{\Phi}(s,\delta))}\chi_{k,m}(\frac{\xi_2}{t\xi_3})
\frac{A_{k,m}(\tilde{\delta}_t\xi')}{(1+|\tilde{\delta}_t\xi'|)^{1/2}}
\beta(2^{-j}|\tilde{\delta}_t\xi'|)f(y_1-x_1,\widehat{\xi'})d\xi'\biggl|dx_1.
\end{align*}

In order to apply the regularity estimate Lemma \ref{mlemmaL^2} for $j\leq 9km$ and the local smoothing estimate Theorem \ref{F_lambda} for $j>9km$ of the Fourier integral operators not satisfying the "cinematic curvature condition" uniformly, we freeze $x_1$, in fact, by  Minkowski's and Young's inequalities, we have
\begin{align*}
&\|\rho_1(y',t)\widetilde{A_{t,j}^k}f(y)\|_{L^p(\mathbb{R}^3\times[1/2,4],dtdy)}
\\
&\leq C\Bigg(\int
_{{\mathbb{R}}^3}\int\biggl(\rho_1(y',t)\int_{\mathbb{R}}\eta_0(x_1)\biggl|\int
_{{\mathbb{R}}^2}e^{i(\xi'\cdot y'-t^2\xi_3\tilde{\Phi}(s,\delta))}\chi_{k,m}(\frac{\xi_2}{t\xi_3})
\frac{A_{k,m}(\tilde{\delta}_t\xi')}{(1+|\tilde{\delta}_t\xi'|)^{1/2}}
\beta(2^{-j}|\tilde{\delta}_t\xi'|)\\
&\quad\times f(y_1-x_1,\widehat{\xi'})d\xi'\biggl|dx_1\biggl)^pdtdy\Biggl)^{1/p}
\\
&\leq C \Biggl\|\int_{\mathbb{R}}\eta_0(x_1)\biggl\|\rho_1(y',t)\int
_{{\mathbb{R}}^2}e^{i(\xi'\cdot y'-t^2\xi_3\tilde{\Phi}(s,\delta))}\chi_{k,m}(\frac{\xi_2}{t\xi_3})
\frac{A_{k,m}(\tilde{\delta}_t\xi')}{(1+|\tilde{\delta}_t\xi'|)^{1/2}}
\beta(2^{-j}|\tilde{\delta}_t\xi'|)\\
&\quad\times f(y_1-x_1,\widehat{\xi'})d\xi'
\biggl\|_{L^p({\mathbb{R}}^2\times[\frac{1}{2},4],dtdy')}dx_1
\Biggl\|_{L^p(\mathbb{R},dy_1)}\\
& \leq C\|\eta_0\|_{L^1}\Biggl\|\biggl\|\rho_1(y',t)\int
_{{\mathbb{R}}^2}e^{i(\xi'\cdot y'-t^2\xi_3\tilde{\Phi}(s,\delta))}\chi_{k,m}(\frac{\xi_2}{t\xi_3})
\frac{A_{k,m}(\tilde{\delta}_t\xi')}{(1+|\tilde{\delta}_t\xi'|)^{1/2}}
\beta(2^{-j}|\tilde{\delta}_t\xi'|)\\
&\quad \times f(y_1,\widehat{\xi'})d\xi'
\biggl\|_{L^p({\mathbb{R}}^2\times[\frac{1}{2},4],dtdy')}\Biggl\|_{L^p(\mathbb{R},dy_1)}.
\end{align*}

Finally, together with the arguments from  Section \ref{vanishm>1}, we finish the proof.


\subsection{Proofs for surfaces of finite type }


\subsubsection{Maximal function theorem  with $da_2\neq a_3$} \label{da2neqa3}

Theorem \ref{3vanish2} will be proved in this section.

First we assume $\eta(x)=\eta_1(x_1)\eta_2(x_2)$ with non-negative functions $\eta_1$,
$\eta_2$ $\in C_0^{\infty}({\mathbb{R}})$ and $f\geq
0$, $a_1=1$. Let $(y_2, y_3)=y'$ and $(\xi_2,\xi_3)=\xi'$. Denote  $(t^{a_2}\xi_2, t^{a_3}\xi_3)$ by
$\delta'_t\xi'$. We choose $B>0$ very small and  $\tilde{\rho}\in C_0^{\infty}(\mathbb{R})$ such that supp $\tilde{\rho}\subset\{x\in \mathbb{R}:B/2\leq|x|\leq 2B\}$  and $\sum_k\tilde{\rho}(2^kx)=1$.

Put
\begin{equation*}
A_tf(y):=\int_{\mathbb{R}^2}f(y_1-tx_1,y_2-t^{a_2}x_2,y_3-t^{a_3}x_2^d\Phi(x))\eta(x)dx:=\sum_k\widetilde{A_t^k}f(y),
\end{equation*}
where $\widetilde{A_t^k}f(y):=\int_{\mathbb{R}^2}f(y_1-tx_1,y_2-t^{a_2}x_2,y_3-t^{a_3}x_2^d\Phi(x))\eta(x)\tilde{\rho}(2^kx_2)dx$.

Define the isometric operator on $L^p(\mathbb{R}^3)$  by
\begin{equation*}
Tf(x_1,x_2,x_3)=2^{k(d+1)/p}f(x_1,2^kx_2,2^{dk}x_3).
\end{equation*}

By the arguments in Section \ref{nonvanishing},  it suffices to prove  the following inequality
\begin{equation*}
\sum_k2^{-k}\|\sup_{t>0}|A_t^k|\|_{L^p\rightarrow L^p}\leq C_p,
\end{equation*}
where
\begin{equation}
A_t^kf(y)=\int_{\mathbb{R}^2}f(y_1-tx_1,y_2-t^{a_2}x_2,y_3-t^{a_3}x_2^d\Phi(x_1,\frac{x_2}{2^k}))\eta_1(x_1)\eta_2(2^{-k}x_2)\tilde{\rho}(x_2)dx.
\end{equation}

 By means of the Fourier inversion formula, we can write
\begin{equation*}
 A_t^kf(y)=\frac{1}{(2\pi)^3}\int_{{\mathbb{R}}^3}e^{i\xi\cdot
  y}\int_{\mathbb{R}}e^{-it\xi_1x_1}\eta_1(x_1)\widehat{d\mu_{k,x_1,d}}(\delta'_t\xi')dx_1\hat{f}(\xi)d\xi,
\end{equation*}
where
\begin{align*}
 \widehat{d\mu_{k,x_1,d}}(\delta'_t\xi')&=\int_{\mathbb{R}}e^{-i(t^{a_2}\xi_2x_2+t^{a_3}\xi_3x_2^d\Phi(x_1,\frac{x_2}{2^k}))}\tilde{\rho}(x_2)\eta_2(2^{-k}x_2)dx_2.
\end{align*}

Choose a non-negative function $\beta\in C_0^{\infty}(\mathbb{R})$
such that
\begin{equation*}
\hspace{0.2cm}\textrm{supp} \hspace{0.2cm} \beta \subset[1/2,2] \hspace{0.5cm}\textrm{and}
\hspace{0.5cm}\sum_{j\in \mathbb{Z}}\beta(2^{-j}r)=1 \hspace{0.5cm}
\textrm{for} \hspace{0.5cm} r>0.
\end{equation*}
Define
\begin{equation*}
A_{t,j}^kf(y):= \frac{1}{(2\pi)^3}\int_{{\mathbb{R}}^3}e^{i\xi\cdot
  y}\int_{\mathbb{R}}e^{-it\xi_1x_1}\eta_1(x_1)\widehat{d\mu_{k,x_1,d}}(\delta'_t\xi')dx_1\beta(2^{-j}|\delta'_t\xi'|)\hat{f}(\xi)d\xi,
\end{equation*}
and denote by $M_j^k$ the corresponding maximal operator.

From the arguments of Section \ref{nonvanishing}, we see that
$\sup_{t>0}\left|\sum_{j\leq 0}^{\infty}A_{t,j}^kf(y)\right|$
can be dominated by the Hardy-Littlewood maximal function $Mf(y)$ defined by (\ref{maximalfuxntion}), then it suffices to prove that
\begin{equation}
\|M_j^k\|_{L^p\rightarrow L^p}\leq
C_p2^{-j\epsilon(p)},\hspace{0.3cm} j\geq 1, \hspace{0.2cm}2<p<\infty,
\hspace{0.2cm} \textrm{some} \hspace{0.2cm}\epsilon(p)>0.
\end{equation}

Since $A_{t,j}^k$ is localized to frequencies $|\delta_t'\xi'|\approx 2^j$,  we can still use Lemma \ref{Appendix}
to prove that
\begin{equation*}
\|M_j^k\|_{L^p\rightarrow L^p}\lesssim \|M_{j,loc}^k\|_{L^p\rightarrow
L^p},
\end{equation*}
where $M_{j,loc}^kf(y):=\sup_{t\in[1,2]}|A_{t,j}^kf(y)|$.

For fixed $t\in[1,2]$, let us estimate $\widehat{d\mu_{k,x_1,d}}(\delta'_t\xi')$.

Set
\begin{equation}
\delta:=2^{-k}, \hspace{0.2cm} s:=s(\xi',t)=-\frac{t^{a_2}\xi_2}{t^{a_3}\xi_3}, \hspace{0.3cm}\textrm{for} \hspace{0.2cm} \xi_3\neq 0,
\end{equation}
and
\begin{equation}
\Psi(x_1,x_2,s,\delta):=-sx_2+\Phi_k(x_1,x_2),
\end{equation}
where  $\Phi_k(x_1,x_2):=x_2^d\Phi(x_1,\delta x_2)$.

Since $x_2\approx 1$ here, in addition to $\Phi(0,0)\neq 0$, then we can reduce our proof to the case $d=2$. So from now on we can proceed similarly as in Section \ref{nonvanishing}.

This finishes the proof of Theorem \ref{3vanish2}.


\subsubsection{Maximal theorem with $da_2=a_3$}\label{da2=a_3}

Combining the proofs of Section \ref{dim32a2=a3} and Section \ref{vanishfinite}, it is easy to get Theorem \ref{vanish3}. We omit the details here.


\subsection{Proofs for surfaces not passing through the origin}\label{notpass1}

$(i)$ We will show the proof of Theorem \ref{corollary}, where $da_2\neq a_3$.

Following the proof of Section \ref{da2neqa3}, we only modify some places.

Define
\begin{equation}\label{Atl}
\begin{aligned}
A_{t,j}^kf(y)&:= \frac{1}{(2\pi)^3}\int_{{\mathbb{R}}^3}e^{i\langle\xi,
  (y_1,y_2,y_3-t^{a_3}2^{dk})\rangle}\int_{\mathbb{R}}e^{-it\xi_1x_1}\eta_1(x_1)\widehat{d\mu_{k,x_1,d}}(\delta'_t\xi')dx_1\\
&\quad\times\beta(2^{-j}|\delta'_t\xi'|)\hat{f}(\xi)d\xi.
\end{aligned}
\end{equation}

Set  $A_{t}^{k,0}f(y):=\sum_{j\leq 0}^{\infty}A_{t,j}^kf(y)$ and $M^{k,0}f(y):=\sup_{t>0}|A_{t}^{k,0}f(y)|$. It is easy to see that
\begin{equation}
A_{t}^{k,0}f(y)=f*K^{\sigma}_{\delta_{t^{-1}}}(y),
\end{equation}
where
$\sigma=(0,0,2^{dk})$ and $K^{\sigma}_{\delta_{t^{-1}}}(y)=t^{-Q}K^{\sigma}(t^{-1}y_1,t^{-a_2}y_2,t^{-a_3}y_3)$ and $K^{\sigma}$ is the translate
\begin{equation}\label{K1}
K^{\sigma}(y)=K(y-\sigma).
\end{equation}
Indeed, since
\begin{align*}
{}&\frac{1}{(2\pi)^3}\int_{{\mathbb{R}}^3}e^{i\langle\xi,
  (y_1,y_2,y_3-t^{a_3}2^{dk})\rangle}\int_{\mathbb{R}}e^{-it\xi_1x_1}\eta_1(x_1)\widehat{d\mu_{k,x_1,d}}(\delta'_t\xi')dx_1\rho(|\delta'_t\xi'|)d\xi
\\
&=\frac{1}{(2\pi)^2}\int_{\mathbb{R}}\eta_1(x_1)\int_{{\mathbb{R}}^2}e^{i\langle\xi',
  (y_2,y_3-t^{a_3}2^{dk})\rangle}\widehat{d\mu_{k,x_1,d}}(\delta'_t\xi')\rho(|\delta'_t\xi'|)d\xi'\delta_0(y_1-tx_1)dx_1
\\
&=\frac{1}{(2\pi)^2}\frac{1}{t}\eta_1(\frac{y_1}{t})\int_{{\mathbb{R}}^2}e^{i\langle\xi',
  (y_2,y_3-t^{a_3}2^{dk})\rangle}\widehat{d\mu_{k,\frac{y_1}{t},d}}(\delta'_t\xi')\rho(|\delta'_t\xi'|)d\xi'
\\
&=\frac{1}{(2\pi)^2}t^{-Q}\eta_1(\frac{y_1}{t})\int_{{\mathbb{R}}^2}e^{i\langle\xi',\delta'_{t^{-1}}y'-(0,2^{dk})\rangle
  }\widehat{d\mu_{k,\frac{y_1}{t},d}}(\xi')\rho(|\xi'|)d\xi'
\\
&=K^{\sigma}_{\delta_{t^{-1}}}(y),
\end{align*}
then
\begin{equation*}
K(y):=\frac{1}{(2\pi)^2}\eta_1(y_1)\int_{{\mathbb{R}}^2}e^{i\xi'\cdot
  y'}\widehat{d\mu_{k,y_1,d}}(\xi')\rho(|\xi'|)d\xi',
\end{equation*}
and by integration by parts,
\begin{equation}\label{K2}
|K(y)|\leq C_N(1+|y|)^{-N}, N\in \mathbb{N}.
\end{equation}

Now, we choose $N$ large enough, (\ref{K1}) and (\ref{K2}) show that
\begin{equation}\label{l1}
\|M^{k,0}\|_{L^{\infty}\rightarrow L^{\infty}}\leq C,
\end{equation}
with a constant $C$ which does not depend on $\sigma$. Moreover, scaling by the factor $2^{-kd}$ in the direction of the vector $\sigma$, we see that
\begin{equation}\label{l2}
\|M^{k,0}\|_{L^1\rightarrow L^{1,\infty}}\leq C 2^{kd},
\end{equation}
since we can compare with $2^{kd}M$, where $M$ is the Hardy-Littlewood maximal operator defined by (\ref{maximalfuxntion}).

Finally,  by interpolation between (\ref{l1}) and (\ref{l2}), we obtain that
\begin{equation}
 \|M^{k,0}\|_{L^{p}\rightarrow L^{p}}\leq C 2^{kd/p}, \hspace{0.4cm}  p>1.
\end{equation}
Then $p>d$ implies that $\sum_k2^{-k}\|M^{k,0}\|_{L^p\rightarrow L^p}\leq C\sum_k2^{k(d/p-1)}\lesssim 1$.

Hence, it suffices to prove that
\begin{equation*}
\|M_j^k\|_{L^p\rightarrow L^p}\leq
C_p2^{-j\epsilon(p)}2^{kd/{p}},\hspace{0.2cm} j\geq 1, \hspace{0.2cm}d<p<\infty,
\hspace{0.2cm} \textrm{some }\hspace{0.2cm}\epsilon(p)>0,
\end{equation*}
where $M_j^kf(y):=\sup_{t>0}|A_{t,j}^kf(y)|$.

Since $A_{t,j}^k$ is localized to frequencies $|\delta_t'\xi'|\approx 2^j$, we still use Lemma \ref{Appendix}
to prove that
\begin{equation*}
\|M_j^k\|_{L^p\rightarrow L^p}\lesssim \|M_{j,loc}^k\|_{L^p\rightarrow
L^p},
\end{equation*}
where $M_{j,loc}^kf(y):=\sup_{t\in[1,2]}|A_{t,j}^kf(y)|$.

Furthermore, a standard application of the method of stationary
phase requires us to show the $L^p$-boundedness of the operator $\widetilde{\mathcal{M}_{j,loc}^{k,1}}$ given by
\begin{equation}\label{M}
\begin{aligned}
\widetilde{\mathcal{M}_{j,loc}^{k,1}}f(y)&:=\sup_{t\in[1,2]}|\widetilde{A_{t,j}^{k,1}}f(y)|\\
&=\sup_{t\in[1,2]}\biggl|\int_{\mathbb{R}}\eta_1(\frac{x_1}{t})\int
_{{\mathbb{R}}^2}e^{iQ_{k,x_1,d}(y',t,\xi')}E_{k,x_1/t,d}
(\delta'_t\xi')dx_1\\
&\quad\times\beta(2^{-j}|\delta'_t\xi'|)f(y_1-x_1,\widehat{\xi'})d\xi'dx_1\biggl|,
\end{aligned}
\end{equation}
where
\begin{equation}
Q_{k,x_1,d}(y',t,\xi'):=\langle\xi', (y_2,y_3-t^{a_3}2^{dk})\rangle-t^{a_3}\xi_3\tilde{\Psi}(\frac{x_1}{t},s,\delta),
\end{equation}
\begin{equation}
\tilde{\Psi}(\frac{x_1}{t},s,\delta):=\Psi(\frac{x_1}{t},\psi(\frac{x_1}{t},s,\delta),s,\delta),
\end{equation}
and
\begin{equation}
\Psi(x_1,x_2,s,\delta):=-sx_2+x_2^d\Phi(x_1,\delta x_2).
\end{equation}

By Lemma \ref{lem:Lemma3}, we have
\begin{equation*}
\|\widetilde{\mathcal{M}_{j,loc}^{k,1}}f\|^p_{L^p}\leq C_p\biggl(\int
_{{\mathbb{R}}^3}|\rho(t)\widetilde{A_{t,j}^{k,1}}f(y)|^p\biggl)^{1/p'}\biggl(\int
_{{\mathbb{R}}^3}|\frac{\partial}{\partial t}(\rho(t)\widetilde{A_{t,j}^{k,1}}f(y))|^p\biggl)^{1/p}.
\end{equation*}

Since
\begin{align*}
\left|\frac{\partial}{\partial t}Q_{k,x_1,d}(y',t,\xi'))\right|&=\left|-a_3t^{a_3-1}\xi_32^{kd}+\xi_3\frac{\partial}{\partial t}(t^{a_3}\tilde{\Psi}(\frac{x_1}{t},s,\delta))\right|\leq C2^{kd}|\xi_3|\approx 2^{dk}2^j,
\end{align*}
then
$\frac{\partial}{\partial t}(\rho(t)\widetilde{A_{t,j}^{k,1}})$ behaves like $2^{kd}2^{j}\widetilde{A_{t,j}^{k,1}}$. Clearly we  only consider the $L^p$-estimate for the operator $\widetilde{A_{t,j}^{k,1}}$.

The same estimates will follow from the proof of Section \ref{da2neqa3},
by the assumption $p>d$, finally we obtain
\begin{align*}
&\sum_k2^{-k}\sum_j\|\widetilde{\mathcal{M}_{j,loc}^{k,1}}\|_{L^p\rightarrow L^p}\\
&\leq C_p\sum_k2^{-k}\sum_j\biggl( 2^{-j(1/p+\epsilon(p))}\|\eta_0\|_{L^1}\biggl)^{(p-1)/p}\biggl( 2^{kd}2^{-j(-1+1/p+\epsilon(p))}\|\eta_0\|_{L^1}\biggl)^{1/p}
\\
&= C_p\sum_k2^{-k(1-d/p)}\sum_j2^{-jp\epsilon(p)}\|\eta_0\|_{L^1}\leq C_p.
\end{align*}

We have finished the proof of Theorem \ref{corollary}.

$(ii)$ Theorem \ref{vanishnotorigin} will follow from the argument of Section \ref{da2=a_3} and $(i)$, here we assume that $a_2=1$, $a_3=d$ and $1\leq m<\infty$. We just give main steps to clarify the proof.

Let $\phi(0)=1/d$, $s=-\frac{\xi_2}{t^{d-1}\xi_3}$ and $\tilde{\Phi}(s,\delta):=\Phi(s,\tilde{q}(s,\delta),\delta)$. Standard application of the method of stationary phase requires us to show for $j>0$, the $L^p$-boundedness of the operator $\widetilde{\mathcal{M}_{j,loc}^{k,1}}$ given by
\begin{equation}
\begin{aligned}
\widetilde{\mathcal{M}_{j,loc}^{k,1}}f(y)&:=\sup_{t\in[1,2]}|\widetilde{A_{t,j}^{k,1}}f(y)|\\
&=\sup_{t\in[1,2]}\biggl|\int_{\mathbb{R}}\eta_1(\frac{x_1}{t^{a_1}})\int
_{{\mathbb{R}}^2}e^{i\langle\xi',(y_2,y_3-t^d2^{kd})\rangle-t^d\xi_3\tilde{\Phi}(s,\delta)}E_{k,d,m}
(\delta'_t\xi')\\
&\quad\times\beta(2^{-j}|\delta'_t\xi'|)f(y_1-x_1,\widehat{\xi'})d\xi'dx_1\biggl|,
\end{aligned}
\end{equation}
where
\begin{align*}
-t^d\xi_3\tilde{\Phi}(s,\delta)&=-t^d\xi_3\Phi(s,\tilde{q}(s,\delta),\delta)\\
&=\biggl(\frac{1}{d^{\frac{d}{d-1}}}-\frac{1}{d^{\frac{1}{d-1}}}\biggl)\biggl(-\frac{d\xi_2^d}{\xi_3}\biggl)^{\frac{1}{d-1}}
-\frac{\delta^m\phi^{(m)}(0)}{t^mm!}
\biggl(-\frac{\xi_2}{\xi_3^{\frac{m+1}{m+d}}}\biggl)^{\frac{d+m}{d-1}}+ R(t,\xi',\delta,d),
\end{align*}
where $R(t,\xi,\delta,d)$ is homogeneous of degree one in $\xi$ and has at least $m+1$ power of $\delta$.


The similar argument with (\ref{local}) allows us to choose $\rho_1\in C_0^{\infty}(\mathbb{R}^2\times[1/2,4])$, by Lemma \ref{lem:Lemma3}, we have
\begin{align*}
\|\widetilde{\mathcal{M}_{j,loc}^{k,1}}f\|_{L^p}^p&\leq C_p\biggl(\int_\mathbb{R}^3\int_{1/2}^4\biggl|\rho_1(y',t)\widetilde{A_{t,j}^{k,1}}f(y)\biggl|^pdtdy\biggl)^{1/p'}\\
&\quad\times\biggl(\int_\mathbb{R}^3\int_{1/2}^4|\frac{\partial}{\partial t}(\rho_1(y',t)\widetilde{A_{t,j}^{k,1}}f(y))\biggl|^pdtdy\biggl)^{1/p'}.
\end{align*}

Since \begin{equation}\label{h_k_m}
\frac{\partial}{\partial t}(\rho_1(y',t)\widetilde{A_{t,j}^{k,1}}f(y))=\int_{{\mathbb{R}}^2}e^{i(\xi \cdot y-t^d\xi_2\tilde{\Phi}(s,\delta))}h_{k,m,d}(y,t,\xi,j)\hat{f}(\xi)d\xi,
\end{equation}
where $|h_{k,m,d}(y,t,\xi,j)|\lesssim  2^{j/2}2^{kd}$.

So we still have the same regularity estimates as (\ref{mlemmaL^2}), i.e.
\begin{equation*}
\|\widetilde{\mathcal{M}_{j,loc}^{k,1}}\|_{L^2\rightarrow L^2}\leq C2^{kd/2},
\end{equation*}
 and
\begin{equation*}
\|\widetilde{\mathcal{M}_{j,loc}^{k,1}}\|_{L^{\infty}\rightarrow L^{\infty}}\leq C.
\end{equation*}
Employing interpolation theorem, together the assumption $p>d$ and $k$ sufficiently big, then we have
\begin{align*}
\sum_k2^{-k}\sum_{0<j\leq 9km}\|\widetilde{\mathcal{M}_{j,loc}^{k,1}}\|_{L^p\rightarrow L^p}
&\leq C\sum_k2^{-k}\sum_{0<j\leq 9km}2^{kd/p}\\
&\leq C\sum_k2^{-k(1-d/p)}km \leq Cm,
\end{align*}
for $2< p<\infty$.

Based on these arguments, for the part of $j>9km$, combing the argument of Section 2.3.2 and the assumption $p>d$, we can get the result.

\section[Appendix]{Appendix.}

\begin{prop}
If $1\leq m<\infty$, then $p>2$ is a necessary condition for  the maximal inequality (\ref{equ:planem=1}).
\end{prop}

\begin{proof}

Let $f(x)=\frac{\chi_{[-1/2,1/2]^2}(x)}{|x|\log\frac{1}{|x|}}$. It's clear that $f\in L^p(\mathbb{R}^2)$ if $p\leq 2$.

 Now we take $\epsilon<\frac{1}{100}$ such that $[\epsilon/2,4\epsilon]\subset $ supp $\eta$, then fix it. We may assume that $\phi(0)>0$, $\phi^{(j)}(0)=0$, $j=1, 2,\cdots, m-1$, and $\phi^{(m)}(0)>0$. Let $y_1\in [1/8,1/4]$ and $y_1^2\phi(\epsilon)\leq y_2\leq y_1^2\phi(2\epsilon)$. By the inverse function theorem, we can set $a:=a(y)=\phi^{-1}(y_2/y_1^2)$,
 then $\epsilon\leq a\leq 2\epsilon$  (otherwise, if we assume $\phi(0)>0$ and $\phi^{(m)}(0)<0$, since $\phi'(x)=\frac{\phi^{(m)}(0)}{(m-1)!}x^{m-1}+\mathcal{O}(x^{m})$, then $\phi(x)$ is strictly decreasing in the interval  $[\epsilon/2,4\epsilon]$. In this case, we take $y_1^2\phi(2\epsilon)\leq y_2\leq y_1^2\phi(\epsilon)$). We choose $\epsilon_1>0$ such that $\epsilon_1\ll\epsilon$, then $\{x:|x-a|<\epsilon_1\}\subset$ supp $\eta$.  Now we have
\begin{align*}
Mf(y)&=\sup_{t>0}\left|\int f(y_1-tx,y_2-t^2x^2\phi(x))\eta(x)dx\right|\\
&=
\sup_{t>0}\int\frac{\chi_{[-1/2,1/2]^2}(y_1-tx,y_2-t^2x^2\phi(x))}{|(y_1-tx,y_2-t^2x^2\phi(x))|\log\frac{1}{|(y_1-tx,y_2-t^2x^2\phi(x))|}}\eta(x)dx\\
& \geq
\int_{\{x:|x-a|<\epsilon_1\}}\frac{\chi_{[-1/2,1/2]^2}(y_1-t_0x,y_2-t_0^2x^2\phi(x))}{|(y_1-t_0x,y_2-t_0^2x^2\phi(x))
|\log\frac{1}{|(y_1-t_0x,y_2-t_0^2x^2\phi(x))|}}dx,
\end{align*}
where $t_0=y_1/a$. We also observe that $\frac{1}{16\epsilon}\leq t_0\leq \frac{1}{4\epsilon}$.

If $|x-a|<\epsilon_1$, then $x$ can be written as $x=a+z=\frac{y_1}{t_0}+z$, $|z|<\epsilon_1$, which gives
$|y_1-t_0x|=t_0|x-\frac{y_1}{t_0}|=t_0|z|\leq \frac{\epsilon_1}{4\epsilon}<1/4$, and $|y_2-t_0^2x^2\phi(x))|$ is equal to
\begin{align*}
&|y_2-\frac{y_1^2}{a^2}(a+z)^2(\phi(a)+\phi'(a)z+\mathcal{O}(z^2))|\\
& =|y_2-\frac{y_1^2}{a^2}a^2\phi(a)-\frac{y_1^2}{a^2}(2az+z^2)\phi(a+z)-\frac{y_1^2}{a^2}a^2
(\phi'(a)z+\mathcal{O}(z^2))|\\
&=|\frac{y_1^2}{a^2}(2az+z^2)\phi(a+z)
+y_1^2
(\phi'(a)z+\mathcal{O}(z^2))|\\
&=t_0|z||y_1(2+\frac{z}{a})\phi(a+z)+ay_1(\phi'(a)+\mathcal{O}(z))|
\end{align*}
\begin{align*}
&\leq \frac{\epsilon_1}{4\epsilon}[\frac{1}{4}(2+\frac{\epsilon_1}{\epsilon})\phi(4\epsilon)+\frac{\epsilon}{2}(\phi(a)+\mathcal{O}(\epsilon_1))]t_0|z|<1/4.
\end{align*}

Set $g(z)=y_1(2+\frac{z}{a})\phi(a+z)+ay_1(\phi'(a)+\mathcal{O}(z))$ which is bounded from above by, say $\tilde{C}$. From the above argument, we have $(y_1-t_0x,y_2-t_0^2x^2\phi(x))\in [-1/2,1/2]$. Then making a change of variables, we have
\begin{align*}
Mf(y)&\geq
\int_{\{x:|x-a|<\epsilon_1\}}\frac{dx}{|(y_1-t_0x,y_2-t_0^2x^2\phi(x))
|\log\frac{1}{|(y_1-t_0x,y_2-t_0^2x^2\phi(x))|}}\\
&\geq \int_0^{\epsilon_1}\frac{dz}{t_0z|(1,g(z))|
\log\frac{1}{t_0z|(1,g(z))|}}.
\end{align*}

Let $C=\sqrt{1+\tilde{C}^2}$. It is a fact that $\frac{1}{x\log\frac{1}{x}}$ decrease monotonically in the neighborhood of the origin.  Together with the above estimates, we will have
\begin{align*}
\frac{1}{t_0z|(1,g(z))|
\log\frac{1}{t_0z|(1,g(z))|}} \geq \frac{1}{Ct_0z
\log\frac{1}{Ct_0z}}.
\end{align*}

Finally, for the fixed $\epsilon$ and $\epsilon_1 \ll \epsilon$, we have
\begin{align*}
Mf(y)\geq \int_0^{\epsilon_1}\frac{dz}{Ct_0z
\log\frac{1}{Ct_0z}}\geq \frac{4\epsilon}{C}\int_0^{C\epsilon_1/(16\epsilon)}\frac{dz}{z\log\frac{1}{z}}=\infty.
\end{align*}
\end{proof}

\begin{flushleft}
\vspace{0.3cm}\textsc{Wenjuan Li\\Mathematisches Seminar\\ Christian-Albrechts-Universit\"{a}t zu Kiel\\ Ludewig-Meyn-Stra{\ss}e 4,\\D-24098\\Kiel, Germany\\Current Address\\School of Natural and Applied Sciences\\Northwest Polytechnical University\\710129\\Xi'an, People's Republic of China}

\emph{E-mail address}: \textsf{liwj@nwpu.edu.cn}

\end{flushleft}

\end{document}